\documentclass[12pt]{amsart}%
\usepackage{amsfonts}
\usepackage{amsmath}
\usepackage{amssymb}
\usepackage{amscd}
\usepackage{graphicx}

\newtheorem{theorem}{Theorem}
\newtheorem{lemma}{Lemma}
\newtheorem{proposition}{Proposition}
\newtheorem{remark}{Remark}

\newtheorem{definition}{Definition}
\newtheorem{corollary}{Corollary}
\newtheorem{example}{Example}

\numberwithin{equation}{section}
\begin{document}
\title[Foliation Divisorial Contraction by the Sasaki-Ricci Flow]{Foliation Divisorial Contraction by the Sasaki-Ricci Flow on Sasakian Five-Manifolds}
\author{$^{\ast}$Shu-Cheng Chang}
\address{Department of Mathematics, National Taiwan University, Taipei, Taiwan}
\email{scchang@math.ntu.edu.tw}
\author{$^{\ast\ast}$Chien Lin}
\address{Mathematical Science Research Center, Chongqing University of Technology,
400054, Chongqing, P.R. China}
\email{chienlin@cqut.edu.cn}
\author{$^{\ast\ast\ast}$Chin-Tung Wu}
\address{Department of Applied Mathematics, National Pingtung University, Pingtung
90003, Taiwan}
\email{ctwu@mail.nptu.edu.tw }
\thanks{$^{\ast}$Shu-Cheng Chang and $^{\ast\ast\ast}$Chin-Tung Wu are partially
supported in part by the MOST of Taiwan.}

\begin{abstract}
Let $(M,\eta,\xi,\Phi,g)$ be a compact quasi-regular Sasakian $5$-manifold
with finite cyclic quotient foliation singularities of type $\frac{1}%
{r}(1,a).$ First, we derive the foliation minimal model program by applying
the resolution of cyclic quotient foliation singularities. Secondly, based on
the study of local model of resolution of foliation singularities, we prove
the foliation canonical surgical contraction or the foliation extremal ray
contraction under the Sasaki-Ricci flow. As a consequence, we prove%
 a Sasaki analogue of%
analytic
minimal model program with the Kaehler-Ricci flow due to Song-Tian
and Song-Weinkove%
.

\end{abstract}
\maketitle
\tableofcontents

\section{Introduction}

Sasakian geometry is very rich as the odd-dimensional analogous of K\"{a}hler
geometry. A Sasaki-Einstein $(2n+1)$-manifold is to say that its Kaehler cone
is a Calabi-Yau $(n+1)$-fold. Let $(M,\eta,\xi,\Phi,g)$ be a compact Sasakian
$(2n+1)$-manifold. If the orbits of the Reeb vector field $\xi$ are all
closed, and hence circles, then integrates to an isometric $U(1)$ action on
$(M,g)$. Since it is nowhere zero this action is locally free; that is, the
isotropy group of every point in $M$ is finite. If the $U(1$) action is in
fact free then the Sasakian structure is said to be regular. Otherwise, it is
said to be quasi-regular. Thus the space of leaves of the canonical
$U(1)$-fibration will have orbifold singularities (cf section $3$). In
general, it is said to be irregular if the orbits of are not all closed.

In particular, Sasaki-Einstein $5$-manifolds provide interesting examples of
the AdS/CFT correspondence. On the other hand, The class of simply connected,
closed, oriented, smooth, $5$-manifolds is classifiable under diffeomorphism
due to Smale-Barden (\cite{s}, \cite{b}). Then, in this paper, it is our goal
to focus on a classification of compact quasi-regular Sasakian $5$-manifolds
according to the global properties of the Reeb\textbf{\ }$U(1)$-fibration.

Let $(M,\eta,\xi,\Phi,g)$ be a compact quasi-regular Sasakian $5$-manifold.
Then by the first structure theorem, $M$ \ is a principal $S^{1}$-orbibundle
($V$-bundle) over an orbifold $\mathbf{Z}=(Z\emph{,}\Delta)$ which is also a
$Q$-factorial, polarized, normal projective orbifold surface such that there
is an orbifold Riemannian submersion$\ $%
\[
\pi:(M,g,\omega)\rightarrow(Z,h,\omega_{h})
\]
and
\begin{equation}
K_{M}^{T}=\pi^{\ast}(K_{Z}^{orb})=\pi^{\ast}(\varphi^{\ast}(K_{Z}%
+[\Delta])).\label{41}%
\end{equation}
If the orbifold structure of the leave space $Z$ is well-formed (c.f. section
$3$), then the orbifold canonical divisor $K_{\emph{Z}}^{orb}$ and canonical
divisor $K_{Z}$ are the same and thus
\[
K_{M}^{T}=\pi^{\ast}(\varphi^{\ast}(K_{Z})).
\]
In a such case, there is the Sasaki analogue of Mori's minimal model program
with respect to $K_{Z}$ in a such compact quasi-regular Sasakian $5$-manifold.
More precisely, one can ask the following so-called foliation minimal model
program :

Is there a foliation $(-1)$-curve ? One of Mori program (\cite{kmm}, \cite{m},
\cite{km}) is to replace this criterion with the one dictated by the
transverse canonical divisor $K_{M}^{T}$ of $M$ : Does the transverse
canonical divisor have nonnegative intersection
\[
K_{M}^{T}\cdot V\geq0
\]
with any invariant submanifold \ $V$ on $M$ which is a Sasakian $3$%
-submanifold. In other words,
\[
\mathrm{is}\ \ \ K_{M}^{T}\ \ \ \text{\ }\mathrm{nef\ }?
\]
\ If $K_{M}^{T}$ is not nef, there is an extremal transverse contraction map
which turns out not only to generalize the Sasaki analogue of Castelnuovo's
contractibility criterion but also to provide decisive information on the
global structures of the end results of the foliation minimal model program.
Then an end result of foliation MMP starting from $M$ is a transverse Mori
fiber space if and only if there exists a nonempty open set $U\subset M$ such
that for any $S^{1}$-fibre in $U$, there is an irreducible invariant
$3$-submanifold $V$ passing through such a fibre $S^{1}$ with $\ K_{M}%
^{T}\cdot V<0.$ We refre to Proposition \ref{P41}. Moreover, the Mori's
minimal model program in birational geometry can be viewed as the complex
analogue of Thurston's geometrization conjecture which was proved via
Hamilton's Ricci flow with surgeries on $3$-dimensional Riemannian manifolds
by Perelman (\cite{p1}, \cite{p2}, \cite{p3}). Likewise, there is a conjecture
picture by Song-Tian (\cite{st}) that the Kaehler-Ricci flow should carry out
an analytic minimal model program with scaling on projective varieties.
Recently, Song and Weinkove (\cite{sw1}) established the above conjecture on a
projective algebraic surface.

The Sasaki--Ricci flow is introduced by Smoczyk--Wang--Zhang (\cite{swz}) to
study the existence of Sasaki $\eta$-Einstein metrics on Sasakian manifolds.
They showed that the flow has the longtime solution and asymptotic converges
to a Sasaki $\eta$-Einstein metric when the basic first Chern class is
negative or null. It can be viewed as a Sasaki analogue of Cao's result
(\cite{cao}) for the Kaehler--Ricci flow.

In view of the previous discusses, it is natural to conjecture that the
Sasaki-Ricci flow will carry out an analytic foliation minimal model program
with scaling on quasi-regular Sasakian $5$-manifolds as well. In this paper,
we deal with the case of its orbifold structure $(Z,\Delta)$ of the leave
space $Z$ is well-formed which has the codimension two fixed point set of
every non-trivial isotropy subgroup with no branch divisors ($\Delta
=\emptyset$).

In section $3,$ we first derive the following result concering its foliation
cyclic quotient singularities\textbf{\ }on\textbf{\ }a compact quasi-regular
Sasakian $5$-manifold \textbf{:}

\begin{theorem}
\label{T34} Let $(M,\eta,\xi,\Phi,g)$ be a compact quasi-regular Sasakian
$5$-manifold and $Z$ be its leave space of the characteristic
foliation\textbf{. } Then $Z$ is a $Q$-factorial normal projective algebraic
orbifold surface satisfying

\begin{enumerate}
\item if its leave space $(Z,\emptyset)$ has at least codimension two fixed
point set of every non-trivial isotropy subgroup. That is to say $Z$ is
\textbf{well-formed}, then $Z$ has isolated singularies of a finite cyclic
quotient of $\mathbf{C}^{2}$ and the action is
\[
\mu_{Z_{r}}:(z_{1},z_{2})\rightarrow(\zeta^{a}z_{1},\zeta^{b}z_{2}),
\]
where $\zeta$ is a primitive $r$-th root of unity. We denote the cyclic
quotient singularity by $\frac{1}{r}(a,b)$ with $(a,r)=1=(b,r)$. In
particular, the action can be rescaled so that every cyclic quotient
singularity corresponds to a $\frac{1}{r}(1,a)$-type singularity with
$(r,a)=1,\zeta=e^{\frac{2\pi i}{r}}$. In particular, i\textbf{t is klt
(Kawamata log terminal) singularities}. Moreover, the corresponding
singularities in $(M,\eta,\xi,\Phi,g)$\textbf{\ }is called\textbf{\ foliation
cyclic quotient singularities of type }%
\[
\frac{1}{r}(1,a)
\]
at \textbf{a singular fibre} $\mathbf{S}_{p}^{1}$ in $M$.

\item if its leave space $(Z,\Delta)$ has the codimension one fixed point set
of some non-trivial isotropy subgroup. Then the action is
\[
\mu_{Z_{r}}:(z_{1},z_{2})\rightarrow(e^{\frac{2\pi a_{1}i}{r_{1}}}%
z_{1},e^{\frac{2\pi a_{2}i}{r_{2}}}z_{2}),
\]
for some positive integers $r_{1,}r_{2}$ whose least common multipler is $r$,
and $a_{i},i=1,2$ are integers coprime to $r_{i},i=1,2$. Then the foliation
singular set contains some $3$-dimensional submanifolds of $M.$ More
precisely, the corresponding singularities in $(M,\eta,\xi,\Phi,g)$%
\textbf{\ }is called\textbf{\ the Hopf }$S^{1}$\textbf{-orbibundle over a
Riemann surface }$\Sigma_{h}.$
\end{enumerate}
\end{theorem}

In particular, it is a foliation $(-1)$-curve (Theorem \ref{D31}) at the
regular point if $r=a=1$ as in $(1)$ of Theorem \ref{T34}. Then we have the
following definition (\cite{cu}):

\begin{definition}
\label{D11} Let $(M,\eta,\xi,\Phi,g)$ be a compact quasi-regular Sasakian
$5$-manifold with \textbf{foliation singularities of type }$\frac{1}{r}(1,a).
$ A foliation $(-1)$-curve $V$ in $M$ is said to be \textbf{floating} if $V$
is entirely contained in the smooth locus of $M$ with respective to the
foliation $\mathcal{F}_{\xi}$. Then $M$ is said to be \textbf{minimal} if it
has no floating foliation $(-1)$-curves.
\end{definition}

Secondly, a detail work on the local model of resolution of foliation cyclic
quotient singularities and a Castelnuovo's contraction theorem, we have the
following Sasaki analogue of Castelnuovo's contraction Theorem on Sasakian $5
$-Manifolds as in section $3:$

\begin{theorem}
\label{T33}(Sasaki Castelnuovo's Contraction Theorem) Let $(M,\eta,\xi
,\Phi,g)$ be a compact quasi-regular Sasakian $5$-manifold with
\textbf{foliation singularities of type }$\frac{1}{r}(1,a)$ and $V$ a
\textbf{floating} foliation $(-1)$-curve. Then there exists a transverse
birational morphism $f:M\rightarrow N$ that contracts $V\in S^{3}\times
D^{2}\subset M$ to a regular fibre $S_{\mathbf{o}}^{1}\in S^{1}\times
D^{4}\subset N$ and it is an isomorphism outside of $V$. $N$ is again a
compact quasi-regular Sasakian $5$-manifold. In particular, every proper
transverse birational morphism between compact regular Sasakian $5$-manifolds
can be factored in a sequence of contractions of foliation $(-1)$-curves.
\end{theorem}

In section $4,$ we work on some basic facts of the Sasaki analogue of the
Kaehler-Ricci flow through singularities due to Song-Tian (\cite{st}). In
particular, we have the following definition for the foliation canonical
surgical contraction on a compact quasi-regular Sasakian $5$-manifold.

\begin{definition}
\label{D12}We say that the solution $g(t)$ of the Sasaki-Ricci flow on a
compact Sasakian $5$-manifold $M$ performs a floating foliation canonical
surgical contraction if the following holds: There exist distinct
\textbf{floating} foliation $(-1)$-curve\textbf{s} $V_{1},...,V_{k}$ of $M$, a
compact Sasakian $5$-manifold $N$ and a divisorial contraction $\psi
:M\rightarrow N$ with $\psi(V_{i})=S_{i}^{1}\subset N$ \ and \ $\psi
:M\backslash\cup_{i=1}^{k}V_{i}\rightarrow N\backslash\{S_{1}^{1}%
,...,S_{k}^{1}\}$ a basic transverse biholomorphism onto \ $N\backslash
\{S_{1}^{1},...,S_{k}^{1}\}$\ such that :

\begin{enumerate}
\item The metrics $g(t)$ converge to a smooth Sasakain metric $g_{T}$ on
$M\backslash\cup_{i=1}^{k}V_{i},$ as $t\rightarrow T^{-}$, smoothly on compact
subsets of $M\backslash\cup_{i=1}^{k}V_{i}.$

\item $(M,g(t))$ converges to a unique compact metric space $(N,d_{T})$ in the
Gromov-Hausdorff sense as $t\rightarrow T^{-}$. In particular, $(N,d_{T}) $ is
homeomorphic to the Sasakain $5$-manifold $N$. Here $d_{T}$ is defined to be
the metric on $N$ by extending $(\psi^{-1})^{\ast}g_{T}$ to be zero on
$\{S_{1}^{1},...,S_{k}^{1}\}$.

\item There exists a unique maximal smooth solution $g(t)$ of the Sasaki-Ricci
flow on $N$ for $t\in(T,T_{N})$ with $T<T_{N}\leq\infty$, such that $g(t)$
converges to $(\psi^{-1})^{\ast}g_{T}$ as $t\rightarrow T^{+}$ smoothly on
compact subsets of $N\backslash\{S_{1}^{1},...,S_{k}^{1}\}.$

\item $(N,g(t))$ converges to $(N,d_{T})$ in the Gromov-Hausdorff sense as
$t\rightarrow T^{+}$.
\end{enumerate}
\end{definition}

In the final section, we will apply the results of local model of resolution
of foliation singularities as in section $3$ to prove the foliation canonical
surgical contraction on a compact \textbf{quasi-regular} Sasakian $5$-manifold
$M.$

Along the lines of the arguments in \cite{sw2}, we first deal with the case of
floating foliation canonical surgical contraction :

\begin{theorem}
\label{T61} Let $g(t)$ be a smooth solution of the Sasaki-Ricci flow on a
compact \textbf{quasi-regular} Sasakian $5$-manifold $M$ with the foliation
singularitie of type $\frac{1}{r}(1,a)$ for $t\in\lbrack0,T)$ and assume
$T<\infty$. Suppose there exists a blow-down map $\psi:M\rightarrow N$
contracting disjoint \textbf{floating} foliation $(-1)$-curves $V_{1}%
,...,V_{k}$ on $M$ with $\psi(V_{i})=S_{i}^{1}\subset N$, for a smooth compact
Sasakian $5$-manifold $(N,\omega_{N})$ such that the limiting transverse
Kaehler class satisfies%
\begin{equation}
\lbrack\omega_{0}]_{B}-Tc_{1}^{B}(M)=[\psi^{\ast}\omega_{N}]_{B}.\label{2022}%
\end{equation}
Then the Sasaki-Ricci flow $g(t)$ performs a foliation canonical surgical
contraction with respect to the data $V_{1},...,V_{k}$, $N$ and $\psi$.
\end{theorem}

By applying Theorem \ref{T61} to the minimal resolution of $\frac{1}{r}%
(1,a)$-type foliation singularities of $M$ as in Theorem \ref{T32}, we have

\begin{corollary}
\label{C61}Let $g(t)$ be a smooth solution of the Sasaki-Ricci flow on a
compact \textbf{regular} Sasakian $5$-manifold $M$ for $t\in\lbrack0,T)$ and
assume $T<\infty$. Suppose there exists a divisorial contraction
$\psi:M\rightarrow N$ contracting disjoint foliation $(-1)$-curves
$V_{1},...,V_{k}$ on $M$ with $\psi(V_{i})=S_{i}^{1}\subset N$, for a smooth
compact Sasakian $5$-manifold $(N,\omega_{N})$ such that the limiting
transverse Kaehler class satisfies%
\[
\lbrack\omega_{0}]_{B}-Tc_{1}^{B}(M)=[\psi^{\ast}\omega_{N}]_{B}.
\]
Then the Sasaki-Ricci flow $g(t)$ performs a foliation canonical surgical
contraction with respect to the data $V_{1},...,V_{k}$, $N$ and $\psi$.
\end{corollary}

As a consequence of Theorem \ref{T61} and Proposition \ref{P41}, we have our
main results in the paper on the analytic f
oliation minimal model program with scaling in a compact
quasi-regular Sasakian %
$%
5
$%
-Mmanifold.

\begin{theorem}
\label{T62} Let $(M,\xi,\omega_{0})$ be a compact \textbf{quasi-regular}
Sasakian $5$-manifold with the foliation singularitie of type $\frac{1}%
{r}(1,a).$ Then there exists a unique maximal Sasaki-Ricci flow $\omega(t)$ on
$M_{0},M_{1},...,M_{k}$ starting at $(M,\omega_{0})$ with floating foliation
canonical surgical contractions of a finite number of disjoint
\textbf{floating} foliation $(-1)$-curves \textbf{or} foliation extremal ray
contractions of foliation $K_{M}^{T}$-negative curves $\psi_{i}:M_{i-1}%
\rightarrow M_{i}$. In addition, we have

\begin{enumerate}
\item Either $T_{k}<\infty$ and the flow $\omega(t)$ collapses in the sense
that
\[
Vol_{\xi}(M_{k},\omega(t))\rightarrow0
\]
$\mathrm{as}\ \ \ t\rightarrow T_{k}^{-}.$

\begin{enumerate}
\item there exists a contraction%
\[
\varphi:M_{k}\rightarrow pt,
\]
then $K_{M}^{T}<0$ and thus $M_{k}$ is transverse minimal Fano and the
foliation space $M_{k}/\mathcal{F}_{\xi}$ is minimal log del Pezzo surface of
at worst $\frac{1}{r}(1,a)$-type singularities and Picard number one, or

\item there exists a fibration%
\[
\varphi:M_{k}\rightarrow\Sigma_{h},
\]
then $M_{k}$ is an $S^{1}$-orbibundle of a rule surface over \ Riemann
surfaces $\Sigma_{h}$\ of genus $h$. Or
\end{enumerate}

\item $T_{k}=\infty$ and $M_{k}$ is nef : $\psi_{i}=\psi_{k}$, and $M_{k}$ has
at worst foliation cyclic quotient singularities (orbifold singularities) and
has no foliation $K_{M}^{T}$-negative curves.
\end{enumerate}
\end{theorem}

\begin{corollary}
\label{C11} Let $(M,\xi,\omega_{0})$ be a compact \textbf{regular} Sasakian
$5$-manifold with a smooth transverse Kaehler metric $\omega_{0}$. Then there
exists a unique maximal Sasaki-Ricci flow $\omega(t)$ on $M_{0},M_{1}%
,...,M_{k}$ with foliation canonical surgical contractions starting at
$(M,\omega_{0})$. Moreover, each foliation canonical surgical contraction
corresponds to a divisorial contraction $\psi_{i}:M_{i-1}\rightarrow M_{i}$ of
a finite number of disjoint foliation $(-1) $-curves on $M_{i}$. In addition,
we have

\begin{enumerate}
\item Either $T_{k}<\infty$ and the flow $\omega(t)$ collapses in the sense
that
\[
Vol_{\xi}(M_{k},\omega(t))\rightarrow0,
\]
$\mathrm{as}\ \ \ t\rightarrow T_{k}^{-}.$Then $M_{k}$ is transverse
birational to a regular Sasakian sphere $\mathbf{S}^{5}$ or $(\mathbf{S}%
^{2}\mathbf{\times S}^{3})$ or $\mathbf{S}^{2}\widetilde{\times}\mathbf{S}%
^{3}$or $\Sigma_{h}\mathbf{\times S}^{3}$ or $\Sigma_{h}\widetilde{\times
}\mathbf{S}^{3}.$

\item Or $T_{k}=\infty$ and $M_{k}$ has no foliation $(-1)$-curves.
\end{enumerate}
\end{corollary}

In this paper, we assume that $M$ \ is a compact quasi-regular transverse
Sasakian manifold and the space $Z$ of leaves is well-formed which means its
orbifold singular locus and algebro-geometric singular locus coincide,
equivalently $Z$ has no branch divisors. However when its leave space $Z$ is
not well-formed, the orbifold structure $(Z,\Delta)$ of the leave space $Z$
has the codimension one fixed point set of some non-trivial isotropy subgroup.
Then\ the corresponding singularities in $(M,\eta,\xi,\Phi,g)$\ is the Hopf
$S^{1}$-orbibundle over a Riemann surface $\Sigma_{h}.$ The orbifold canonical
divisor $K_{\emph{Z}}^{orb}$ and canonical divisor $K_{Z}$ are related by
\[
K_{\emph{Z}}^{orb}=\varphi^{\ast}(K_{Z}+[\Delta])
\]
and then
\[
K_{M}^{T}=\pi^{\ast}(K_{\emph{Z}}^{orb}).
\]
There is a Sasaki analogue of Log minimal model program (\cite{km}) with
respect to $K_{Z}+[\Delta].$ Thus we conjecture that there is a Sasaki
analogue of analytic Log minimal model program with respect to $K_{Z}%
+[\Delta]$ via the conical Sasaki-Ricci flow which is the odd dimensional
counterpart of the conical K\"{a}hler-Ricci flow (\cite{lz}, \cite{shen}). We
hope to address this issue in the near future.

\textbf{Acknowledgement. }The authors would like to thank Yi-Sheng Wang for
sharing his ideas on the construction of foliation singularities as in the
section three.

\section{Sasakian Geometry}

In this section, we will give some preliminaries on structures theorems for
Sasakian structures and the orbifold structure of its leave space, the
foliation normal local coordinates, the basic cohomology and its
Type II deformation
. We refer to \cite{bg}, \cite{fow}, and references therein for some details.

\subsection{Sasakian Structure and Orbifold Structure of Its Leave Space}

Let $(M,g,\nabla)$ be a Riemannian $(2n+1)$-manifold. $(M,g)$ is called Sasaki
if \ the cone
\[
(C(M),\overline{g}):=(\mathbf{R}^{+}\times M\mathbf{,\ dr^{2}+r}^{2}g)
\]
such that $(C(M),\overline{g},J,\overline{\omega})$ is Kaehler with
\[
\overline{\omega}=\frac{1}{2}i\partial\overline{\partial}r^{2}.
\]
The function $\frac{1}{2}r^{2}$ is hence a global Kaehler potential for the
cone metric. As $\{r=1\}=\{1\}\times M\subset C(M)$, we may define%

\[
\overline{\xi}=J(r\frac{\partial}{\partial r})
\]
and the Reeb vector field $\xi$ on $M$
\[
\xi=J(\frac{\partial}{\partial r}).
\]
Also
\[
\overline{\eta}(\cdot)=\frac{1}{2}\overline{g}(\xi,\cdot)
\]
and the contact $1$-form $\eta$ on $TM$
\[
\eta(\cdot)=g(\xi,\cdot)
\]
Then $\xi$ is the killing vector field with unit length such that
\[
\eta(\xi)=1\ \text{\textrm{and}}\ \ d\eta(\xi,X)=0.
\]

The tensor field of $type(1,1)$, defined by $\Phi(Y)=\nabla_{Y}\xi$, satisfies
the condition%
\[
(\nabla_{X}\Phi)(Y)=g(\xi,Y)X-g(X,Y)\xi
\]
for any pair of vector fields $X$ and $Y$ on $M$. Then such a triple
$(\eta,\xi,\Phi)$ is called a Sasakian structure on a Sasakian manifold
$(M,g).$ Note that the Riemannian curvature satisfying the following
\[
R(X,\xi)Y=g(\xi,Y)X-g(X,Y)\xi
\]
for any pair of vector fields $X$ and $Y$ on $M$. In particular, the sectional
curvature of every section containing $\xi$ equals one.%

The first structure theorem on Sasakian manifolds states that

\begin{proposition}
\label{P21}(\cite{ru}, \cite{sp}) Let $(M,\eta,\xi,\Phi,g)$ be a compact
quasi-regular Sasakian manifold of dimension $2n+1$ and $Z$ denote the space
of leaves of the characteristic foliation $\mathcal{F}_{\xi}$ (just as
topological space). Then

(i) $Z$ carries the structure of a Hodge orbifold $\mathcal{Z=}(Z,\Delta)$
with an orbifold Kaehler metric $h$ and Kaehler form $\omega$ which defines an
integral class $[p^{\ast}\omega]$ \ in $H_{orb}^{2}(Z,\mathbf{Z)}$ in such a
way that $\pi:$ $(M,g)\rightarrow(Z,h)$ is an orbifold Riemannian submersion,
and a principal $S^{1}$-orbibundle ($V$-bundle) over $Z.$ Furthermore,it
satisfies $\frac{1}{2}d\eta=\pi^{\ast}(\omega).$The fibers of $\pi$ are geodesics.

(ii) $Z$ is also a $Q$-factorial, polarized, \textbf{normal projective}
algebraic variety.

(iii) $(M,\xi,g)$ is Sasaki-Einstein if and only if $(Z,h)$ is
Kaehler-Einstein with scalar curvature $4n(n+1).$

(iv) If $(M,\eta,\xi,\Phi,g)$ is regular then the orbifold structure is
trivial and $\pi$ is a principal circle bundle over a smooth projective
algebraic variety.

(v) As real cohomology classes, there is a relation between the basic Chern
class and orbifold chern class
\[
c_{k}^{B}(M):=c_{k}(\emph{F}_{\xi})=\pi^{\ast}c_{k}^{orb}(\mathbf{Z}).
\]

Conversely, let $\pi$: $M\rightarrow Z$ be a $U(1)$-orbibundle over a compact
Hodge orbifold $(Z,h)$ whose first Chern class is an integral class defined by
$[\omega_{Z}]$, and $\eta$ be a $1$-form with $\frac{1}{2}d\eta=\pi^{\ast
}\omega_{Z}$. Then $(M,\pi^{\ast}h+\eta\otimes\eta)$ is a Sasakian manifold if
all the local uniformizing groups inject into the structure group $U(1).$
\end{proposition}

On the other hand, the second
structure theorem on Sasakian manifolds states that 

\begin{proposition}
(\cite{ru}) Let $(M,g)$ be a compact Sasakian manifold of dimension $2n+1$.
Any Sasakian structure $(\xi,\eta,\Phi,g)$ on $M$ is either quasi-regular or
there is a sequence of quasi-regular Sasakian structures $(M,\xi_{i},\eta
_{i},\Phi_{i},g_{i})$ converging in the compact-open $C^{\infty}$-topology to
$(\xi,\eta,\Phi,g).$ In particular, if $M$ admits an irregular Sasakian
structure, it admits many \textbf{locally free circle} actions.
\end{proposition}

\subsection{The Foliated Normal Coordinate}

Let $(M,\eta,\xi,\Phi,g)$ be a compact Sasakian $(2n+1)$-manifold with
$g(\xi,\xi)=1$ and the integral curves of $\xi$ are geodesics. For any point
$p\in M$, we can construct local coordinates in a neighbourhood of $p$ which
are simultaneously foliated and Riemann normal coordinates (\cite{gkn}). That
is, we can find Riemann normal coordinates $\{x,z^{1},z^{2},\cdot\cdot
\cdot,z^{n}\}$ on a neighbourhood $U$ of $p$, such that $\frac{\partial
}{\partial x}=\xi$ on $U$. Let $\{U_{\alpha}\}_{\alpha\in A}$ be an open
covering of the Sasakian manifold and
\[
\pi_{\alpha}:U_{\alpha}\rightarrow V_{\alpha}\subset%
\mathbb{C}
^{n}%
\]
submersion such that $\pi_{\alpha}\circ\pi_{\beta}^{-1}:\pi_{\beta}(U_{\alpha
}\cap U_{\beta})\rightarrow\pi_{\alpha}(U_{\alpha}\cap U_{\beta})$ is
biholomorphic. On each $V_{\alpha},$ there is a canonical isomorphism
\[
d\pi_{\alpha}:D_{p}\rightarrow T_{\pi_{\alpha}(p)}V_{\alpha}%
\]
for any $p\in U_{\alpha},$ where $D=\ker\xi\subset TM.$ Since $\xi$ generates
isometries, the restriction of the Sasakian metric $g$ to $D$ gives a
well-defined Hermitian metric $g_{\alpha}^{T}$ on $V_{\alpha}.$ This Hermitian
metric in fact is K\"{a}hler. More precisely, let $z^{1},z^{2},\cdot\cdot
\cdot,z^{n}$ be the local holomorphic coordinates on $V_{\alpha}$. We pull
back these to $U_{\alpha}$ and still write the same. Let $x$ be the coordinate
along the leaves with $\xi=\frac{\partial}{\partial x}.$ Then we have the
foliation local coordinate $\{x,z^{1},z^{2},\cdot\cdot\cdot,z^{n}\}$ on
$U_{\alpha}\ $and $(D\otimes%
\mathbb{C}
)$ is spanned by the form
\[
Z_{\alpha}=\left(  \frac{\partial}{\partial z^{\alpha}}-\theta\left(
\frac{\partial}{\partial z^{\alpha}}\right)  \frac{\partial}{\partial
x}\right)  ,\ \ \ \alpha=1,2,\cdot\cdot\cdot,n.
\]
Moreover%
\[
\Phi=i\left(  \frac{\partial}{\partial z^{j}}+ih_{j}\frac{\partial}{\partial
x}\right)  \otimes dz^{j}+conj
\]
and
\[
\ \eta=dx-ih_{j}dz^{j}+ih_{\overline{j}}d\overline{z}^{j}.
\]
Here $h$ is basic : $\frac{\partial h}{\partial x}=0$ and $h_{j}%
=\frac{\partial h}{\partial z^{j}},h_{j\overline{l}}=\frac{\partial^{2}%
h}{\partial z^{j}\partial\overline{z}^{l}}$ with
\begin{equation}
\text{\textrm{The normal coordinate :} }h_{j}(p)=0,h_{j\overline{l}}%
(p)=\delta_{j}^{l},dh_{j\overline{l}}(p)=0.\label{AAA3}%
\end{equation}
A frame
\[
\{\frac{\partial}{\partial x},Z_{j}=\left(  \frac{\partial}{\partial z^{j}%
}+ih_{j}\frac{\partial}{\partial x}\right)  ,\ \ \ j=1,2,\cdot\cdot\cdot,n\}
\]
and the dual
\[
\{\eta,dz^{j},\ j=1,2,\cdot\cdot\cdot,n\}
\]
with
\[
\lbrack Z_{i},Z_{j}]=[\xi,Z_{j}]=0.
\]

Since $i(\xi)d\eta=0,$
\[
d\eta(Z_{\alpha},\overline{Z_{\beta}})=d\eta(\frac{\partial}{\partial
z^{\alpha}},\frac{\partial}{\overline{\partial}z^{\beta}}).
\]
Then the K\"{a}hler $2$-form $\omega_{\alpha}^{T}$ of the Hermitian metric
$g_{\alpha}^{T}$ on $V_{\alpha},$ which is the same as the restriction of the
Levi form $d\eta$ to $\widetilde{D_{\alpha}^{n}}$, the slice $\{x=$
\textrm{constant}$\}$ in $U_{\alpha},$ is closed. The collection of K\"{a}hler
metrics $\{g_{\alpha}^{T}\}$ on $\{V_{\alpha}\}$ is so-called a transverse
K\"{a}hler metric. We often refer to $d\eta$ as the K\"{a}hler form of the
transverse K\"{a}hler metric $g^{T}$ in the leaf space $\widetilde{D^{n}}.$

The Kaehler form $d\eta$ on $D$ and the Kaehler metric $g^{T}$ is define such
that
\[
g=g^{T}+\eta\otimes\eta
\]
and
\[
g_{i\overline{j}}^{T}=g^{T}(\frac{\partial}{\partial z^{i}},\frac{\partial
}{\partial\overline{z}^{j}})=d\eta(\frac{\partial}{\partial z^{i}},\Phi
\frac{\partial}{\partial\overline{z}^{j}})=2h_{i\overline{j}}.
\]
In terms of the normal coordinate, we have:%
\[
g^{T}=g_{i\overline{j}}^{T}dz^{i}d\overline{z}^{j},\omega=d\eta
=2ih_{i\overline{j}}dz^{i}\wedge d\overline{z}^{j}.
\]

The transverse Ricci curvature $Ric^{T}$ of the Levi-Civita connection
$\nabla^{T}$ associated to $g^{T}$ is
\[
Ric^{T}=Ric+2g^{T}%
\]
and%
\[
R^{T}=R+2n.
\]
The transverse Ricci form $\rho^{T}$
\[
\rho^{T}=Ric^{T}(J\cdot,\cdot)=-iR_{i\overline{j}}^{T}dz^{i}\wedge
d\overline{z}^{j}%
\]
with
\[
R_{i\overline{j}}^{T}=-\frac{\partial_{2}}{\partial z^{i}\partial\overline
{z}^{j}}\log\det(g_{\alpha\overline{\beta}}^{T})
\]
and it is a closed basic $(1,1)$-form
\[
\rho^{T}=\rho+2d\eta.
\]

\subsection{Basic Cohomology and
Type II Deformation in a Sasakian Manifold
}

\begin{definition}
Let $(M,\eta,\xi,\Phi,g)$ be a Sasakian $(2n+1)$-manifold. Define

A $p$-form $\gamma$ is called basic if%
\[
i(\xi)\gamma=0
\]
and
\[
\mathcal{L}_{\xi}\gamma=0.
\]

\end{definition}

Let $\Lambda_{B}^{p}$ be the sheaf of germs of basic $p$-forms and
\ $\Omega_{B}^{p}$ be the set of all global sections of $\Lambda_{B}^{p}$. It
is easy to check that $d\gamma$ is basic if $\gamma$ is basic. Set
$d_{B}=d|_{\Omega_{B}^{p}}.$ Then%
\[
d_{B}:\Omega_{B}^{p}\rightarrow\Omega_{B}^{p+1}.
\]
We then have the well-defined operators
\[
d_{B}:=\partial_{B}+\overline{\partial}_{B}%
\]
with
\[
\partial_{B}:\Lambda_{B}^{p,q}\rightarrow\Lambda_{B}^{p+1,q}%
\]
and
\[
\overline{\partial}_{B}:\Lambda_{B}^{p,q}\rightarrow\Lambda_{B}^{p,q+1}.
\]
Then for $d_{B}^{c}:=\frac{i}{2}(\overline{\partial}_{B}-\partial_{B}),$ we
have%
\[
d_{B}d_{B}^{c}=i\partial_{B}\overline{\partial}_{B},d_{B}^{2}=(d_{B}^{c}%
)^{2}=0.
\]

The basic Laplacian is defined by
\[
\Delta_{B}:=d_{B}d_{B}^{\ast}+d_{B}^{\ast}d_{B}.
\]
Then we have the basic de Rham complex $($\ $\Omega_{B}^{\ast},d_{B})$ and the
basic Dolbeault complex $($\ $\Omega_{B}^{p,\ast},\overline{\partial}_{B})$
and its cohomology group \ $H_{B}^{\ast}(M,\mathbf{R})$ (\cite{eka}]).

\begin{definition}
(i) We define the basic Cohomology of the foliation $F_{\xi}$by
\[
H_{B}^{\ast}(F_{\xi}):=H_{B}^{\ast}(M,\mathbf{R}).
\]
Then by transverse Hodge decomposition and transverse Serre duality
\[
H_{B}^{p,q}(F_{\xi})\simeq H_{B}^{q,p}(F_{\xi})
\]
and the cohomology of the leaf space $Z=M/U(1)$ to this basic cohomology of
the foliation :
\[
H_{orb}^{\ast}(Z,\mathbf{R})=H_{B}^{\ast}(F_{\xi}):=H_{B}^{\ast}%
(M,\mathbf{R}).
\]

(ii) ([34], Theorem 7.5.17) : Define the basic first Chern class $c_{1}%
^{B}(M)$ by
\[
c_{1}^{B}=[\rho^{T}]_{B}%
\]
and the transverse Einstein (Sasaki $\eta$-Einstein) equation up to a
$D$-homothetic deformation
\[
\lbrack\rho^{T}]_{B}=\varkappa\lbrack d\eta]_{B},\varkappa=-1,0,1.
\]
Basic $k$-th Chern class $c_{k}^{B}(M)$ is represented by a closed basic
$(k,k)$-form $\gamma_{k}$ which is determined by the formula :%
\[
\det(\mathbf{I}_{n}-\frac{1}{2\pi i}\Omega^{T})=1+\gamma_{1}+...+\gamma_{k}.
\]
Here $\Omega^{T}$ is the curvature $2$-form of type basic $(1,1)$ with respect
to the transverse connection $\nabla^{T}$.
\end{definition}

\begin{example}
Let $(M,\eta,\xi,\Phi,g)$ be a compact Sasakian $(2n+1)$-manifold. If $g^{T}$
is a transverse Kaehler metric on $M,$ then $h_{\alpha}=\det((g_{i\overline
{j}}^{\alpha})^{T})$ on $U_{\alpha}$ defines a basic Hermitian metric on the
canonical bundle $K_{M}^{T}$. The inverse $(K_{M}^{T})^{-1}$ of $K_{M}^{T}$ is
sometimes called the anti-canonical bundle. Its basic first Chern class
$c_{1}^{B}((K_{M}^{T})^{-1})$ is called the basic first Chern class of $M$ and
is often denoted by $c_{1}^{B}(M).$Then it follows from the previous result
that $c_{1}^{B}(M)$ that $c_{1}^{B}(M)=[Ric^{T}(\omega)]_{B}$ for any
transversal Kaehler metric $\omega$ on a Sasakian manifold $M$.
\end{example}

\begin{definition}
We define
Type II deformations of Sasakian structures
$(M,\eta,\xi,\Phi,g)$ by fixing the $\xi$ and varying $\eta$. That is, for
$\varphi\in\Omega_{B}^{0}$, define
\[
\widetilde{\eta}=\eta+d_{B}^{c}\varphi,
\]
then
\[
d\widetilde{\eta}=d\eta+d_{B}d_{B}^{c}\varphi=d\eta+i\partial_{B}%
\overline{\partial}_{B}\varphi
\]
and
\[
\widetilde{\omega}=\omega+i\partial_{B}\overline{\partial}_{B}\varphi.
\]

\end{definition}

Note that
we have
\textbf{the same transversal holomorphic foliation ( }$\xi$ is fixed) but with
\textbf{the new Kaehler structure} on the Kaehler cone $C(M)$ and new contact
bundle $\widetilde{D}$ : $\widetilde{\omega}=dd^{c}\widetilde{r}%
,\widetilde{r}=re^{\varphi}$. \textbf{The same holomorphic structure} :
$r\frac{\partial}{\partial r}=\widetilde{r}\frac{\partial}{\partial
\widetilde{r}};\xi=J($ $r\frac{\partial}{\partial r})$ and $\xi+ir\frac
{\partial}{\partial r}=\xi-i\Phi(\xi)$ is the holomorphic vector field on
$C(M).$ Moreover, we have%
\[%
\begin{array}
[c]{ccc}%
\widetilde{\Phi} & = & \Phi-\xi\otimes(d_{B}^{c}\varphi)\circ\Phi,\\
\widetilde{g} & = & d\widetilde{\eta}\circ(Id\otimes\widetilde{\Phi
})+\widetilde{\eta}\otimes\widetilde{\eta}.
\end{array}
\]
and
\[
\mathcal{L}_{\xi}\widetilde{\Phi}=\mathcal{L}_{\xi}\Phi=0.
\]

\section{Blow-up and Resolution of Foliation Singularities}

\subsection{Foliation Singularities}

In this section, parts of notions are due to the papers of Boyer-Galicki
(\cite{bg}) and Wang (\cite{w}). Let $(M,\eta,\xi,\Phi,g)$ be a compact
quasi-regular Sasakian manifold of dimension five and $Z$ denote the space of
leaves of the characteristic foliation $\mathcal{F}_{\xi}$. It follows from
the first srtucture theorem as in Proposition \ref{P21} that $Z$ is a compact
normal, orbifold surface locally given by charts written as quotients of
smooth coordinate charts. That is, $Z$ can be covered by open charts $Z=\cup
U_{i}.$ The orbifold charts on $(Z,U_{i},\varphi_{i})$ is defined by the local
uniformizing systems $(\widetilde{U_{i}},G_{i},\varphi_{i})$ centered at the
point $z_{i}$, where $G_{_{i}}$ is the local uniformizing finite group acting
on a smooth complex space $\widetilde{U_{i}} $ such that
\[
\varphi_{i}:\widetilde{U_{i}}\rightarrow U_{i}=\widetilde{U_{i}}/G_{_{i}}%
\]
is the biholomorphic map. A point $x$ of complex orbifold $Z$ whose isotropy
subgroup $\Gamma_{x}\neq Id$ is called a singular point. Those points with
$\Gamma_{x}=Id$ are called regular points. The set of singular points is
called the orbifold singular locus or orbifold singular set, and is denoted by
$\Sigma^{orb}(Z)$. Let $\Gamma\subset GL(2,\mathbf{C})$ be a finite subgroup.
Then the quotient space $\mathbf{C}^{2}/\Gamma$ is smooth if and only if
$\Gamma$ is a reflection group which fixes a hyperplane in $\mathbf{C}^{2}$.
By the first structure theorem, the underlying complex orbifold space
$\emph{Z}=(Z\emph{,U}_{i}\emph{)}$ is a normal, orbifold variety with the
algebro-geometric singular set $\Sigma(Z)$. Then $\Sigma(Z)\subset\Sigma
^{orb}(Z)$ and it follows that $\Sigma(Z)=\Sigma^{orb}(Z)$ if and only if none
of the local uniformizing groups $\Gamma_{i}$ of the orbifold $\emph{Z}%
=(Z\emph{,U}_{i}\emph{)}$ contain a refletion. Moreover if some $\Gamma_{i}$
contains a reflection, then the reflection fixes a hyperplane in
$\widetilde{\emph{U}_{i}}$ giving rise to a ramification divisor on
$\widetilde{\emph{U}_{i}}$ and a branch divisor on $Z.$ Then we have the
following definition regarding orbifold singular locus and algebro-geometric
singular locus.

\begin{definition}
\label{d2}(i) The branch divisor $\Delta$ of an orbifold $\mathbf{Z}%
=(Z\emph{,}\Delta)$ is a $Q$-divisor on $Z$ of the form
\[
\Delta=\sum_{\alpha}(1-\frac{1}{m_{\alpha}})D_{\alpha},
\]
where the sum is taken over all Weil divisors $D_{\alpha}$ that lie in the
orbifold singular locus $\Sigma^{orb}(Z)$, and $m_{\alpha}$ is the $gcd$ of
the orders of the local uniformizing groups taken over all points of
$D_{\alpha}$ and is called the ramification index of $D_{\alpha}.$

(ii) The orbifold structure $\mathbf{Z}=(Z\emph{,}\Delta)$ is called
\textbf{well-formed} if the fixed point set of every non-trivial isotropy
subgroup has \textbf{codimension at least two}. Then $Z$ is well-formed if and
only if its orbifold singular locus and algebro-geometric singular locus
coincide, equivalently $Z$ has no branch divisors.
\end{definition}

\begin{example}
Fo instance, the weighted projective $\mathbf{CP}(1;4;6)$ has a branch divisor
$\frac{1}{2}D_{0}=\{z_{0}=0\}.$ But $\mathbf{CP}(1;2;3)$ is a unramified
\textbf{well-formed} \ orbifold with two singular points, $(0;1;0) $ with
local uniformizing group the cyclic group $\mathbf{Z}_{2}$, and $(0;0;1)$ with
local uniformizing group $\mathbf{Z}_{3}$.
\end{example}

Let $(M,\xi,\alpha_{\xi})$ be a compact quasi-regular Sasakian $5$-manifold
and the leave space $Z:=M/S^{1}$ be an orbifold Kaehler surface. Denote by the
quotient map
\[
\pi:M\rightarrow Z
\]
as before, and call such a $5$-manifold $(M,\xi)$ an $S^{1}$-orbibundle. More
precisely, $M$ admits a locally free, effective $S^{1}$-action
\[
\alpha_{\xi}:S^{1}\times M\rightarrow M
\]
such that $\alpha_{\xi}(t)$ is orientation-preserving, for every $t\in S^{1}$.
Since $M$ is compact, $\alpha_{\xi}$ is proper and the isotropy group
$\Gamma_{p}$ of every point $p\in M$ is finite.

\begin{definition}
The principal orbit type $M_{\operatorname{reg}}$ corresponds to points in
$(M,\xi,\alpha_{\xi})$ with the trivial isotropy group and
$M_{\operatorname{reg}}\rightarrow M_{\operatorname{reg}}/S^{1}$ is a
principle $S^{1}$-bundle. Furthermore, the orbit $\mathbf{S}_{p}^{1}$ of a
point $p\in M$ is called a regular fiber if $p\in M_{\operatorname{reg}}$, and
a singular fiber otherwise. In this case, $M_{\mathrm{sing}}/S^{1}\simeq
\Sigma^{orb}(Z).$
\end{definition}

\textbf{The Proof of Theorem \ref{T34} : }

\begin{proof}
Since each orbit ${\mathsf{S_{p}^{1}}}$ in an $S^{1}$-orbibundle
$(M,\xi,\alpha_{\xi})$ admits an $S^{1}$-invariant neighborhood $U_{p}$ such
that it is $S^{1}$-equivalent to a standard $S^{1}$-structure $V_{p}%
:=S^{1}\times\mathbb{C}^{2}/\mu_{p}$, where $\mu_{p}$ is a linear
representation of $\Gamma_{p}$ of $p$ on the vector space $\mathbb{C}^{2}$,
and the $S^{1}$-action on $\mathbb{C}^{2}$ is induced by the Lie group
structure of $S^{1}$. Since $\alpha$ is effective and $\mu_{p}$ is faithful.
If in addition $\mathbb{C}^{2}$ is endowed with a complex structure, one can
enhance the notion of a standard $S^{1}$-structure, and define the standard
complex $S^{1}$-structure which is the standard $S^{1}$-structure $S^{1}%
\times\mathbb{C}^{2}/\mu_{p}$ with $\mu_{p}$ complex linear. Note that, up to
conjugation, a representation $\mu_{p}$ of $\mathbb{Z}_{r}$ on $\mathbb{C}%
^{2}$ is conjugate to
\begin{align*}
\mathbb{Z}_{r} &  \rightarrow\mathbb{C}^{2}\\
\gamma &  \mapsto%
\begin{bmatrix}
e^{2\pi i\frac{a_{1}}{r_{1}}} & 0\\
0 & e^{2\pi i\frac{a_{2}}{r_{2}}},
\end{bmatrix}
\end{align*}
for some positive integers $r_{1},r_{2}$ whose least common multipler is $r$,
and $a_{i},i=1,2$ are integers coprime to $r_{i},i=1,2$. If $r_{1},r_{2}<r$,
then $(M,\xi,\alpha_{\xi})$ has a non-discrete set of fibers; more precisely,
$M_{\mathrm{sing}}$ contains some $3$-dimensional submanifolds of $M$. In
particular, in the case where $M_{\mathrm{sing}}$ is discrete, we have%
\[
r_{1}=r_{2}=r.
\]
Up to change of generator, the representation $\mu_{p}$ can be normalized as
follows:
\begin{align}
\mathbb{Z}_{r} &  \rightarrow\mathbb{C}^{2}\label{2022B}\\
\gamma &  \mapsto%
\begin{bmatrix}
e^{2\pi i\frac{1}{r}} & 0\\
0 & e^{2\pi i\frac{a}{r}},
\end{bmatrix}
\nonumber
\end{align}
for some $a$ coprime to $r$. Then Theorem \ref{T34} follows easily.
\end{proof}

Now we come out with the following definition.

\begin{definition}
Given $p\in M_{\mathrm{sing}}$, the singular fiber $\mathbf{S}_{p}^{1}$ is a
foliation singularity of type $\frac{1}{r}(1,a)$ at $p$ with $r>0,a$ two
integers and $a$ coprime with $r$, if there exists an isomorphism $\Gamma
_{p}\simeq\mathbb{Z}_{r}=<\gamma>$ such that $\mu_{p}$ can be identified with
the representation (\ref{2022B}), up to conjugation. This corresponds to the
orbifold structure $\mathbf{Z}$ is well-formed where the fixed point set of
every non-trivial isotropy subgroup has codimension at least two.
\end{definition}

Since $M-M_{\operatorname{reg}}$ is a union of smooth submanifolds of $M$. The
present paper concerns mainly the case where $M_{\mathrm{sing}}%
=M-M_{\operatorname{reg}}$ is a discrete, and hence finite set.

\subsection{Local Model of $\frac{1}{k}(1,1)$-type Foliation Singularities}

Let $(M,\eta,\xi,\Phi,g)$ be a compact quasi-regular Sasakian $5$-manifold and
$\mathbf{S}_{p}^{1}$ be a discrete singular fiber of type $\frac{1}{k}(1,a)$
at $p\in M.$ Then there is an $\mathbf{S}^{1}$-equivariant neighborhood
$U_{p}$ of $\mathbf{S}_{p}^{1}$ and $\mathbf{S}^{1}$-equivalent to a standard
complex $\mathbf{S}^{1}$-structure $\mathbf{S}^{1}\times\mathbf{C}%
^{2}/\mathbf{Z}_{r}.$ In this subsection, we examine the $S^{1}$-structure of
a neighborhood of a singular fiber of $\frac{1}{k}(1,a)$ type. Then, by the
slice theorem as discussed in the previous subsection, there is an $S^{1}%
$-equivariant neighborhood $U_{p}$ of ${\mathsf{S_{p}^{1},}}$ $S^{1}
$-equivalent to a standard complex $S^{1}$-structure $S^{1}\times
\mathbb{C}^{2}/\mathbb{Z}_{r}$. The $S^{1}$-action on $S^{1}\times
_{\mathbb{Z}_{r}}\mathbb{C}^{2}$ is induced by the $S^{1}$-action on
$S^{1}\times\mathbb{C}^{2}$ given by
\begin{align*}
S^{1}\times S^{1}\times\mathbb{C}^{2} &  \rightarrow S^{1}\times\mathbb{C}%
^{2}\\
(t,w,z_{1},z_{2}) &  \mapsto(tw,z_{1},z_{2}).
\end{align*}

Now we consider the local diffeomorphism
\begin{align*}
\pi:S^{1}\times\mathbb{C}^{2} &  \rightarrow S^{1}\times\mathbb{C}^{2}\\
(w,z_{1},z_{2}) &  \mapsto(w^{k},w^{-1}z_{1},w^{-a}z_{2})
\end{align*}
which descends to a diffeomorphism from $S^{1}\times\mathbb{C}^{2}%
/\mathbb{Z}_{r}$ to $S^{1}\times\mathbb{C}^{2}$. Via the diffeomorphism, one
can identify the neighborhood $U_{p}$ with $S^{1}\times\mathbb{C}^{2}$ with
the $S^{1}$-action on $S^{1}\times\mathbb{C}^{2}$ given by
\begin{align}
S^{1}\times S^{1}\times\mathbb{C}^{2} &  \rightarrow S^{1}\times\mathbb{C}%
^{2}\label{eq:s1_action_nbhd_singular}\\
(t,u,v_{1},v_{2}) &  \mapsto(t^{k}u,t^{-1}v_{1},t^{-a}v_{2}).\nonumber
\end{align}

The above identification also yields an $S^{1}$-equivalence between a closed
tubular neighborhood of ${\mathsf{S_{p}^{1}}}$ and $S^{1}\times D^{4}$, where
$D^{4}$ is the $4$-ball in $\mathbb{C}^{2}$, and the $S^{1}$-action on
$S^{1}\times D^{4}$ is given by
\begin{align}
S^{1}\times S^{1}\times D^{4} &  \rightarrow S^{1}\times D^{4}%
\label{eq:s1_action_closed_nbhd_singular}\\
(t,u,v_{1},v_{2}) &  \mapsto(t^{k}u,t^{-1}v_{1},t^{-a}v_{2}).\nonumber
\end{align}

Now we are ready to consider the resolution of a singular fiber
${\mathsf{S_{\mathbf{p}}^{1}}}$ of type $\frac{1}{k}(1,1)$ at $p\in M$.
Consider the space ${\mathcal{S}(-k)}$ given by
\[
\{\big((z_{1},z_{2}),(u_{1},u_{2})\big)\in\mathbb{C}^{2}\times S^{3}\mid
z_{1}u_{2}^{k}=z_{2}u_{1}^{k}\},
\]
where $S^{3}$ is the standard sphere in $\mathbb{C}^{2}$ and oriented using
inward normal. ${\mathcal{S}(-k)}$ admits a natural free $S^{1}$-action given
by the restriction of
\begin{align}
\alpha_{k}:S^{1}\times\mathbb{C}^{2}\times S^{3} &  \rightarrow\mathbb{C}%
^{2}\times S^{3}\label{eq:s1_action_on_Sk}\\
(t,z_{1},z_{2},u_{1},u_{2}) &  \mapsto(z_{1},z_{2},t^{-1}u_{1},t^{-1}%
u_{2}).\nonumber
\end{align}
${\mathcal{S}(-k)}$ can be identified with $\mathbb{C}\times S^{3}$ via the
diffeomorphism
\begin{align}
\mathbb{C}\times S^{3} &  \rightarrow{\mathcal{S}(-k)}%
\label{eq:identification_Sk_CS3}\\
(z,(w_{1},w_{2})) &  \mapsto(zw_{1}^{k},zw_{2}^{k},w_{1},w_{2})\nonumber
\end{align}
which induces a free $S^{1}$-action $\beta_{k}$ on $\mathbb{C}\times S^{3}$:
\begin{align}
\beta_{k} &  :S^{1}\times\mathbb{C}\times S^{3}\rightarrow\mathbb{C}\times
S^{3}\label{eq:s1_action_on_tildeU}\\
(t,z,w_{1},w_{2}) &  \mapsto(t^{k}z,t^{-1}w_{1},t^{-1}w_{2}).\nonumber
\end{align}
We remark that the quotient of the embedding ${\mathcal{S}(-k)}\subset
\mathbb{C}^{2}\times S^{3}$ by $\alpha_{0}$ is the embedding ${\mathcal{O}%
(-k)}\subset\mathbb{C}^{2}\times\mathbf{C}P^{1}$, and furthermore, there is a
pullback diagram%
\begin{equation}%
\begin{array}
[c]{ccccc}%
\overline{M}\supset\mathcal{S}(-k)\thickapprox\mathbf{C}\times\mathbf{S}^{3} &
\overset{\pi_{2}}{\rightarrow} & \mathbf{S}^{3} & \overset{}{\longrightarrow}
& M\supset\mathbf{S}_{\mathbf{p}}^{1}\\
\downarrow &  & \downarrow-H &  & \downarrow\pi_{M}\\
Z\supset O(-k) & \overset{\pi_{d}}{\rightarrow} & \mathbf{CP}^{1} &
\overset{}{\longrightarrow} & Z\ni\pi_{M}(\mathbf{p}).
\end{array}
\label{2021a}%
\end{equation}
Here $\pi_{2}$ is the projection onto the second factor, and $\mathsf{H}$ is
the Hopf fibration. Note that $D^{2}\times\mathbf{S}^{3}\simeq\mathbf{C\times
S}^{3}$ can be identified with the subspace of $\mathbf{C}^{2}\times
\mathbf{S}^{3}$ given by%
\[
\{(z_{1},z_{2}),(u_{1},u_{2})|\text{ }z_{1}u_{2}^{k}-z_{2}u_{1}^{k}=0\}
\]
which is the coordinate for the canonical line bundle $\mathcal{S}(-k)$ over
$\mathbf{S}^{3}$and $\overline{H}$ can be realized by projecting down the
first three coordinates to with the subspace of $\mathbf{C}^{2}\times
\mathbf{CP}^{1}$ given by%
\[
\{(z_{1},z_{2},[l_{1},l_{2}])|\text{ }z_{1}l_{2}^{k}-z_{2}l_{1}^{k}=0\}
\]
which is the coordinate for the canonical line bundle $\mathcal{O}(-k)$ over
$\mathbf{CP}^{1}.$

Now, to resolve the singular fiber ${\mathsf{S_{\mathbf{p}}^{1}}}$ amounts to
replacing $U$ with ${\mathcal{S}(-k)}$. To see how this is done, we observe
that there is an $S^{1}$-equivariant embedding
\begin{align}
\Psi:S^{1}\times(\mathbb{C}^{2}-\{0,0\}) &
\xrightarrow{\iota} {\mathcal{S}(-k)}\subset\mathbb{C}^{2}\times
S^{3}\label{eq:s1_equivariant_embedding}\\
(s,z_{1},z_{2}) &  \mapsto(sz_{1}^{k},sz_{2}^{k},\frac{z_{1}}{\sqrt
{|z_{1}|^{2}+|z_{2}|^{2}}},\frac{z_{2}}{\sqrt{|z_{1}|^{2}+|z_{2}|^{2}}%
}).\nonumber
\end{align}
The complement of the image in $\tilde{U}$ is precisely the $3$-sphere
$\{0,0\}\times S^{3}$ where $\alpha_{0}$ restricts to give the negative Hopf
fibration $-\mathsf{H}$.

We define $(M^{\prime},\alpha^{\prime})$ to be the new $S^{1}$-orbibundle
given by gluing $M-{\mathsf{S_{\mathbf{p}}^{1}}}$ and ${\mathcal{S}(-k)}$ via
the identification between $U-{\mathsf{S_{\mathbf{p}}^{1}}}$ and
${\mathcal{S}(-k)}-\{0,0\}\times S^{3}$ given in above with $\alpha^{\prime}$
induced by $\alpha_{\xi}$ and $\alpha_{k}$. In other words, we remove the
singular fiber by replacing it with a negative Hopf fibration. It follows that

\begin{lemma}
\label{L52} $(M^{\prime},\alpha^{\prime})$ is again a compact quasi-regular
Sasakian $5$-manifold.
\end{lemma}

\subsubsection{The Coordinates on $\mathcal{S}(-k)$}

Let $U_{i},i=1,2$ be two copies of $\mathbf{C}\times\mathbf{C}\times
\mathbf{S}^{1}$ denoted by $(x_{i},v_{i},s_{i})$ the coordinate of a point in
$U_{i}$. Consider the embedding%
\[%
\begin{array}
[c]{ccl}%
U_{1} & \overset{}{\rightarrow} & \mathcal{S}(-k)\\
(x_{1},v_{1},s_{1}) & \mapsto & (v_{1},v_{1}x_{1}^{k},(\frac{s_{1}}%
{\sqrt{1+|x_{1}|^{2}}},\frac{s_{1}x_{1}}{\sqrt{1+|x_{1}|^{2}}}))
\end{array}
\]
and
\[%
\begin{array}
[c]{ccl}%
U_{2} & \overset{}{\rightarrow} & \mathcal{S}(-k)\\
(x_{2},v_{2},s_{2}) & \mapsto & (v_{2}x_{2}^{k},v_{2},(\frac{s_{2}x_{2}}%
{\sqrt{1+|x_{2}|^{2}}},\frac{s_{2}}{\sqrt{1+|x_{2}|^{2}}}))
\end{array}
\]
which induce coordinate charts covering $\mathcal{S}(-k)$ with the transition
function given by
\[
x_{2}=x_{1}^{-1},v_{2}=v_{1}x_{1}^{k},s_{2}=s_{1}\frac{x_{1}}{x_{2}}.
\]

Note that the quotients $\overline{U}_{i}$ of $U_{i},i=1,2,$ give us the
standard coordinates of $\mathcal{O}(-k),$ namely,%
\[%
\begin{array}
[c]{ccl}%
\overline{U}_{1} & \overset{}{\rightarrow} & \mathcal{O}(-k)\\
(x_{1},v_{1}) & \mapsto & (v_{1},v_{1}x_{1}^{k},[1,x_{1}])
\end{array}
\]
and
\[%
\begin{array}
[c]{ccl}%
\overline{U}_{i2} & \overset{}{\rightarrow} & \mathcal{O}(-k)\\
(x_{2},v_{2}) & \mapsto & (v_{2}x_{2}^{k},v_{2},[x_{2},1]),
\end{array}
\]
where $[l_{1},l_{2}]$ is the homogeneous coordinate of a point in
$\mathbf{CP}^{1}$. In particular, we have $x_{1}=\frac{l_{2}}{l_{1}}%
;x_{2}=\frac{l_{1}}{l_{2}}$, and $v_{i},i=1,2$ correspond to vectors in a
fiber of $\mathcal{O}(-k).$

\subsubsection{\textbf{Distance to the Zero Section}}

Let $\mathbf{S}_{\mathbf{0}}^{3}$ be the zero section of
\[
\pi_{2}:\mathcal{S}(-k)\subset\mathbf{C}^{2}\mathbf{\times S}^{3}%
\rightarrow\mathbf{S}^{3}%
\]
\ and endow $\mathcal{S}(-k)$ with the metric induced by the standard metric
on $\mathbf{C}^{2}\mathbf{\times S}^{3}.$ Then the distance square
$dist^{2}(y,\mathbf{S}^{3})$ between a point $y\in\mathcal{S}(-k)$ and
$\mathbf{S}^{3}$ expressed in terms of the coordinates charts is
\[
dist^{2}(y,\mathbf{S}_{0}^{3})=(1+|x_{i}|^{2k})|v_{i}|^{2},
\]
where $y=(x_{i},v_{i},s_{i})\in U_{i},i=1,2.$

Now we define the section $s$ of $[V\mathcal{]}$ over $\mathcal{S}(-k)$ by%
\[
s_{i}:U_{i}\rightarrow\mathbf{C,}\text{ \ \ }s_{i}=v_{i}.
\]
Note that $s$ is vanishing along the \ exceptional basic divisor $V$ and the
hermitian metric $h$ on $[V]$ such that
\[
h_{i}=(\frac{|l_{1}|^{2k}+|l_{2}|^{2k}}{|l_{i}|^{2k}})=(1+|x_{i}|^{2k})
\]
on $U_{i}.$ If furthermore the composition
\[%
\begin{array}
[c]{ccc}%
\mathcal{S}(-k)\thickapprox\mathbf{C}\times\mathbf{S}^{3} & \overset{\pi
_{2}}{\rightarrow} & \mathbf{S}^{3}\\
&  & \downarrow-H\\
&  & \mathbf{CP}^{1}%
\end{array}
\]
send $y$ to $[l_{1},l_{2}],$ then
\begin{equation}
dist^{2}(y,\mathbf{S}^{3})=(\frac{|l_{1}|^{2k}+|l_{2}|^{2k}}{|l_{i}|^{2k}%
})|v_{i}|^{2}:=|s|_{h}^{2}\label{21}%
\end{equation}
where $y\in U_{i}.$

\subsubsection{Distance to the Singular Fiber}

The standard metric on $\mathbf{C}^{2}\mathbf{\times S}^{3}$ also induces a
metric on $U-\mathbf{S}_{\mathbf{p}}^{1}=$ $\mathbf{S}^{1}\times
(\mathbf{C}^{2}-\{(0,0)\})$ via
\[%
\begin{array}
[c]{ccl}%
\mathbf{S}^{1}\times(\mathbf{C}^{2}-\{(0,0)\}) & \overset{i}{\rightarrow} &
\mathcal{S}(-k)=\mathbf{C}\times\mathbf{S}^{3}\subset\mathbf{C}^{2}%
\times\mathbf{S}^{3}.
\end{array}
\]
In particular, given a point $x=(s,z_{1},z_{2})$, the image $\Psi\circ\tau(t)$
of the ray $\tau(t)=(s,tz_{1},tz_{2}),0\leq t\leq1$ is a geodesic
\[
(sz_{1}^{k}t^{k},sz_{2}^{k}t^{k},\frac{z_{1}}{\sqrt{|z_{1}|^{2}+|z_{2}|^{2}}%
},\frac{z_{2}}{\sqrt{|z_{1}|^{2}+|z_{2}|^{2}}})\subset\mathbf{C}^{2}%
\times\mathbf{S}^{3}%
\]
away from zero, so the distance square to the singular fiber is%
\begin{equation}
dist^{2}(x,\mathbf{S}_{p}^{1})=|z_{1}|^{2k}+|z_{2}|^{2k}:=r_{\mathbf{S}%
_{\mathbf{p}}^{1}}^{2k}.\label{22}%
\end{equation}

Furthermore, $\Psi\circ\tau(t)$ is a geodesic from $\Psi(x)\in\mathcal{S}(-k)$
to the zero section $\mathbf{S}_{\mathbf{0}}^{3}$. Then
\begin{equation}
\Psi^{\ast}(dist^{2}(\Psi(x),\mathbf{S}^{3}))=dist^{2}(x,\mathbf{S}_{p}%
^{1}).\label{2022-9}%
\end{equation}

\subsubsection{Foliation Blow-up and Sasaki Castelnuovo's Contraction Theorem}

As in (\ref{2021a}) with $k=1,$ we pull back a negative Hopf fibration $-H$
along the canonical disk bundle $\pi_{d}:O(-1)\rightarrow\mathbf{CP}^{1}$, we
get a $\mathbf{S}^{1}$-bundle $\mathcal{S}(-1)$ over $O(-1)$, which is
necessarily diffeomorphic to $\mathbf{S}^{3}\times D^{2}$, since the only
$D^{2}$-bundle over $\mathbf{S}^{3}$ is the trivial one :
\begin{equation}%
\begin{array}
[c]{ccccc}%
\overline{M}\supset\mathcal{S}(-1)=\mathbf{S}^{3}\times D^{2} & \rightarrow &
\mathbf{S}^{3} & \overset{\psi|_{V}}{\longrightarrow} & \mathbf{S}_{p}%
^{1}\subset\mathbf{S}^{1}\times D^{4}\subset M\\
\downarrow &  &  &  & \downarrow\pi\\
\overline{Z}\supset O(-1) & \rightarrow & \mathbf{CP}^{1} & \overset{\phi
|_{E}}{\longrightarrow} & z_{p}=\pi(p)\in Z.
\end{array}
\label{2021b}%
\end{equation}
Here $\phi:\overline{Z}\rightarrow Z$ is a blow-up of $Z$ at a nonsingular
point $z_{p}=\pi(p)\in Z.$

As the first consequence of previous subsections (cf Lemma \ref{L52}, etc), it
follows that we have the following foliation blow-up along a regular fibre
$\mathbf{S}_{p}^{1}$ of type $\frac{1}{k}(1,1)$ with $k=1:$

\begin{theorem}
\label{D31} Let $(M,\eta,\xi,\Phi,g)$ be a compact regular Sasakian
$5$-manifold and $Z$ denote the space of leaves of the characteristic
foliation which is a smooth projective surface with the regular $\mathbf{S}%
^{1}$-principal bundle over $M$. Then

\begin{enumerate}
\item There is a compact regular Sasakian $5$-manifold $\overline{M}$ which is
obtained by gluing
\[
M-\overset{\circ}{\mathcal{N}}({\mathsf{S_{\mathbf{p}}^{1}}})
\]
and $\mathbf{S}^{3}\times D^{2}$ along their boundary via an
orientation-preserving diffeomorphism%
\[
f:\mathbf{S}^{3}\times\mathbf{S}^{1}\simeq\partial(\mathbf{S}^{3}\times
D^{2})\rightarrow-\partial(M-\overset{\circ}{\mathcal{N}}(\mathbf{S}_{p}%
^{1}))\simeq\mathbf{S}^{3}\times\mathbf{S}^{1}.
\]
In other words, we have the blow-up map $\psi$
\[
\psi:\overline{M}\rightarrow M
\]
and $\overline{M}$ is called the foliation blowing up along a regular fibre
$\mathbf{S}_{p}^{1}$ at a point $p\in M$
\[
\overline{M}=(M-\overset{\circ}{\mathcal{N}}(\mathbf{S}_{p}^{1}))\cup
_{f}\mathbf{S}^{3}\times D^{2}%
\]

\item For the irreducible transverse exceptional divisor $V\in Exc(\psi)$ and
the irreducible exceptional divisor $E\in Exc(\phi)$, the foliation
$(-1)$-curve $V$ of such a regular fibre $\mathbf{S}_{p}^{1}$ at a point $p\in
M$ is a regular $\mathbf{S}^{1}$-principal bundle over a compact Riemann
sphere $\mathbf{S}^{2}$. Moreover, $E$ is biholomorphic to $\mathbf{CP}^{1}$
and $V$ is transverse biholomorphic to $\mathbf{S}^{3}.$ Furthermore it
follows from (\ref{A1}) and (\ref{A2}) that
\[
E\cdot E=-1
\]
and then
\[
V\cdot V=-1.
\]

\item For the foliation $(-1)$-curve blow-up map $\psi:\overline{M}\rightarrow
M$, it follows from (\ref{21}), (\ref{22}) and (\ref{2022-9}) that
\begin{equation}
\psi^{\ast}r_{\mathbf{S}_{p}^{1}}^{2}=|s|_{h}^{2}.\label{23}%
\end{equation}
In general, for the foliation $(-k)$-curve contraction map, we have
\begin{equation}
\psi^{\ast}r_{\mathbf{S}_{p}^{1}}^{2k}=|s|_{h}^{2}.\label{24}%
\end{equation}

\end{enumerate}
\end{theorem}

In particular, let $X$ and $Y$ be smooth projective surfaces with the
Castelnuovo's Contraction map $\phi:X\rightarrow Y,$ we have the following
lifting transverse contraction morphism $\psi:M\rightarrow N$ via the regular
$S^{1}$-principal bundles $\pi_{M}$ and $\pi_{N}$ \ as in
\[%
\begin{array}
[c]{ccl}%
M\supset V & \overset{\psi}{\longrightarrow} & N\supset\mathbf{S}_{p}^{1}\\
\downarrow\pi_{M} & \circlearrowright & \downarrow\pi_{N}\\
X\supset E & \overset{\phi}{\longrightarrow} & Y\ni\pi_{N}(p).
\end{array}
\]
Then, as the consequence of Theorem \ref{D31}, we have the Sasaki analogue of
Castelnuovo's contraction Theorem \ref{T33} on Sasakian Five-Manifolds.

\begin{remark}
In general for a Sasakian $3$-manifold, it is an $S^{1}$-Seifert $3$-manifold
and the Euler characteristic $e$ of the Seifert bundle is nonzero if it is
closed. For a Seifert bundle, \textbf{the geometry is determined by the Euler
characteristic }$\chi$\textbf{\ of the base }$2$\textbf{-orbifold }%
(\cite{gei}) :
\[%
\begin{tabular}
[c]{llll}
& $\chi>0$ & $\chi=0$ & $\chi<0$\\
$e=0$ & $S^{2}\times\mathbb{E}$ & $\mathbb{E}^{3}=\mathbb{E}^{2}%
\times\mathbb{E}$ & $H^{2}\times\mathbb{E}$\\
$e\neq0$ & $S^{3}$ & $Nil$ & $\widetilde{SL(2,\mathbb{R)}}$%
\end{tabular}
\ \ \ \ \ \ \ \ \ \ \ \ \ \ .\ \
\]

\end{remark}

\begin{example}
As an example, $\mathbf{S}^{5}$ is equipped with a free $\mathbf{S}^{1}%
$-action given by the Hopf bration:%
\[%
\begin{array}
[c]{ccl}%
\mathbf{S}^{5} & \overset{H}{\longrightarrow} & \mathbf{CP}^{2}.
\end{array}
\]
Then the blowing up along a regular fibre $\mathbf{S}_{p}^{1}$ in
$\mathbf{S}^{5}$ via the the gluing map gives us the $\mathbf{S}^{1}$-bundle
\[%
\begin{array}
[c]{ccl}%
\mathbf{S}^{5}\#(\mathbf{S}^{3}\mathbf{\times S}^{2}) & \overset{\pi
}{\longrightarrow} & \mathbf{CP}^{2}\#\overline{\mathbf{CP}^{2}}%
\end{array}
\]
or%
\[%
\begin{array}
[c]{ccl}%
\mathbf{S}^{5}\#(\mathbf{S}^{3}\widetilde{\times}\mathbf{S}^{2}) &
\overset{\pi}{\longrightarrow} & \mathbf{CP}^{2}\#\overline{\mathbf{CP}^{2}}.
\end{array}
\]
In the case $N$ is simply-connected, the blowing up along a regular fibre
$\mathbf{S}_{p}^{1}$ is either $N\#(\mathbf{S}^{3}\mathbf{\times S}^{2})$ or
$N\#(\mathbf{S}^{3}\widetilde{\times}\mathbf{S}^{2})$, twisted $\mathbf{S}%
^{3}$-bundle over $\mathbf{S}^{2}$, a non-spin manifold. More precisely,
since
\[
\pi_{1}(Diff(\mathbf{S}^{3}))\simeq\pi_{1}(SO(4))\simeq\mathbf{Z}_{2},
\]
the blowing up along $\mathbf{S}_{p}^{1}$ can produce at most two
$\mathbf{S}^{1}$-manifolds, depending on the gluing map. As an example,
$\mathbf{S}^{5}$ is equipped with a free $\mathbf{S}^{1}$-action given by the
negative Hopf fibration. Then blowing up along $\mathbf{S}_{p}^{1}$ via $f=id$
gives us the $\mathbf{S}^{1}$-bundle
\[%
\begin{array}
[c]{ccc}%
M\#(\mathbf{S}^{3}\mathbf{\times S}^{2}) & \rightarrow & \mathbf{S}^{5}\\
\downarrow &  & \downarrow-H\\
\mathbf{CP}^{2}\#\overline{\mathbf{CP}^{2}} & \overset{}{\rightarrow} &
\mathbf{CP}^{2}.
\end{array}
\]
On the other hand, if we take the gluing map $f$ to be
\[
(t,[%
\begin{array}
[c]{cccc}%
\cos t & -\sin t & 0 & 0\\
\sin t & \cos t & 0 & 0\\
0 & 0 & 1 & 0\\
0 & 0 & 0 & 1
\end{array}
]),
\]
then%
\[%
\begin{array}
[c]{ccc}%
M\#(\mathbf{S}^{3}\widetilde{\times}\mathbf{S}^{2}) & \rightarrow &
\mathbf{S}^{5}\\
\downarrow &  & \downarrow-H\\
\mathbf{CP}^{2}\#\overline{\mathbf{CP}^{2}} & \overset{}{\rightarrow} &
\mathbf{CP}^{2}.
\end{array}
\]

\end{example}

\subsection{Reid's Model and Resolution of $\frac{1}{r}(1,a)$-type Foliation
Singularities}

In general, for a foliation singularity of type $\frac{1}{r}(1,a),$ we can
reduce to a foliation singularity of type $\frac{1}{k}(1,1)$ in this
subsection. More precisely, given $p\in M_{\mathrm{sing}}$, the singular fiber
${\mathsf{S_{p}^{1}}}$ is a foliation singularity of type $\frac{1}{r}(1,a)$
at $p$, one can measure its complexity by the Hirzebruch-Jung continued fraction.

\begin{definition}
Given $p\in M_{\mathrm{sing}}$, the singular fiber ${\mathsf{S_{p}^{1}}}$ is
foliation singularity of type $\frac{1}{r}(1,a)$. The Hirzebruch Jung
continued fraction $[b_{1},\dots,b_{l}]$ of $\frac{r}{a}$ is defined by
\[
\frac{r}{a}=b_{1}-\frac{1}{b_{2}-\frac{1}{b_{3}-\dots}}.
\]
Then we define the length of the singular fiber to be $l$. Note that the
fraction can be calculated by the recursive formula:
\begin{align}
r &  =ab_{1}-a_{1},\nonumber\\
a &  =a_{1}b_{2}-a_{2},\nonumber\\
a_{i} &  =a_{i+1}b_{i+2}-a_{i+2},\nonumber\\
&  \dots.
\end{align}

\end{definition}

In the paper of \cite[section $2$]{w}, Wang presents a construction that
reduce a foliation singularity of type $\frac{1}{r}(1,a)$ to a foliation
singularity of type $\frac{1}{a}(1,a_{1})$ with $r=ab_{1}-a_{1}.$ \ More
precisely, we consider the five manifold ${\mathbf{C}}\times{\mathbf{S}^{3}} $
equipped with the ${\mathbf{S}^{1}}$-action
\begin{align*}
\delta_{a}:{\mathbf{S}^{1}}\times({\mathbf{C}}\times{\mathbf{S}^{3}}) &
\rightarrow{\mathbf{C}}\times{\mathbf{S}^{3}}\\
(x,y_{1},y_{2}) &  \mapsto(t^{r}x,t^{-1}y_{1},t^{-a}y_{2}),
\end{align*}
where $(x,y_{1},y_{2})\in{\mathbf{C}}\times{\mathbf{S}^{3}}\subset{\mathbf{C}%
}^{3}$. $({\mathbf{C}}\times{\mathbf{S}^{3}},\delta_{a})$ has a singular fiber
of type $\frac{1}{a}(1,a_{1})$ with $\frac{a}{a_{1}}=[b_{2},\dots,b_{l}]$ and
$\frac{r}{a}=[b_{1},b_{2},\dots,b_{l}]$.

Now, given a Sasakian manifold $(M,\alpha),$ $p\in M_{\mathrm{sing}}$, the
singular fiber ${\mathsf{S_{\mathbf{p}}^{1}}}$ is of type $\frac{1}{r}(1,a)$
of lenth $l$. Then there exists an ${\mathbf{S}^{1}}$-equivariant neighborhood
$\overset{\circ}{\mathcal{N}}({\mathsf{S_{\mathbf{p}}^{1}}})$ of
${\mathsf{S_{\mathbf{p}}^{1}}}$ such that $\overset{\circ}{\mathcal{N}}$ is
${\mathbf{S}^{1}}$-equivariant to $(N,\nu)$ with $N\simeq{\mathbf{S}^{1}%
}\times{\mathbf{C}}^{2}$ and $\nu$ given by%
\begin{equation}%
\begin{array}
[c]{ccc}%
{\mathbf{S}^{1}}\times N & \rightarrow & N\\
(t,u,v_{1},v_{2}) & \mapsto & (t^{r}u,t^{-1}v_{1},t^{-a}v_{2}).
\end{array}
\label{2022A}%
\end{equation}
For the ${\mathbf{S}^{1}}$-equivariant embedding
\begin{align}
\Psi:(\overset{\circ}{\mathcal{N}}({\mathsf{S_{\mathbf{p}}^{1}}}%
)-{\mathsf{S_{\mathbf{p}}^{1}}})\simeq(N-{\mathbf{S}^{1}}\times(0,0)) &
\rightarrow{\mathbf{C}}\times{\mathbf{S}^{3}}\nonumber\\
(u,v_{1},v_{2}) &  \mapsto(u|v|,\frac{v_{1}}{|v|},\frac{v_{2}}{|v|}%
),\label{eq:gluing_embedding}%
\end{align}
where $v=(v_{1},v_{2})$, we glue the manifold%
\begin{equation}
(M^{\prime},\alpha^{\prime})=(M-\overset{\circ}{\mathcal{N}}%
({\mathsf{S_{\mathbf{p}}^{1}}}),\alpha)\underset{\Psi}{\cup}({\mathbf{C}%
}\times{\mathbf{S}^{3}},\delta_{a}).\label{2022b}%
\end{equation}
Then resulting new Sasakian manifold $(M^{\prime},\alpha^{\prime})$ has all
singular fibers the same as $(M,\alpha)$ except the singular fiber
${\mathsf{S_{\mathbf{p}}^{1}}}$ is now replaced with a (singular) fiber of
smaller length. Moreover, we have
\[
M^{\prime}\simeq M\#{\mathbf{S}^{3}}\times{\mathbf{S}^{2}}%
\]
topologically.

Furthermore, it follows from Reid's model (\cite{r}) on Kaehler surfaces that
the gluing map $\Psi$ descends to the gluing maps $(2.9)$ and $(2.10)$ of the
paper \cite[section $2.2$]{w}. Then we have the following the minimal
resolution of the general $\frac{1}{r}(1,a)$-type foliation singularities of
$M$ which is compatibable with the minimal resolution in case of Kaehler surfaces.

\begin{theorem}
\label{T32} Let $(M,\xi,g)$ be a compact quasi-regular Sasakian $5$-manifold
and $Z$ denote the space of leaves of the characteristic foliation which is a
normal projective orbifold surface of $\frac{1}{r}(1,a)$-type singularities.
Let $\phi:\overline{Z}\rightarrow Z$ be a resolution of singularities of $Z$
for a nonsigular projective variety $\overline{Z}$ \ Then there exist a
lifting $\psi:\overline{M}\rightarrow M$ and an $\overline{\pi}:\overline
{M}^{5}\rightarrow\overline{Z}$ the regular $S^{1}$-principal bundle over
$\overline{Z}$ such that the following diagram is commutative :%
\begin{equation}%
\begin{array}
[c]{lcl}%
\overline{M}^{5}\supset V_{i} & \overset{\psi}{\longrightarrow} & M^{5}%
\supset\mathbf{S}_{p}^{1}\\
\downarrow\overline{\pi} & \circlearrowright & \downarrow\pi\\
\overline{Z}\supset E_{i} & \overset{\phi}{\longrightarrow} & Z\ni\pi(p),
\end{array}
\label{3}%
\end{equation}
where the irreducible transverse exceptional divisors $V_{i}\in Exc(\psi)$ and
the irreducible exceptional divisors $E_{i}\in Exc(\phi).$ $\psi$ is defined
to be the resolution of $\frac{1}{r}(1,a)$-type foliation singularities from
$\overline{M}$ to$M.$ Moreover, the Hirzebruch Jung continued fraction
\[
\frac{r}{a}=[b_{1},\cdot\cdot\cdot,b_{l}]
\]
gives the information on the foliation resolution $\psi:\overline
{M}\rightarrow M$ of $M$. The exceptional foliation $(-b_{i})$-curves form a
chain of $\{V_{1},\cdot\cdot\cdot,V_{l}\}$ such that each $V_{i}$ has self
intersection
\[
V_{i}^{2}=-b_{i}%
\]
for every $i=1,...,l$ and $V_{i}$ intersects another foliation curve $V_{j}$
transversely only if $j=i-1$ or $j=i+1.$ In particular, for a foliation cyclic
quotient singularity of type $\frac{1}{k}(1,1),$ we have $\frac{k}{1}=[k]:$
the foliation $(-k)$-curve.
\end{theorem}

\begin{proof}
Note that, for the quasi-regular Sasakian structure on $M$ with
\textbf{foliation singularities of type }$\frac{1}{r}(1,a)$ in which the
well-formed leave (foliation) space $Z$ has orbifold singularities as same as
the algebraic singularities
\[
\pi^{\ast}(K_{Z})=K_{M}^{T};\text{ \ \ }\overline{\pi}^{\ast}(K_{\overline{Z}%
})=K_{\overline{M}}^{T}%
\]
and
\[
\psi^{\ast}\circ\pi^{\ast}=\overline{\pi}^{\ast}\circ\phi^{\ast}.
\]
It follows from (\ref{a}) that
\[
\overline{\pi}^{\ast}(K_{\overline{Z}})=\overline{\pi}^{\ast}(\phi^{\ast
}(K_{Z}))+\sum_{i}a_{i}\overline{\pi}^{\ast}([E_{i}]),
\]
where the sum is over the irreducible exceptional divisors $E_{i}\in
Exc(\phi)\subset\overline{Z}$, $\pi(p)\in Z$ such that%
\[
\overline{\pi}(V_{i})=E_{i}%
\]
with $\mathbf{S}_{p}^{1}\subset M,$\quad$p\in M.$

\textbf{(i) Foliation singularity of type }$\frac{1}{k}(1,1):$ It is the most
simple case. For a resolution of singularities of type $\frac{1}{k}(1,1) $ in
$Z$
\[
\phi:\overline{Z}\rightarrow Z
\]
with a nonsigular projective surface $\overline{Z}$ and $E$ is the exceptional
curve of such resolution. Then, over the singular point $\pi(\mathbf{S}%
_{p}^{1})$, the exceptional $(-k)$-curve $E$ has self intersection
\[
E^{2}=-k.
\]
Then, from the previous section construction for $\mathcal{S}(-k)$ in
$\overline{M}^{5}$ as in subsection $3.3,$ it follows that there exists a
lifting $\psi:\overline{M}\rightarrow M$ as in Theorem \ref{D31}. Thus there
is an $\overline{\pi}:\overline{M}^{5}\rightarrow\overline{Z}$ the regular
$S^{1}$-principal bundle over $\overline{Z}$ such that the following diagram
is commutative
\begin{equation}%
\begin{array}
[c]{ccccc}%
\overline{M}^{5}\supset\mathcal{S}(-k)\thickapprox\mathbf{C}\times
\mathbf{S}^{3} & \overset{\pi_{2}}{\rightarrow} & \mathbf{S}^{3}\thickapprox
V & \overset{\psi|_{V}}{\longrightarrow} & \mathbf{S}_{p}^{1}\in M\\
\downarrow &  &  &  & \downarrow\pi_{M}\\
\overline{Z}\supset O(-k) & \overset{\pi_{d}}{\rightarrow} & \mathbf{CP}%
^{1}\thickapprox E & \overset{\phi|_{E}}{\longrightarrow} & \pi_{M}(p)\in Z.
\end{array}
\label{777}%
\end{equation}
Hence the exceptional foliation $(-k)$-curve $V$ has self intersection
\[
V^{2}=-k.
\]

\textbf{(ii) Foliation singularity of type }$\frac{1}{r}(1,a):$ The similar
situation as in (i). Let
\[
\phi:\overline{Z}\rightarrow Z
\]
be a resolution of singularities of type $\frac{1}{r}(1,a)$ in $Z$ with a
nonsigular projective surface $\overline{Z}$ and $E_{i}$ be the exceptional
curves of such resolution. Then it follows from M. Reid (\cite{r}) thar the
Hirzebruch Jung continued fraction
\[
\frac{r}{a}=[b_{1},\cdot\cdot\cdot,b_{l}]
\]
gives the information on the resolution $\phi:\overline{Z}\rightarrow Z$. More
precisely, over the singular point $\pi(\mathbf{S}_{p}^{1})$, the exceptional
curves form a chain of $\{E_{1},\cdot\cdot\cdot,E_{l}\}$ such that each
$E_{i}$ has self intersection
\[
E_{i}^{2}=-b_{i}%
\]
for every $i=1,...,l$ and the $(-b_{i})$-curve $E_{i}$ intersects another
$(-b_{j})$-curve $E_{j}$ transversely only if $j=i-1$ or $j=i+1.$

Again we consider the resolution of foliation singularities of type $\frac
{1}{r}(1,a)$
\[
\psi:\overline{M}\rightarrow M
\]
in $M$ with a regular Sasakian $5$-manifold $\overline{M}$. Let $V_{i}$ be the
exceptional foliation curves of $\psi$ lifting from $E_{i}$ via the diagram
(\ref{777}) such that
\[
\pi(V_{i})=E_{i}%
\]
Then, \textbf{by the previous construction (\ref{2022b}) of foliation
singularity of type }$\frac{1}{r}(1,a)$\textbf{\ over the singular fiber
}$\mathbf{S}_{p}^{1}$ as in subsection $3.4.$, the exceptional foliation
curves form a chain of $\{V_{1},\cdot\cdot\cdot,V_{l}\}$ such that each
$V_{i}$ has self intersection
\[
V_{i}^{2}=-b_{i}%
\]
for every $i=1,...,l$ and the foliation $(-b_{i})$-curve $V_{i}$ intersects
another foliation $(-b_{j})$-curve $V_{j}$ transversely only if $j=i-1$ or
$j=i+1.$ We refer to (\ref{A1}), (\ref{A2}) and \cite[$(6.3)$]{chlw} for some details.
\end{proof}

Here we come out some definition on the $\frac{1}{r}(1,a)$-type foliation
singularities of $M:$

\begin{definition}
Let $(M,\xi,g)$ be a compact quasi-regular Sasakian $5$-manifold and $Z$
denote the space of leaves of the characteristic foliation which is a normal
projective orbifold surface of $\frac{1}{r}(1,a)$-type singularities. Let
$\psi:\overline{M}\rightarrow M$ be a minimal resolution of $\frac{1}{r}(1,a)
$-type foliation singularities from $\overline{M}$ to $M.$ If
\begin{equation}
K_{\overline{M}}^{T}=\psi^{\ast}(K_{M}^{T})+\sum_{i}a_{i}[V_{i}]_{B},\label{a}%
\end{equation}
where the sum is over the irreducible transverse exceptional divisors
$V_{i}\in Exc(\psi)$ over $\mathbf{S}_{p}^{1}\subset M$ and the $a_{i}$ are
rational numbers, called the discrepancies. Then the $\frac{1}{r}(1,a)$-type
foliation singularities of $M$ are called :

\begin{enumerate}
\item terminal if $a_{i}>0$ for all $i.$

\item canonical if $a_{i}\geq0$ for all $i.$

\item log terminal if $a_{i}>-1$ for all $i.$

\item log canonical if $a_{i}\geq-1$ for all $i$.
\end{enumerate}
\end{definition}

\section{The Sasaki-Ricci Flow Through Singularities}

In this section, we warm up some basic facts of the Sasaki analogue of the
Kaehler-Ricci flow through singularities due to Song-Tian (\cite{st}).

\subsection{The Sasaki-Ricci Flow}

By a $\partial_{B}\overline{\partial}_{B}$-Lemma (\cite{eka}]) in the basic
Hodge decomposition, there is a basic function $F:M\rightarrow%
\mathbb{R}
$ such that
\[
\rho^{T}(x,t)-\varkappa d\eta(x,t)=d_{B}d_{B}^{c}F=i\partial_{B}%
\overline{\partial}_{B}F.
\]

We focus on finding a new $\eta$-Einstein Sasakian structure $(M,\xi
,\widetilde{\eta},\widetilde{\Phi},\widetilde{g})$ with
\[
\widetilde{\eta}=\eta+d_{B}^{c}\varphi,\varphi\in\Omega_{B}^{0}%
\]
and
\[
\widetilde{g}^{T}=(g_{i\overline{j}}^{T}+\varphi_{i\overline{j}}%
)dz^{i}d\overline{z}^{j}=2i(h_{i\overline{j}}+\frac{1}{2}\varphi
_{i\overline{j}})dz^{i}\wedge d\overline{z}^{j}%
\]
such that
\[
\widetilde{\rho}^{T}=\varkappa d\widetilde{\eta}.
\]
Hence
\[
\widetilde{\rho}^{T}-\rho^{T}=\kappa d_{B}d_{B}^{c}\varphi-d_{B}d_{B}^{c}F
\]
and it follows \ that
\begin{equation}
\frac{\det(g_{\alpha\overline{\beta}}^{T}+\varphi_{\alpha\overline{\beta}}%
)}{\det(g_{\alpha\overline{\beta}}^{T})}=e^{-\kappa\varphi+F}.\label{B}%
\end{equation}

This is a Sasakian analogue of the Monge-Ampere equation for the orbifold
version of Calabi-Yau Theorem ((\cite{eka}]).

Now we consider the Sasaki-Ricci flow on $M\times\lbrack0,T)$%
\begin{equation}
\frac{d}{dt}g^{T}(x,t)=-(Ric^{T}(x,t)-\varkappa g^{T}(x,t))\label{2021}%
\end{equation}
or
\[
\frac{d}{dt}d\eta(x,t)=-(\rho^{T}(x,t)-\varkappa d\eta(x,t)).
\]
It is equivalent to consider%
\begin{equation}
\frac{d}{dt}\varphi=\log\det(g_{\alpha\overline{\beta}}^{T}+\varphi
_{\alpha\overline{\beta}})-\log\det(g_{\alpha\overline{\beta}}^{T}%
)+\kappa\varphi-F.\label{C}%
\end{equation}
Note that,for any two Saskian structures with the fixed Reeb vector field
$\xi,$ we have
\[
Vol(M,g)=Vol(M,g^{\prime})
\]
and
\[
\widetilde{\omega}^{n}\wedge\eta=i^{n}\det(g_{\alpha\overline{\beta}}%
^{T}+\varphi_{\alpha\overline{\beta}})dz^{1}\wedge d\overline{z}^{1}%
\wedge...\wedge dz^{n}\wedge d\overline{z}^{n}\wedge dx.
\]

With all the above discussiones, let $(M,\eta,\xi,\Phi,g)$ be a compact
quasi-regular Sasakian manifold of dimension $2n+1$ and $Z$ denote the space
of leaves of the characteristic foliation $\mathcal{F}_{\xi}$. There is an
orbifold Riemannian submersion, and a principal $S^{1}$-orbibundle
($V$-bundle) $\pi:(M,g,\omega)\rightarrow(Z,h,\omega_{h})$ with $\omega
=\pi^{\ast}(\omega_{h}).$ Now for the natual projection
\begin{equation}
\Pi:(C(M),\overline{g},J,\overline{\omega})\rightarrow(Z,h,\omega
_{h})\label{2022-b}%
\end{equation}
with $\Pi|_{(M,g,\omega)}=\pi,$ then we have the volume form of the Kaehler
cone metric on the cone $C(M):$
\[
\overline{\omega}^{n+1}=r^{2n+1}(\Pi^{\ast}\omega_{h})^{n}\wedge
dr\wedge\overline{\eta}%
\]
and the volume form of the Sasaki metric on $M:$%
\begin{equation}
i_{\frac{\partial}{\partial r}}\overline{\omega}^{n+1}=(\Pi^{\ast}\omega
_{h})^{n}\wedge\eta.\label{2022-a}%
\end{equation}

\subsection{Basic Cohomology Characterization of the Maximal-Time Solution}

Let $(M,\xi_{0},\eta_{0},\Phi_{0},g_{0},\omega_{0})$ be a compact
quasi-regular Sasakian $(2n+1)$-manifold and its leave space $Z$ of the
characteristic foliation be \textbf{well-formed. }We consider a solution
$\omega=\omega(t)$ of the Sasaki-Ricci flow%
\begin{equation}%
\begin{array}
[c]{c}%
\frac{\partial}{\partial t}\omega(t)=-\mathrm{Ric}^{T}(\omega(t)),\text{
}\omega(0)=\omega_{0}.
\end{array}
\label{1}%
\end{equation}
As long as the solution exists, the cohomology class $[\omega(t)]_{B}$ evolves
by%
\[%
\begin{array}
[c]{c}%
\frac{\partial}{\partial t}\left[  \omega(t)\right]  _{B}=-c_{1}^{B}(M),\text{
}\left[  \omega(0)\right]  _{B}=\left[  \omega_{0}\right]  _{B},
\end{array}
\]
and solving this ordinary differential equation gives%
\[%
\begin{array}
[c]{c}%
\left[  \omega(t)\right]  _{B}=\left[  \omega_{0}\right]  _{B}-tc_{1}^{B}(M).
\end{array}
\]
We see that a necessary condition for the Sasaki-Ricci flow to exist for $t>0
$ such that%
\[
\left[  \omega_{0}\right]  _{B}-tc_{1}^{B}(M)>0.
\]
This necessary condition is in fact sufficient. In fact we define
\[%
\begin{array}
[c]{c}%
T_{0}:=\sup\{t>0|\text{ }\left[  \omega_{0}\right]  _{B}-tc_{1}^{B}(M)>0\}.
\end{array}
\]
That is to say that
\begin{equation}
\left[  \omega_{0}\right]  _{B}-T_{0}c_{1}^{B}(M)\in\overline{C_{M}^{B}%
}\label{E}%
\end{equation}
which is a nef class.

For a representative $\chi\in-C_{1}^{B}(M)$, we can fix the adapted measure
$\Omega$ on the leave space $Z$ \ and then a volume form $\Omega\wedge\eta
_{0}$ on $(M,\xi_{0},\eta_{0},\Phi_{0},g_{0},\omega_{0})$ such that
\begin{equation}
\Omega\wedge\eta_{0}=(\sqrt{-1})^{n}F(z_{1},...,z_{n})dz_{1}\wedge
d\overline{z}_{1}\wedge...\wedge dz_{n}\wedge d\overline{z}_{n}\wedge
dx\label{F}%
\end{equation}
with
\[
\sqrt{-1}\partial_{B}\overline{\partial}_{B}\log F=-Ric^{T}(\Omega)=\chi
\]
and \
\[
\int_{M}\Omega\wedge\eta_{0}=\int_{M}\omega_{0}^{n}\wedge\eta_{0}.
\]
We choose a reference (transverse) Kaehler metric%
\[
\widehat{\omega}_{t}:=\omega_{0}+t\chi.
\]
Then the corresponding transverse parabolic Monge-Ampere equation for the
basic function $\varphi(x,t)$ to (\ref{1}) on $M\times\lbrack0,T_{0})$ is%
\begin{equation}
\left\{
\begin{array}
[c]{rll}%
\frac{\partial}{\partial t}\varphi(x,t) & = & \log\frac{(\widehat{\omega}%
_{t}+\sqrt{-1}\partial_{B}\overline{\partial}_{B}\varphi)^{n}\wedge\eta_{0}%
}{\Omega\wedge\eta_{0}},\\
\widehat{\omega}_{t} & = & \omega_{0}+t\chi.\\
\sqrt{-1}\partial_{B}\overline{\partial}_{B}\log\Omega & = & \chi,\\
\widehat{\omega}_{t}+\sqrt{-1}\partial_{B}\overline{\partial}_{B}\varphi & > &
0,\\
\varphi(0) & = & 0.
\end{array}
\right. \label{2}%
\end{equation}

Based on \cite{sw1}, \cite{t}, \cite{chlw} and references therein as in the
Kaehler case, we have the following cohomological characterization for the
maximal solution of the Sasaki-Ricci flow :

\begin{theorem}
There exists a unique maximal solution $\omega(t)$ of the Sasaki-Ricci flow
(\ref{1}) on $M\times\lbrack0,T_{0})$ for $t\in\lbrack0,T_{o})$.
\end{theorem}

\subsection{The Weak Sasaki--Ricci Flow%
}

Let $(M,\eta,\xi,\Phi,g)$ be a compact quasi-regular Sasakian $5$-manifold
with finite cyclic quotient foliation singularities of type $\frac{1}%
{r}(1,a).$ Such singularities are rather mild and they do not become worse
after divisorial contractions are performed in the foliation minimal model
program we considered. More precisely, $\overline{M}\rightarrow M$ be a
resolution of foliation singularity, then the pullback of any volume measure
$\Omega\wedge\eta$ on $M$ is $L^{p}$-integrable on the nonsingular model
$\overline{M}$ for some $p>1$.

\begin{definition}
Let $(M,\eta,\xi,\Phi,g)$ be a compact quasi-regular Sasakian $5$-manifold
with finite cyclic quotient foliation singularities of type $\frac{1}{r}(1,a)$
and $H^{T}$ be a big and semi-ample $Q$-divisor with a basic transverse
birationalmorphism $\Psi_{mH^{T}}:M\rightarrow\mathbf{CP}^{N_{m}}$ induced by
the linear system $|mH^{T}|$ for some $m>>0$. Let $\overline{\omega}=\frac
{1}{m}\Psi^{\ast}(\omega_{FS})\in\lbrack H^{T}]$ and an volume measure
$\Omega\wedge\eta$ on $M$, where $\omega_{FS}$ is the Fubini-Study metric on
$\mathbf{CP}^{N_{m}}$ and define for $p\in(1,\infty]$, $\varphi$ is the basic
function%
\[
\mathcal{PSH}_{p}(M,\overline{\omega},\Omega\wedge\eta)=\{\varphi
\in\mathcal{PSH}(M,\overline{\omega})\cap L^{\infty}|\frac{(\overline{\omega
}+i\partial_{B}\overline{\partial}_{B}\varphi)^{n}}{\Omega\wedge\eta}\in
L^{p}(M,\Omega\wedge\eta)\}
\]
and
\[
K_{H_{T_{0},,}^{T}p}(M)=\{\overline{\omega}+i\partial_{B}\overline{\partial
}_{B}\varphi|\varphi\in\mathcal{PSH}_{p}(M,\overline{\omega},\Omega\wedge
\eta)\}.
\]

\end{definition}

Now we can define the weak Sasaki-Ricci flow on a compact quasi-regular
Sasakian $5$-manifold for which its leave space $Z$ is normal projective
surface with mild singularities.

\begin{definition}
Let $(M,\xi,\eta_{0})$ be a compact quasi-regular Sasakian $5$-manifold and
its leave space \ $Z$ be normal projective variety with finite cyclic quotient
singularities of type $\frac{1}{r}(1,a)$ which are klt singularities and
$H^{T}$ be a big and semi-ample $Q$-divisor so that $H^{T}+tK_{M}^{T}$ is
ample for a small $t\in Q^{+}$ such that
\[
T_{0}=\sup\{t>0\mathbf{\ |\ }H^{T}+tK_{M}^{T}\text{ \textrm{nef }}\}>0.
\]
A family of closed basic semi-positive $(1,1)$-currents $\omega(t,\cdot)$ on
$M$ for $t\in\lbrack0,T_{0})$ are said to be a solution of the weak
Sasaki--Ricci flow starting $\omega_{0}\in K_{H_{,,}^{T}p}(M)$ for some $p>1 $
if the following conditions hold

\begin{enumerate}
\item $\omega\in C^{\infty}((0,T_{0})\times M_{\operatorname{reg}})$ and
$\varphi\in L^{\infty}([0,T]\times M),$

\item $\omega(t,\cdot)$ satisfies
\begin{equation}
\left\{
\begin{array}
[c]{lcc}%
\frac{\partial}{\partial t}\omega=-Ric^{T}(\omega) & \mathrm{on} &
(0,T_{0})\times M_{\operatorname{reg}},\\
\omega(0)=\omega_{0} & \mathrm{on} & M.
\end{array}
\right. \label{1a}%
\end{equation}

\end{enumerate}
\end{definition}

Let $(M,\xi,\eta_{0})$ be a compact quasi-regular Sasakian $5$-manifold and
the space of leaves $Z$ be normal projective surface with klt singularities
and $H^{T}$ be a big and semi-ample $Q$-divisor. There exists a basic
transverse holomorphic map $\Psi:M\rightarrow(\mathbf{CP}^{N},\omega_{FS})$
defined by the basic transverse holomorphic section $\{s_{0},s_{1},...s_{N}\}$
of $H^{0}(M.(K_{M}^{T})^{m})$ which is $S^{1}$-equivariant with respect to the
weighted $\mathbf{C}^{\ast}$action in $\mathbf{C}^{N+1}$ with $N=\dim
H^{0}(M.(K_{M}^{T})^{m})-1$ for a large positive integer $m$ and
$\widehat{\omega}_{\infty}=\frac{1}{m}\Psi^{\ast}(\omega_{FS})\in\lbrack
H_{0}^{T}]$ and an adapted measure $\Omega_{Z}$ on $Z$ as in \cite[Theorem
4.3]{st} with $\Omega=\pi^{\ast}(\Omega_{Z})$, where $\omega_{FS}$ is the
Fubini-Study metric on $\mathbf{CP}^{N_{m}}.$ We also denote
\[
\chi=\sqrt{-1}\partial_{B}\overline{\partial}_{B}\log\Omega\in-c_{1}^{B}(M)
\]
and
\[
\widehat{\omega}_{t}=\widehat{\omega}_{\infty}+t\chi.
\]

The Sasaki-Ricci flow is equivalent to the following Monge-Ampere flow for the
basic function $\varphi(x,t)$ :
\begin{equation}
\left\{
\begin{array}
[c]{lll}%
\frac{\partial}{\partial t}\varphi(x,t) & = & \log\frac{(\widehat{\omega}%
_{t}+\sqrt{-1}\partial_{B}\overline{\partial}_{B}\varphi)^{n}\wedge\eta_{0}%
}{\Omega\wedge\eta_{0}},\\
\widehat{\omega}_{t}+\sqrt{-1}\partial_{B}\overline{\partial}_{B}\varphi & > &
0,\\
\varphi(0) & = & \varphi_{0}.
\end{array}
\right. \label{1b}%
\end{equation}

In order to define the Monge-Ampere flow (\ref{1b}) on $M,$ one might want to
lift the equation to the regular Sasakian manifold $\overline{M}$ and its
foliation space $\overline{Z}$ is a smooth projective surfaces. However,
$\widehat{\omega}_{Z}$ is not Kaehler on $\overline{Z}$ and $\Omega_{Z}$ in
general vanishes or blow up along $E_{i}\in Exc(f)$, so $\widehat{\omega}$ is
not transverse Kaehler on $\overline{M}$ and $\Omega$ in general vanishes or
blow up along $V_{i}\in Exc(\overline{f}).$ Then the lifted equation is
degenerate near the exceptional locus\ $Exc(\overline{f})$. we have to perturb
the Monge--Ampere equation (\ref{1b}) to (\ref{1d}) and obtain uniform
estimates so that the equation descend to $M$.

\begin{theorem}
\label{T51}Let $(M,\xi,\eta_{0})$ be a compact quasi-regular Sasakian
$5$-manifold and its leave space \ $Z$ be normal projective variety with
finite cyclic quotient singularities of type $\frac{1}{r}(1,a)$ which are klt
singularities and $\varphi_{0}\in\mathcal{PSH}_{p}(M,\widehat{\omega},\Omega)$
for some $p>1$. Define the minimal resolution of foliation singularities of
$M$ to be $\overline{f}:\overline{M}\rightarrow M$. Then the Monge-Ampere flow
on $\overline{M}$ defined by%
\begin{equation}
\left\{
\begin{array}
[c]{lll}%
\frac{\partial}{\partial t}\overline{\varphi}(x,t) & = & \log\frac
{(\overline{f}^{\ast}\widehat{\omega}_{t}+\sqrt{-1}\partial_{B}\overline
{\partial}_{B}\overline{\varphi})^{n}\wedge\eta_{0}}{\overline{f}^{\ast}%
\Omega\wedge\eta_{0}},\\
\overline{f}^{\ast}\widehat{\omega}_{t}+\sqrt{-1}\partial_{B}\overline
{\partial}_{B}\overline{\varphi} & > & 0,\\
\overline{\varphi}(0) & = & \overline{f}^{\ast}\varphi_{0}.
\end{array}
\right. \label{1d}%
\end{equation}
has a unique solution
\[
\overline{\varphi}\in C^{\infty}((0,T_{0})\times\overline{M}\backslash V)\cap
C^{0}([0,T_{0})\times\overline{M}\backslash V)
\]
such that
\[
\overline{\varphi}\in L^{\infty}(\overline{M})\cap\mathcal{PSH}(\overline
{M},\overline{f}^{\ast}\omega_{t}),
\]
for all $t\in\lbrack0,T_{0}).$ Moreover, $\overline{\varphi}$ is constant
along each fibre of $\overline{f}$, and so $\overline{\varphi}$ descends to a
unique solution
\[
\varphi\in C^{\infty}((0,T_{0})\times M_{\operatorname{reg}})\cap
C^{0}([0,T_{0})\times M_{\operatorname{reg}})
\]
of the Monge--Ampere flow (\ref{1b}) such that
\[
\varphi\in C^{0}(M)\cap\mathcal{PSH}(M,\omega_{t})
\]
for each $t\in\lbrack0,T_{0}).$
\end{theorem}

\begin{proof}
As notions in Proposition \ref{P21}, for a transverse Kaehler metric
$\omega_{Z}$ on the foliation space $Z=M/\mathcal{F}_{\xi}$ such that

(i) $\omega_{Z}\in C^{\infty}((0,T_{0})\times Z_{\text{\textrm{reg}}})$ and
$\varphi\in L^{\infty}([0,T]\times Z),$

(ii) $\omega_{Z}(t,\cdot)$ satisfies%
\begin{equation}
\left\{
\begin{array}
[c]{lcl}%
\frac{\partial}{\partial t}\omega_{Z}=-Ric(\omega_{Z}) & \mathrm{on} &
(0,T_{0})\times Z_{\operatorname{reg}},\\
\pi^{\ast}(\omega_{Z}(0))=\omega_{0} & \mathrm{on} & Z.
\end{array}
\right. \label{1c}%
\end{equation}
It follows from Proposition \ref{P21} that
\[
\pi:(M,\xi,\eta,g)\rightarrow(Z,h,\omega_{h})
\]
is an orbifold Riemannian submersion and a principal $S^{1}$-orbibundle
($V$-bundle) over $Z$ with
\[
\frac{1}{2}d\eta=\pi^{\ast}(\omega).
\]
Then the basic transverse Kaehler form
\[
\omega=\pi^{\ast}(\omega_{Z})
\]
satisfies the flow (\ref{1a}).

Let $f:\overline{Z}\rightarrow Z$ be a minimal resolution of singularities of
$Z$ for a nonsigular projective variety $\overline{Z}$ and $\overline{\pi
}:\overline{M}\rightarrow\overline{Z}$ is the regular $S^{1}$-principal bundle
over $\overline{Z}$ and regular Sasakian manifold $\overline{M}.$ The minimal
resolution of foliation singularities of $M$ to be $\overline{f}:\overline
{M}\rightarrow M$ such that the following diagram is commutative :%
\begin{equation}%
\begin{array}
[c]{ccc}%
\overline{M}\supset V_{i} & \overset{\overline{f}}{\longrightarrow} & M\supset
S_{p}^{1}\\
\downarrow\overline{\pi} & \circlearrowright & \downarrow\pi\\
\overline{Z}\supset E_{i} & \overset{f}{\longrightarrow} & Z\ni\pi(p).
\end{array}
\label{1e}%
\end{equation}
Then we have
\[
K_{\overline{M}}^{T}=\overline{f}^{\ast}(K_{M}^{T})+\sum_{i}a_{i}V_{i},
\]
where $V_{i}\in Exc(\overline{f})\subset\overline{M}$ over $\mathbf{S}_{p}%
^{1}\subset M$ and $a_{i}>-1$ for all $i.$

As in \cite[Theorem 4.2]{st}, they lifted (\ref{1c}) to the minimal resolution
$\overline{Z}$ of singularities of $Z$ to have the desired estimates. Here we
have to lift the Monge--Ampere equation (\ref{1b}) to (\ref{1d}). Thus we
obtain the estimates via the diagram (\ref{1e}).
\end{proof}

As a consequence, we have the following existence theorem for the solution of
the weak Sasaki-Ricci flow \ (\ref{1a}).

\begin{corollary}
\label{C51}(\cite[Theorem 4.3]{st}) Let $(M,\xi,\eta_{0})$ be a compact
quasi-regular Sasakian $5$-manifold and its leave space $Z$ be normal
projective surface with finite cyclic quotient orbifold singularities of type
$\frac{1}{r}(1,a)$ which are klt singularities. Let $H^{T}$ be an ample
$Q$-divisor such that
\[
T_{0}=\sup\{t>0\mathbf{\ |\ }H^{T}+tK_{M}^{T}\text{ \textrm{nef }}\}>0.
\]
If $\omega_{0}\in K_{H_{,,}^{T}p}(M)$ \ for some $p>1$, then there exists a
unique solution $\omega(t,\cdot)$ of the weak Sasaki--Ricci flow (\ref{1a})
for $t\in\lbrack0,T_{0})$. Moreover, $\omega(t,\cdot)$ is a smooth transverse
orbifold Kaehler-metric on $M$ for $t\in(0,T_{0})$ and so the weak
Sasaki-Ricci flow becomes the smooth orbifold Sasaki-Ricci flow on $M$
immediately when $t>0$.
\end{corollary}

\subsection{The Sasaki-Ricci Flow Through Divisorial Contractions}

Now based on the proofs of Theorem \ref{T51}, Corollary \ref{C51} and
\cite[Theorem 5.3]{st}, the weak Sasaki-Ricci flow (\ref{1a}) can be continued
through divisorial contractions on a compact quasi-regular Sasakian
$5$-manifold $M$.

\begin{proposition}
\label{P51}Let $(M,\xi,\eta_{0})$ be a compact quasi-regular Sasakian
$5$-manifold and its leave space $Z$ be normal projective surface with finite
cyclic quotient orbifold singularities of type $\frac{1}{r}(1,a)$ which are
klt singularities. Let $H^{T}$ be an ample $Q$-divisor such that
\[
T_{0}=\sup\{t>0\mathbf{\ |\ }H^{T}+tK_{M}^{T}\text{ \textrm{nef }}\}>0.
\]
is the first singular time and $L_{T_{0}}^{T}:=H^{T}+tK_{M}^{T}$ is the
semi-ample basic divisorwhich induces a divisoral contraction
\[
\psi_{(L_{T_{0}}^{T})^{m}}:M\rightarrow N\subset\mathcal{P(}H_{B}%
^{0}(M,(L_{T_{0}}^{T})^{m})),
\]
for some $m>>1.$ If $\omega_{0}\in K_{H^{T},p}(M)$ for some $p>1$, Let
$\omega(t,\cdot)$ be the unique solution $\omega(t,\cdot)$ of the weak
Sasaki--Ricci flow (\ref{1a}) for $t\in\lbrack0,T_{0})$If $\omega_{0}\in
K_{H^{T},p}(M)$ for some $p>1$, \ starting with $\omega_{0}\in K_{H^{T},p}(M)$
for some $p>1$. Then there esists $\omega_{N}(t,\cdot)$ such that

\begin{enumerate}
\item
\[
\omega_{N}\in K_{L_{T_{0}}^{T},p^{\prime}}(N)\cap C^{\infty}%
(N_{\operatorname{reg}}\backslash\psi(Exc(\psi))
\]
for some $p^{\prime}>1$.

\item
\[
\omega(t,\cdot)\rightarrow\psi^{\ast}\omega_{N}%
\]
in $C^{\infty}(M_{\operatorname{reg}}\backslash Exc(\psi))$-topology as
$t\rightarrow T_{0}^{-}.$

\item There exists a unique solution $\omega(t,\cdot)$ of the weak
Sasaki-Ricci flow on $N$ starting with $\omega_{N}$ at $T_{0}$ for $t\in
(T_{0},T_{N})$ with $T_{0}<T_{N}\leq\infty$, such that
\[
\omega(t,\cdot)\rightarrow\omega_{N}%
\]
in $C^{\infty}(M_{\operatorname{reg}}\backslash Exc(\psi))$-topology as
$t\rightarrow T_{0}^{+}.$
\end{enumerate}

Therefore the weak Sasaki-Ricci flow can be uniquely continued on $N$ starting
with $\omega_{N}$ at $T_{0}.$
\end{proposition}

In particular, we have Definition \ref{D12} for the floating foliation
canonical surgical contraction on a compact quasi-regular Sasakian $5$-manifold.

\section{Foliation MMP with Scaling via the Sasaki-Ricci Flow}

Since the Reeb vector field and the transverse holomorphic structure are both
invariant, all the quantities in this section are only involved with the
transverse K\"{a}hler structure $\omega(t)$ and basic tensors. Hence, under
the Sasaki-Ricci flow, when one applies for instance, the maximal principle
and foliation resolutions etc, the expressions involved behave essentially the
same as the K\"{a}hler-Ricci flow.

\subsection{Canonical Surgical Contraction for Floating Foliation
$(-1)$-Curves}

In this subsection, primarily along the lines of the arguments in \cite{sw2},
we are ready to give the proof of Theorem \ref{T61}. We write the Sasaki-Ricci
flow (\ref{1}) as a transverse parabolic complex Monge-Amp\'{e}re equation.
First, using the assumption (\ref{2021}), define a family of reference
transverse K\"{a}hler metrics $\widehat{\omega}_{t}$ for $t\in\lbrack0,T_{0})$
by%
\begin{equation}%
\begin{array}
[c]{c}%
\widehat{\omega}_{t}=\frac{1}{T_{0}}((T_{0}-t)\omega_{0}-t\psi^{\ast}%
\omega_{N})\in\lbrack\omega(t)]_{B}=[\omega_{0}]_{B}-tc_{1}^{B}(M).
\end{array}
\label{61}%
\end{equation}
We can fix the adapted measure $\Omega$ on the leave space $Z$ and then a
volume form $\Omega\wedge\eta_{0}$ on $M$ such that%
\begin{equation}%
\begin{array}
[c]{c}%
\sqrt{-1}\partial_{B}\overline{\partial}_{B}\log\Omega=\frac{\partial
}{\partial t}\widehat{\omega}_{t}=-\frac{1}{T_{0}}(\omega_{0}+\psi^{\ast
}\omega_{N})\in c_{1}^{B}(M),\text{ }\int_{M}\Omega\wedge\eta_{0}=1.
\end{array}
\label{62}%
\end{equation}
Then the corresponding transverse parabolic Monge-Amp\'{e}re equation for the
basic function $\varphi=\varphi(t)$ to (\ref{2}) on $M\times\lbrack0,T_{0})$
is%
\begin{equation}%
\begin{array}
[c]{c}%
\frac{\partial}{\partial t}\varphi=\log\frac{(\widehat{\omega}_{t}+\sqrt
{-1}\partial_{B}\overline{\partial}_{B}\varphi)^{2}\wedge\eta_{0}}%
{\Omega\wedge\eta_{0}},\text{ }\varphi(0)=0.
\end{array}
\label{63}%
\end{equation}

\subsubsection{Key Estimates for the Sasaki-Ricci Flow}

In the following, by using the Sasaki analogue of Koaira Lemma
(\cite[Proposition 3.3]{chlw}) under assumption (\ref{2022}) and (\ref{D}), we
prove the main estimates for the Sasaki-Ricci flow under the assumptions of
Theorem \ref{T61}. In particular, the estimates of the first two lemmas are
essentially contained in \cite{chlw}.

\begin{lemma}
\label{L61}There is a uniform constant $C$ depending only on $(M,\omega_{0}) $
such that the solution $\varphi=\varphi(t)$ of (\ref{63}) satisfies, for
$t\in\lbrack0,T_{0}),$%
\[
\left\Vert \varphi\right\Vert _{L^{\infty}}\leq C,\text{ }\varphi^{\prime
}=\partial\varphi/\partial t\leq C\text{ }\mathrm{and}\text{ }\omega^{2}%
\wedge\eta_{0}\leq C\Omega\wedge\eta_{0}.
\]
As $t\rightarrow T_{0}^{-}$, $\varphi(t)$ converges pointwise on $M$ to a
bounded basic function $\varphi_{T}$ satisfying%
\[%
\begin{array}
[c]{c}%
\omega_{T_{0}}:=\widehat{\omega}_{T_{0}}+\sqrt{-1}\partial_{B}\overline
{\partial}_{B}\varphi_{T_{0}}\geq0,
\end{array}
\]
and $\omega(t)$ converges weakly in the sense of currents to the closed
positive $(1,1)$ current $\omega_{T_{0}}$.
\end{lemma}

We have the $C^{\infty}$ estimates for the solution $\omega(t)$ of the
Sasaki-Ricci flow on compact subsets of $M\backslash V$ as following.

\begin{lemma}
\label{L62}With the assumptions of Theorem \ref{T61}, the solution
$\omega=\omega(t)$ of the Sasaki-Ricci flow (\ref{1}) satisfies

(i) There exists a uniform constant $c>0$ such that%
\begin{equation}%
\begin{array}
[c]{c}%
\omega\geq c\psi^{\ast}\omega_{N}.
\end{array}
\label{64}%
\end{equation}

(ii) For every compact subset $K\subset M\backslash V$, there exist constants
$C_{K,k}$ for $k=0,1,2,...,$ such that%
\begin{equation}%
\begin{array}
[c]{c}%
\left\Vert \omega(t)\right\Vert _{C^{k}(K,\omega_{0})}\leq C_{K,k}.
\end{array}
\label{65}%
\end{equation}

(iii) The closed $(1,1)$ current $\omega_{T_{0}}$, given in Lemma \ref{L61},
is a smooth transverse K\"{a}hler metric on $M\backslash V$.

(iv) As $t\rightarrow T_{0}^{-}$, the metrics $\omega(t)$ converge to
$\omega_{T_{0}}$ in $C^{\infty}$ on compact subsets of $M\backslash V$.
\end{lemma}

Then it follows from Lemma \ref{L62} that part $(1)$ in the Definition
\ref{D12} of foliation canonical surgical contraction holds.

We now compare $\psi^{\ast}\omega_{N}$ with a fixed transverse K\"{a}hler
metric $\omega_{0}$ on $M$ by using the description of the neighborhood of $V$
in the blow-up map
\[%
\begin{array}
[c]{c}%
\psi:M\rightarrow N.
\end{array}
\]
As the notion in section $2,$ for $p=(x_{i},v_{i},s_{i})\in U_{i}%
\subset\mathbb{C}\times\mathbb{C}\times S^{1},$ $i=1,2,$ the composition
\[%
\begin{array}
[c]{ccc}%
\mathcal{S}(-1)\thickapprox\mathbb{C}\times S^{3} & \overset{\pi
_{2}}{\rightarrow} & S^{3}\\
&  & \downarrow H\\
&  & CP^{1}%
\end{array}
\]
send $p$ to $[l_{1},l_{2}].$ we have $x_{1}=\frac{l_{2}}{l_{1}};$ $x_{2}%
=\frac{l_{1}}{l_{2}}$, and $v_{i},$ $i=1,2$ correspond to vectors in a fiber
of $\mathcal{O}(-1).$ Now $s$ is the section of $[V\mathcal{]}$ over
$\mathcal{S}(-1)$ given by%
\[%
\begin{array}
[c]{c}%
s_{i}:U_{i}\rightarrow\mathbb{C}\mathbf{,}\text{\ }s_{i}=v_{i}%
\end{array}
\]
and $h$ is the Hermitian metric on $[V]$ such that
\[%
\begin{array}
[c]{c}%
h_{i}=(\frac{|l_{1}|^{2}+|l_{2}|^{2}}{|l_{i}|^{2}})=(1+|x_{i}|^{2})
\end{array}
\]
on $U_{i}.$ Then we have
\begin{equation}%
\begin{array}
[c]{c}%
\mathrm{dist}^{2}(p,S_{0}^{3})=(\frac{|l_{1}|^{2}+|l_{2}|^{2}}{|l_{i}|^{2}%
})|v_{i}|^{2}:=|s|_{h}^{2}%
\end{array}
\label{A}%
\end{equation}
On the other hand, the distance square to the singular fiber $S_{0}^{1}$ is%
\begin{equation}%
\begin{array}
[c]{c}%
\mathrm{dist}^{2}(p,S_{0}^{1})=|z_{1}|^{2}+|z_{2}|^{2}:=r_{S^{1}}^{2}.
\end{array}
\label{AA}%
\end{equation}

If $\psi:\mathbb{C}^{1}\times S^{3}\rightarrow S^{1}\times\mathbb{C}^{2}$ is
the composition of
\[%
\begin{array}
[c]{ccccc}%
S(-k)\thickapprox\mathbb{C}\times S^{3} & \overset{\pi_{2}}{\rightarrow} &
S^{3} & \overset{}{\longrightarrow} & S^{1}\times\mathbb{C}^{2}\supset
S^{1}\times D^{4},
\end{array}
\]
it follows from (\ref{A}) and (\ref{AA}) that the function $|s|_{h}^{2}$ on
$\psi^{-1}(S^{1}\times D_{1/2}^{4})$ is given by
\begin{equation}%
\begin{array}
[c]{c}%
\psi^{\ast}r_{S^{1}}^{2}=|s|_{h}^{2}.
\end{array}
\label{AAA}%
\end{equation}
and then write
\[
|s|_{h}^{2}(x)=r_{S^{1}}^{2}=|z_{1}|^{2}+|z_{2}|^{2}%
\]
for $\psi(x)=(s,z_{1},z_{2}),$ $s\in S^{1}.$ Hence the transverse Ricci
curvature of $h$ is given by
\[%
\begin{array}
[c]{c}%
R_{h}^{T}=-\frac{\sqrt{-1}}{2\pi}\partial_{B}\overline{\partial}_{B}%
\log(|z_{1}|^{2}+|z_{2}|^{2})
\end{array}
\]
on $\psi^{-1}(S^{1}\times D_{1/2}^{4}\backslash\{0\}).$

We have the following lemmas.

\begin{lemma}
\label{L63}For sufficiently small $\varepsilon_{0}>0$,%
\[
\omega_{M}:=\psi^{\ast}\omega_{N}-\varepsilon_{0}R_{h}^{T}%
\]
is a transverse K\"{a}hler form on $M$. Furthermore, in term of Sasaki normal
coordinate (\ref{AAA3})
\begin{equation}%
\begin{array}
[c]{c}%
\omega_{M}:=\psi^{\ast}\omega_{N}+\frac{\sqrt{-1}}{2\pi}\frac{\varepsilon_{0}%
}{r_{S^{1}}^{2}}\sum_{i,j=1}^{2}(\delta_{ij}-\frac{\overline{z}_{i}z_{j}%
}{r_{S^{1}}^{2}})dz_{i}\wedge d\overline{z}_{j}%
\end{array}
\label{66}%
\end{equation}
on $\psi^{-1}(S^{1}\times D_{1/2}^{4}\backslash\{0\}).$
\end{lemma}

\begin{lemma}
\label{L64}There exist positive constants $C_{1},$ $C_{2}$ such that%
\begin{equation}%
\begin{array}
[c]{c}%
\psi^{\ast}\omega_{N}\leq\omega_{M}\leq\frac{C_{1}}{|s|_{h}^{2}}\psi^{\ast
}\omega_{N}%
\end{array}
\label{67}%
\end{equation}
and%
\begin{equation}%
\begin{array}
[c]{c}%
\frac{1}{C_{2}}\psi^{\ast}\omega_{N}\leq\omega_{0}\leq\frac{C_{2}}{|s|_{h}%
^{2}}\psi^{\ast}\omega_{N}.
\end{array}
\label{68}%
\end{equation}

\end{lemma}

\begin{proof}
We observe that $\omega_{M}$ and $\omega_{0}$ are uniformly equivalent
transverse K\"{a}hler metrics on $M$. Hence (\ref{68}) follows easily from
(\ref{67}). Thus it suffices for us to prove (\ref{67}) in $\psi^{-1}%
(S^{1}\times D_{1/2}^{4}\backslash\{0\})$. By using (\ref{66}), the first
inequality of (\ref{67}) follows from the fact that if $X^{i}$ is any
$T^{1,0}$ vector then by the Cauchy-Schwarz inequality,%
\[%
\begin{array}
[c]{c}%
\sum_{i,j=1}^{2}\frac{\overline{z}_{i}z_{j}}{r_{S^{1}}^{2}}X^{i}%
\overline{X^{j}}=\left(  \sum_{i=1}^{2}\frac{\overline{z}_{i}}{r_{S^{1}}}%
X^{i}\right)  \left(  \sum_{j=1}^{2}\frac{z_{j}}{r_{S^{1}}}\overline{X^{j}%
}\right)  \leq\left\vert X\right\vert ^{2}=\sum_{i,j=1}^{2}\delta_{ij}%
X^{i}\overline{X^{j}}.
\end{array}
\]
The second inequality of (\ref{67}) follows from the fact that $\overline
{z}_{i}z_{j}$ is semi-positive definite.
\end{proof}

\begin{lemma}
\label{L65}There exists $\delta>0$ and a uniform constant $C$ such that for
$\omega=\omega(t)$ a solution of the Sasaki-Ricci flow (\ref{1}) satisfies
\[%
\begin{array}
[c]{c}%
\omega\leq\frac{C}{|s|_{h}^{2}}\psi^{\ast}\omega_{N}\text{ \textrm{and}
}\omega\leq\frac{C}{|s|_{h}^{2(1-\delta)}}\omega_{0}.
\end{array}
\]

\end{lemma}

\begin{proof}
From Lemma \ref{L64} we have%
\begin{equation}%
\begin{array}
[c]{c}%
|s|_{h}^{2}\mathrm{tr}_{\psi^{\ast}\omega_{N}}\omega\leq C\mathrm{tr}%
_{\omega_{0}}\omega.
\end{array}
\label{71}%
\end{equation}
We fix $0<\varepsilon\leq1$ and will apply the maximum principle to the
quantity
\[%
\begin{array}
[c]{c}%
Q_{\varepsilon}=\log\mathrm{tr}_{\omega_{0}}\omega+A\log(|s|_{h}%
^{2+2\varepsilon}\mathrm{tr}_{\psi^{\ast}\omega_{N}}\omega)-A^{2}\varphi,
\end{array}
\]
on $M\backslash V$ where $A$ is a constant to be determined later. It follows
that at any fixed time $t$, $Q_{\varepsilon}(x,t)$ tends to negative infinity
as $x\in M$ tends to $V$. Suppose there exists $(x_{0},t_{0})\in M\backslash
V\times(0,T_{0})$ with $\sup_{M\backslash V\times(0,t_{0}]}Q_{\varepsilon
}=Q_{\varepsilon}(x_{0},t_{0})$. For any fixed transverse K\"{a}hler metric
$\widetilde{\omega}$ on $M$, if $\omega=\omega(t)$ solves the Sasaki-Ricci
flow, we have the following estimate like the K\"{a}hler-Ricci flow%
\begin{equation}%
\begin{array}
[c]{ll}
& (\frac{\partial}{\partial t}-\Delta_{B})\log\mathrm{tr}_{\widetilde{\omega}%
}\omega\\
= & -\frac{1}{\mathrm{tr}_{\widetilde{\omega}}\omega}\left(  g^{Ti\overline
{j}}\widetilde{R}_{i\overline{j}}^{T\text{ \ }k\overline{l}}g_{k\overline{l}%
}^{T}+g^{Ti\overline{j}}\widetilde{g}^{Tk\overline{l}}g^{Tp\overline{q}%
}\widetilde{\bigtriangledown}_{i}^{T}g_{k\overline{q}}^{T}%
\widetilde{\bigtriangledown}_{\overline{j}}^{T}g_{p\overline{l}}^{T}%
-\frac{\left\vert \bigtriangledown^{T}\mathrm{tr}_{\widetilde{\omega}}%
\omega\right\vert ^{2}}{\mathrm{tr}_{\widetilde{\omega}}\omega}\right) \\
\leq & \widetilde{C}\mathrm{tr}_{\omega}\widetilde{\omega},
\end{array}
\label{73}%
\end{equation}
for $\widetilde{C}$ depending only on the lower bound of the transverse
bisectional curvature of $\widetilde{g}^{T}.$ At $(x_{0},t_{0})$ we have
\begin{equation}%
\begin{array}
[c]{lll}%
0\leq(\frac{\partial}{\partial t}-\Delta_{B})Q_{\varepsilon} & \leq &
C_{1}\mathrm{tr}_{\omega}\omega_{0}-A\mathrm{tr}_{\omega}(A\widehat{\omega
}_{t_{0}}-(1+\varepsilon)R_{h}^{T}-C_{2}\psi^{\ast}\omega_{N})\\
&  & -A^{2}\log\frac{\omega^{2}\wedge\eta_{0}}{\Omega\wedge\eta_{0}}+16A,
\end{array}
\label{72}%
\end{equation}
for some uniform constants $C_{1},$ $C_{2}$. Here we apply (\ref{73}) with
$\widetilde{\omega}=\omega_{0}$ and then $\widetilde{\omega}=\psi^{\ast}%
\omega_{N}$ at the point $(x_{0},t_{0})$. From Lemma \ref{L63} and the
definition of $\widehat{\omega}_{t}$ we can choose $A$ sufficiently large and
independent of $\varepsilon$, $t_{0}$ so that
\[%
\begin{array}
[c]{c}%
A(A\widehat{\omega}_{t_{0}}-(1+\varepsilon)R_{h}^{T}-C_{2}\psi^{\ast}%
\omega_{N})\geq(C_{1}+1)\omega_{0}.
\end{array}
\]
Then we have
\[%
\begin{array}
[c]{l}%
(A^{2}\log\frac{\omega^{2}\wedge\eta_{0}}{\Omega\wedge\eta_{0}}+\mathrm{tr}%
_{\omega}\omega_{0})(x_{0},t_{0})\leq16A,
\end{array}
\]
which implies that%
\[%
\begin{array}
[c]{l}%
(\mathrm{tr}_{\omega}\omega_{0})(x_{0},t_{0})\leq C.
\end{array}
\]
Since the volume form $\omega^{2}\wedge\eta_{0}$ is uniformly bounded from
above by Lemma \ref{L61},
\[%
\begin{array}
[c]{c}%
(\mathrm{tr}_{\omega_{0}}\omega)(x_{0},t_{0})\leq(\mathrm{tr}_{\omega}%
\omega_{0})(x_{0},t_{0})\left(  \frac{\omega^{2}\wedge\eta_{0}}{\omega_{0}%
^{2}\wedge\eta_{0}}\right)  (x_{0},t_{0})\leq C.
\end{array}
\]
From (\ref{71}), we obtain%
\[%
\begin{array}
[c]{c}%
(|s|_{h}^{2}\mathrm{tr}_{\psi^{\ast}\omega_{N}}\omega)(x_{0},t_{0})\leq C.
\end{array}
\]
Since $\varphi$ is uniformly bounded, we can get%
\[
Q_{\varepsilon}\leq C,
\]
for $C$ independent of $\varepsilon$. Letting $\varepsilon\rightarrow0$ and
applying (\ref{71}) again we get
\[%
\begin{array}
[c]{c}%
\omega\leq\frac{C}{|s|_{h}^{2}}\psi^{\ast}\omega_{N}.
\end{array}
\]

Using the inequality $\mathrm{tr}_{\omega_{0}}\omega\leq C\mathrm{tr}%
_{\psi^{\ast}\omega_{N}}\omega$ in Lemma \ref{L64}, we also have
\[%
\begin{array}
[c]{l}%
\log(|s|_{h}^{2A}(\mathrm{tr}_{\omega_{0}}\omega)^{A+1})\leq C,
\end{array}
\]
this yields that $\mathrm{tr}_{\omega_{0}}\omega\leq\frac{C}{|s|_{h}%
^{2(1-\delta)}}$ for $\delta=\frac{1}{A+1}>0,$ and also $\omega\leq\frac
{C}{|s|_{h}^{2(1-\delta)}}\omega_{0}.$
\end{proof}

For $(s,z_{1},z_{2})\in S^{1}\times D^{4},$ we consider the transverse
holomorphic vector field
\[%
\begin{array}
[c]{c}%
\sum_{i=1}^{2}z_{i}\frac{\partial}{\partial z_{i}}%
\end{array}
\]
on $S^{1}\times D^{4},$ which defines a transverse holomorphic vector field
$X$ on $\psi^{-1}(S^{1}\times D^{4})\subset M$ via $\psi$ and we extend $X$ to
be a smooth $T^{1,0}$ vector field on $M$. We then have the following lemma.

\begin{lemma}
\label{L66}For $\omega=\omega(t)$ be a solution of the Sasaki-Ricci flow, we
have the estimate%
\begin{equation}%
\begin{array}
[c]{c}%
|X|_{\omega}^{2}\leq C|s|_{h}.
\end{array}
\label{77}%
\end{equation}
for a uniform constant $C$. Locally, in $S^{1}\times D_{1/2}^{4}%
\backslash\{0\}$ we have%
\begin{equation}%
\begin{array}
[c]{c}%
|W|_{g^{T}}^{2}\leq\frac{C}{r_{S^{1}}},
\end{array}
\label{78}%
\end{equation}
for $W=\sum_{i=1}^{2}\frac{1}{r_{S^{1}}}(x_{i}\frac{\partial}{\partial x_{i}%
}+y_{i}\frac{\partial}{\partial y_{i}})$ the unit transverse radial vector
field with respect to $g_{Eucl}^{T}$, where $z_{i}=x_{i}+\sqrt{-1}y_{i}$.
\end{lemma}

\begin{proof}
By the expression (\ref{66}) of $\omega_{M},$ we have
\[
|X|_{\omega_{M}}^{2}=|X|_{\psi^{\ast}\omega_{N}}^{2}%
\]
in $S^{1}\times D_{1/2}^{4}.$ It follows that $|X|_{\omega_{0}}^{2}$ is
uniformly equivalent to $|s|_{h}^{2}=r_{S^{1}}^{2}$ in $S^{1}\times
D_{1/2}^{4}$. Hence there exists a positive constant $C(t)$ depending on $t$
such that%
\begin{equation}%
\begin{array}
[c]{c}%
\frac{1}{C(t)}|s|_{h}^{2}\leq|X|_{\omega}^{2}\leq C(t)|s|_{h}^{2}.
\end{array}
\label{6b}%
\end{equation}
Now define a transverse K\"{a}hler metric $\widetilde{\omega}_{N}$ on $N$ by
$\widetilde{\omega}_{N}=\omega_{Eucl}$ on $S^{1}\times D^{4}$, and extending
in an arbitrary way to be a smooth transverse K\"{a}hler metric on $N$. For
small $\varepsilon>0$, we apply the maximum principle to the quantity%
\[%
\begin{array}
[c]{c}%
Q_{\varepsilon}=\log(|X|_{\omega}^{2+2\varepsilon}\mathrm{tr}_{\psi^{\ast
}\widetilde{\omega}_{N}}\omega)-t.
\end{array}
\]
For any fixed $t$, using (\ref{71}), (\ref{6b}) and $\widetilde{\omega}_{N}$
is uniformly equivalent to $\omega_{N}$, we know that $(|X|_{\omega
}^{2+2\varepsilon}\mathrm{tr}_{\psi^{\ast}\widetilde{\omega}_{N}}\omega)(x,t)$
tends to zero as $x$ tends to $V$ and thus $Q_{\varepsilon}(x,t)$ tends to
negative infinity. Since $Q_{\varepsilon}$ is uniformly bounded from above on
the complement of $S^{1}\times D_{1/2}^{4}\backslash\{0\}$, thus
$Q_{\varepsilon}$ may attain its maximum at a point in $S^{1}\times
D_{1/2}^{4}\backslash\{0\}$. So we assume that at some point $(x_{0},t_{0})\in
S^{1}\times D_{1/2}^{4}\backslash\{0\}\times(0,T)$ with $\sup_{(M\backslash
V)\times(0,t_{0}]}Q_{\varepsilon}=Q_{\varepsilon}(x_{0},t_{0})$. By
(\ref{73}), we have
\begin{equation}%
\begin{array}
[c]{c}%
(\frac{\partial}{\partial t}-\Delta_{B})\mathrm{tr}_{\psi^{\ast}%
\widetilde{\omega}_{N}}\omega\leq0,
\end{array}
\label{6e}%
\end{equation}
since the transverse bisectional curvature of $\widetilde{\omega}_{N}$ is
zero. Using the Cauchy-Schwarz inequality to get%
\begin{equation}%
\begin{array}
[c]{c}%
(\frac{\partial}{\partial t}-\Delta_{B})\log|X|_{\omega}^{2}=\frac
{1}{|X|_{\omega}^{2}}\left(  -g^{Ti\overline{j}}g_{k\overline{l}}^{T}%
(\partial_{i}^{k}X)(\overline{\partial_{j}^{l}X})+\frac{\left\vert
\bigtriangledown^{T}|X|_{\omega}^{2}\right\vert _{\omega}^{2}}{|X|_{\omega
}^{2}}\right)  \leq0.
\end{array}
\label{6f}%
\end{equation}
All these (\ref{6e}) and (\ref{6f}) yield that
\[%
\begin{array}
[c]{c}%
0\leq(\frac{\partial}{\partial t}-\Delta_{B})Q_{\varepsilon}(x_{0},t_{0})<0.
\end{array}
\]
This implies $Q_{\varepsilon}$ is uniformly bounded from above. Letting
$\varepsilon\rightarrow0$ we obtain%
\[
|X|_{\omega}^{2}\mathrm{tr}_{\psi^{\ast}\omega_{N}}\omega\leq C,
\]
for some uniform constant $C$. By Lemma \ref{L64} we have $|X|_{\omega}%
^{2}\mathrm{tr}_{\omega_{0}}\omega\leq C$, and since $|X|_{\omega}^{2}%
\leq|X|_{\omega_{0}}^{2}\mathrm{tr}_{\omega_{0}}\omega,$ this gives
$|X|_{\omega}^{4}\leq C|X|_{\omega_{0}}^{2}$ and (\ref{77}) follows from the
fact that $|X|_{\omega_{0}}^{2}$ is uniformly equivalent to $|s|_{h}^{2} $ in
$S^{1}\times D_{1/2}^{4}$.

Since $2\operatorname{Re}(\frac{1}{r_{S^{1}}}X)=W$ in $S^{1}\times D_{1/2}%
^{4}\backslash\{0\}$ and hence (\ref{78}) follows from (\ref{77}).
\end{proof}

From Lemma \ref{L66} we have the bound of the lengths of spherical and
transverse radial paths in $S^{1}\times D_{1/2}^{4}\backslash\{0\}.$ For
$0<r_{S^{1}}<1/2$, the diameter of the product sphere $S^{1}\times
S_{r_{S^{1}}}^{3}$ in $S^{1}\times D^{4}$ with the metric induced from
$\omega$ is uniformly bounded from above, independent of $r_{S^{1}}$. For any
$(s_{0},z_{1},z_{2})\in S^{1}\times D_{1/2}^{4}\backslash\{0\}$, the length of
a transverse radial path $\gamma(u)=(s_{0},uz_{1},uz_{2})$ for $u\in(0,1]$
with respect to $\omega$ is uniformly bounded from above by $Cr_{S^{1}}$ with
$r_{S^{1}}=\sqrt{|z_{1}|^{2}+|z_{2}|^{2}}$. Hence the diameter of $S^{1}\times
D_{1/2}^{4}\backslash\{0\}$ with respect to $\omega$ is uniformly bounded from
above and%
\begin{equation}
\mathrm{diam}_{g^{T}(t)}M\leq C.\label{2022-4}%
\end{equation}
We then obtain a diameter bound of $M$ with respect to $g^{T}(t)$.

\subsubsection{Gromov-Hausdorff Convergence as $t\rightarrow T_{0}^{-}$}

In this subsection, we will derive the part $(2)$ in the Definition \ref{D12}
of floating foliation canonical surgical contraction under the assumption
(\ref{2022}). For the simplicity, we assume that there is only one exceptional
foliation $(-1)$-curve $V$.

More precisely, we will show that $(M,g^{T}(t))$ converges to $(N,d_{T_{0}%
}^{T})$ in the sense of Gromov-Hausdorff as $t\rightarrow T_{0}^{-}$ and
$\omega(t)$ converges in $C^{\infty}$ on any compact subset of $M\backslash V$
to a smooth transverse K\"{a}hler metric $\omega_{T}$ on $M\backslash V$. We
will first define a compact metric space $(N,d_{T_{0}})$ and show that
$(M,g(t))$ converges in the Gromov-Hausdorff sense to this metric space.

Recall that $(M,\eta,\xi,\Phi,g)$ is a compact quasi-regular Sasakian
manifold, then by the%
first structure theorem on Sasakian manifolds, 
$M$\ is a principal $S^{1}$-orbibundle ($V$-bundle) over $Z$ which is also a
$Q$-factorial, polarized, normal projective orbifold such that there is an
orbifold Riemannian submersion $\pi_{M}:(M,g,\omega)\rightarrow(Z,h,\omega
_{h})$ with%
\[%
\begin{array}
[c]{c}%
g=g^{T}+\eta\otimes\eta
\end{array}
\]
such that
\[%
\begin{array}
[c]{c}%
g^{T}=\pi_{M}{}^{\ast}(h);\text{ \ }\frac{1}{2}d\eta=\pi_{M}{}^{\ast}%
(\omega_{h}).
\end{array}
\]
The orbit $\xi_{x}$ is compact for any $x\in M,$ we then define the transverse
distance function as
\[%
\begin{array}
[c]{c}%
d^{T}(x_{1},x_{2})\triangleq d_{g}(\xi_{x_{1}},\xi_{x_{2}}),
\end{array}
\]
where $d_{g}$ is the distance function defined by the Sasaki metric $g.$ Then
\[%
\begin{array}
[c]{c}%
d_{M}^{T}(x_{1},x_{2})=d_{h}(\pi_{M}(x_{1}),\pi_{M}(x_{2})).
\end{array}
\]
We define a transverse ball $B_{\xi,g}(x,r)$ as follows:
\[%
\begin{array}
[c]{c}%
B_{\xi,g}(x,r)=\left\{  \widetilde{x}:d_{M}^{T}(x,\widetilde{x})<r\right\}
=\left\{  y:d_{h}(\pi_{M}(x),\pi_{M}(\widetilde{x}))<r\right\}  .
\end{array}
\]

Since $M\backslash V$ can be identified with $N\backslash S_{y_{0}}^{1}$ via
the map $\psi$. Abusing notation, we denote $g_{T_{0}}$ for the smooth
K\"{a}hler metric $(\psi^{-1})^{\ast}g_{T_{0}}$ on $N\backslash S_{y_{0}}^{1}$
and extend $g_{T_{0}}$ to a nonnegative $(1,1$)-tensor $\overline{g}_{T_{0}}$
on the whole space $N$ by setting $\overline{g}_{T_{0}}|_{S_{y_{0}}^{1}}%
(\cdot,\cdot)=0.$ Since $N$ is a $S^{1}$-bundle over $Y,$ there is a
Riemannian submersion
\begin{equation}
\pi_{N}:(N,\overline{g}_{T_{0}},\omega_{T_{0}})\rightarrow(Y,k_{T_{0}}%
,\omega_{k})\label{2022-1}%
\end{equation}
such that
\[
\pi_{N}^{\ast}(\omega_{k})=\omega_{T_{0}}.
\]
We define a distance function $d_{T_{0}}^{T}$ on $N$ as follows: For
$\overline{g}_{T_{0}}^{T}:=\overline{g}_{T_{0}}(\pi_{N}(\cdot),\pi_{N}%
(\cdot)),$
\[%
\begin{array}
[c]{c}%
d_{T_{0}}^{T}(y_{1},y_{2})=\inf\int_{0}^{1}\sqrt{\overline{g}_{T_{0}}%
^{T}(\gamma^{\prime}(s),\gamma^{\prime}(s))}ds
\end{array}
\]
where the infimum is taken over all piecewise smooth paths $\gamma
:[0,1]\rightarrow N$ with $\gamma(0)=y_{1}$ and $\gamma(1)=y_{2}$ such that
\[%
\begin{array}
[c]{c}%
d_{T_{0}}^{T}(y_{1},y_{2})=d_{\overline{g}_{T_{0}}}(\pi_{N}(y_{1}),\pi
_{N}(y_{2})).
\end{array}
\]
Then
\[%
\begin{array}
[c]{c}%
d_{T_{0}}^{T}=\pi_{N}^{\ast}(d_{k_{T_{0}}}),
\end{array}
\]
where $d_{k_{T_{0}}}$ is the distance with respect to $k_{T_{0}}$ on $Y$ as in
(\ref{2022-1}).

Adapting from section $3$, it follows from the previous set-up that there is a
commutative diagram
\begin{equation}%
\begin{array}
[c]{cclcl}%
S(-1)\thickapprox\mathbb{C}\times S^{3} & \overset{\pi_{2}}{\rightarrow} &
S^{3}\subset(M,g^{T},\omega,d^{T}) & \overset{}{\longrightarrow} & S_{y_{0}%
}^{1}\subset(N,\overline{g}_{T_{0}},\omega_{T},d_{T_{0}}^{T})\\
\downarrow\overline{H} &  & \downarrow\pi_{M} &  & \downarrow\pi_{N}\\
O(-1) & \overset{\pi_{d}}{\rightarrow} & CP^{1}\subset(Z,h,\omega_{h},d_{h}) &
\overset{}{\longrightarrow} & \pi_{N}(y_{0})\in(Y,k_{T_{0}},d_{k_{T_{0}}}).
\end{array}
\label{2022-2}%
\end{equation}
Here $d_{h}$ is the distance with respect to $h$ on $Z$. Note that
\begin{equation}
d^{T}=\pi_{M}^{\ast}(d_{h})\text{ \ \textrm{and} \ }\omega=\pi_{M}^{\ast
}(\omega_{h}).\label{2022-3}%
\end{equation}
Hence, in view of the commutative diagram (\ref{2022-2}), all the estimates as
in \cite{sw2} can be adapted on $Z,$ $Y$ and then on $M,$ $N$ via
(\ref{2022-3}). More precisely, it follows from \cite[Lemma 3.2]{sw2} that

\begin{lemma}
\label{b}There exists a uniform constant $C$ such that for any points
$x_{1},x_{2}\in V\subset(M,g^{T},\omega,d^{T})$%
\[%
\begin{array}
[c]{c}%
d_{\omega}^{T}(x_{1},x_{2})\leq C(T-t)^{\frac{1}{3}}%
\end{array}
\]
for any $t\in\lbrack0,T_{0})$.
\end{lemma}

Note that when $r$ small enough, $B_{\xi,\overline{g}_{T_{0}}}(y,r)$ is a
trivial $S^{1}$-bundle over the geodesic ball $B_{k}(\pi_{N}(y),r)$. Then it
follows from \cite[Lemma 3.3]{sw2}, Lemma \ref{b} and (\ref{2022-4}) that

\begin{lemma}
\label{c}For any $\varepsilon>0,$ there exist $\delta$ and $T\in\lbrack
0,T_{0})$ such that
\[%
\begin{array}
[c]{c}%
\mathrm{diam}_{d_{T_{0}}^{T}}B_{\xi,\overline{g}_{T_{0}}}(y,\delta
)<\varepsilon
\end{array}
\]
and
\[%
\begin{array}
[c]{c}%
\mathrm{diam}_{\omega(t)}\psi^{-1}(B_{\xi,\overline{g}_{T_{0}}}(y,\delta
))<\varepsilon
\end{array}
\]
for all $t\in\lbrack0,T_{0}).$
\end{lemma}

Combining Lemma \ref{c} and \cite[Theorem 6.2]{chlw}, it follows that

\begin{proposition}
$(M,g^{T}(t))$ converges to $(N,d_{T_{0}}^{T})$ and $(M,g(t))$ converges to
$(N,d_{T_{0}})$ in the sense of Gromov-Hausdorff as $t\rightarrow T_{0}^{-}.$
In particular, $(N,d_{T_{0}})$ is compact metric space homeomorphic to the
Sasakian $5$-manifold $N$.
\end{proposition}

This establishes the part $(2)$ in the Definition \ref{D12} of floating
foliation canonical surgical contraction under the assumption (\ref{2022}).

\subsubsection{Higher Order Estimates for $\omega(t)$ as $t\rightarrow
T_{0}^{-}$}

Under the assumption of Theorem \ref{T61}, we know that $(M,g(t))$ converges
to $(N,d_{T_{0}})$ in the sense of Gromov-Hausdorff as $t\rightarrow T_{0}%
^{-}$ and $\omega(t)$ converges in $C^{\infty}$ on any compact subset of
$M\backslash\cup_{i=1}^{k}V_{i}$ to a smooth transverse K\"{a}hler metric
$\omega_{T}$ on $M\backslash\cup_{i=1}^{k}V_{i}$. In particular, we already
have $C^{\infty}$ estimates for $\omega(t)$ away from the floating foliation
$(-1)$-curves $V_{i}$ as $t$ approaches $T_{0}^{-}$. However, to understand
how the Sasaki-Ricci flow can be continued past the singular time we need more
precise estimates.

We assume as before that there is only one exceptional foliation $(-1)$-curve
$V$. From Lemmas \ref{L62}, \ref{L64} and \ref{L65}, there exist positive
constants $\delta$ and $C$ such that%
\begin{equation}%
\begin{array}
[c]{c}%
\frac{|s|_{h}^{2}}{C}\omega_{0}\leq\omega(t)\leq\frac{C}{|s|_{h}^{2(1-\delta
)}}\omega_{0}.
\end{array}
\label{81}%
\end{equation}
For simple we assume that $|s|_{h}\leq1$ on $M$. Following the same method in
\cite{pss}, we define an endomorphism $H=H(t)$ of the transverse tangent
bundle by $H_{\ell}^{i}=g_{0}^{Ti\overline{j}}g_{\ell\overline{j}}^{T}$ and
consider the quantity $S=S(t)$ given by%
\[%
\begin{array}
[c]{c}%
S=\left\vert \nabla^{T}\log H\right\vert ^{2},
\end{array}
\]
here the transverse covariant derivative $\nabla^{T}$ and the norm $|\cdot| $
are taken with respect to the evolving metric $g^{T}$.

\begin{proposition}
There exist positive constants $\alpha$ and $C$ such that for $t\in
\lbrack0,T)$,%
\begin{equation}%
\begin{array}
[c]{c}%
S\leq\frac{C}{|s|_{h}^{2\alpha}}.
\end{array}
\label{83}%
\end{equation}

\end{proposition}

\begin{proof}
Following the same computation as in \cite{pss}, we have%
\begin{equation}%
\begin{array}
[c]{l}%
(\frac{\partial}{\partial t}-\Delta_{B})S\\
\leq-|\nabla^{T}\nabla^{T}\log H|^{2}-|\overline{\nabla}^{T}\nabla^{T}\log
H|^{2}+CS+C|\nabla^{T}Rm^{T}(g_{0}^{T})|^{2},
\end{array}
\label{84}%
\end{equation}
where $Rm^{T}(g_{0}^{T})$ is the transverse Riemannian curvature of $g_{0}%
^{T}$. Using (\ref{81})%
\begin{equation}%
\begin{array}
[c]{c}%
|\nabla^{T}Rm^{T}(g_{0}^{T})|^{2}\leq C|s|_{h}^{-K}(S+1),
\end{array}
\label{85}%
\end{equation}
for a positive constant $K$. By the Cauchy-Schwarz inequality we have%
\begin{equation}
|\nabla^{T}S|\leq S^{1/2}(|\nabla^{T}\nabla^{T}\log H|+|\overline{\nabla}%
^{T}\nabla^{T}\log H|),\label{85a}%
\end{equation}
and the inequalities%
\begin{equation}
|\nabla^{T}|s|_{h}^{4K}|\leq C|s|_{h}^{3K},\text{ }|\Delta_{B}|s|_{h}%
^{4K}|\leq C|s|_{h}^{3K},\label{85b}%
\end{equation}
where we are choosing $K$ sufficiently large. Combining (\ref{84}),
(\ref{85}), (\ref{85a}) and (\ref{85b}), we obtain%
\begin{equation}%
\begin{array}
[c]{ll}
& (\frac{\partial}{\partial t}-\Delta_{B})(|s|_{h}^{4K}S)\\
= & |s|_{h}^{4K}(\frac{\partial}{\partial t}-\Delta_{B})S-2\operatorname{Re}%
(\nabla^{T}|s|_{h}^{4K}\cdot\overline{\nabla}^{T}S)-(\Delta_{B}|s|_{h}%
^{4K})S\\
\leq & -|s|_{h}^{4K}(|\nabla^{T}\nabla^{T}\log H|^{2}+|\overline{\nabla}%
^{T}\nabla^{T}\log H|^{2})+C|s|_{h}^{2K}S\\
& +C|s|_{h}^{3K}S^{1/2}(|\nabla^{T}\nabla^{T}\log H|+|\overline{\nabla}%
^{T}\nabla^{T}\log H|)+C\\
\leq & C(1+|s|_{h}^{2K}S).
\end{array}
\label{86}%
\end{equation}
Moreover, we know that $\mathrm{tr}_{\omega_{0}}\omega$ satisfies%
\[%
\begin{array}
[c]{lll}%
(\frac{\partial}{\partial t}-\Delta_{B})\mathrm{tr}_{\omega_{0}}\omega & = &
-g^{Ti\overline{j}}R_{i\overline{j}}^{T\text{ \ }k\overline{l}}(g_{0}%
^{T})g_{k\overline{l}}^{T}-g_{0}^{Tk\overline{l}}g^{Ti\overline{j}%
}g^{Tp\overline{q}}\overset{0}{\nabla_{i}^{T}}g_{k\overline{q}}^{T}%
\overset{0}{\nabla_{\overline{j}}^{T}}g_{p\overline{l}}^{T}\\
& \leq & C|s|_{h}^{-K}-\frac{1}{C}|s|_{h}^{K}S-\frac{1}{2}g_{0}^{Tk\overline
{l}}g^{Ti\overline{j}}g^{Tp\overline{q}}\overset{0}{\nabla_{i}^{T}%
}g_{k\overline{q}}^{T}\overset{0}{\nabla_{\overline{j}}^{T}}g_{p\overline{l}%
}^{T},
\end{array}
\]
where we used (\ref{81}) again. Here$\overset{0}{\text{ }\nabla^{T}}$ denotes
the transverse covariant derivative with respect to $g_{0}^{T}$. Since%
\begin{equation}%
\begin{array}
[c]{c}%
|\nabla^{T}\mathrm{tr}_{\omega_{0}}\omega|^{2}\leq(\mathrm{tr}_{\omega_{0}%
}\omega)g_{0}^{Tk\overline{l}}g^{Ti\overline{j}}g^{Tp\overline{q}%
}\overset{0}{\nabla_{i}^{T}}g_{k\overline{q}}^{T}\overset{0}{\nabla
_{\overline{j}}^{T}}g_{p\overline{l}}^{T},
\end{array}
\label{88}%
\end{equation}
we have%
\begin{equation}%
\begin{array}
[c]{ll}
& (\frac{\partial}{\partial t}-\Delta_{B})(|s|_{h}^{K}\mathrm{tr}_{\omega_{0}%
}\omega)\\
\leq & -\frac{1}{C}|s|_{h}^{2K}S+C-2\operatorname{Re}(\nabla^{T}|s|_{h}%
^{K}\cdot\overline{\nabla}^{T}\mathrm{tr}_{\omega_{0}}\omega)\\
& -\frac{1}{2}|s|_{h}^{K}g_{0}^{Tk\overline{l}}g^{Ti\overline{j}%
}g^{Tp\overline{q}}\overset{0}{\nabla_{i}^{T}}g_{k\overline{q}}^{T}%
\overset{0}{\nabla_{\overline{j}}^{T}}g_{p\overline{l}}^{T}\\
\leq & -\frac{1}{C}|s|_{h}^{2K}S+C,
\end{array}
\label{89}%
\end{equation}
where
\[%
\begin{array}
[c]{lll}%
2|\operatorname{Re}(\nabla^{T}|s|_{h}^{K}\cdot\overline{\nabla}^{T}%
\mathrm{tr}_{\omega_{0}}\omega)| & \leq & C+\frac{1}{C}|\nabla^{T}|s|_{h}%
^{K}|^{2}|\nabla^{T}\mathrm{tr}_{\omega_{0}}\omega|^{2}\\
& \leq & C+\frac{1}{2}|s|_{h}^{K}g_{0}^{Tk\overline{l}}g^{Ti\overline{j}%
}g^{Tp\overline{q}}\overset{0}{\nabla_{i}^{T}}g_{k\overline{q}}^{T}%
\overset{0}{\nabla_{\overline{j}}^{T}}g_{p\overline{l}}^{T},
\end{array}
\]
which follows from (\ref{81}) and (\ref{88}).

By applying the maximum principle to $Q=|s|_{h}^{4K}S+A|s|_{h}^{K}%
\mathrm{tr}_{\omega_{0}}\omega-Bt,$ for constants $A$ and $B$ be choosing
sufficiently large, from (\ref{86}) and (\ref{89}) we obtain
\[%
\begin{array}
[c]{c}%
(\frac{\partial}{\partial t}-\Delta_{B})Q<0.
\end{array}
\]
This gives a uniform upper bound for $Q$ and (\ref{83}) follows.
\end{proof}

We have the estimates of curvature of $g^{T}$ and all its higher order
estimates as in \cite{sw2}.

\begin{proposition}
For each integer $m\geq0$ there exist $C_{m},$ $\alpha_{m}>0$ such that for
$t\in\lbrack0,T_{0})$,%
\begin{equation}%
\begin{array}
[c]{c}%
|(\nabla_{\mathbb{R}}^{T})^{m}Rm^{T}(g^{T}(t))|\leq\frac{C_{m}}{|s|_{h}%
^{2\alpha_{m}}}\text{ \ \textrm{and \ }}|(\overset{0}{\nabla_{\mathbb{R}}^{T}%
})^{m}g^{T}(t)|_{g_{0}^{T}}\leq\frac{C_{m}}{|s|_{h}^{2\alpha_{m}}},
\end{array}
\label{90}%
\end{equation}
where $\nabla_{\mathbb{R}}^{T}=\frac{1}{2}(\nabla^{T}+\overline{\nabla}^{T})$
and $\overset{0}{\nabla_{\mathbb{R}}^{T}}$ denote the real covariant
derivative of $g^{T}$ and $g_{0}^{T},$ respectively.
\end{proposition}

\subsubsection{Continuing the Sasaki-Ricci Flow at $T_{0}$}

Now we will show how to continue the Sasaki-Ricci flow past time $T_{0}$ on
the manifold $N$ by following the methods in \cite{sw2}. First we know that
$(M,g^{T}(t))$ converges to $(N,d_{T_{0}}^{T})$ in the sense of
Gromov-Hausdorff as $t\rightarrow T_{0}^{-}$. We explain how one can continue
the Sasaki-Ricci flow through the singularity. We replace $M$ with the
manifold $N$ at the singular time $T_{0}$. Again, we assume for simplicity
that we have only one exceptional floating foliation $(-1)$-curve $V$.

Write $\widehat{\omega}_{T_{0}}=\psi^{\ast}\omega_{N}$, where $\omega_{N}$ is
the smooth transverse K\"{a}hler metric on $N$. From Lemma \ref{L61}, there is
a closed positive $(1,1)$ current $\omega_{T_{0}}$ and a bounded basic
function $\varphi_{T_{0}}$ with%
\begin{equation}%
\begin{array}
[c]{c}%
\omega_{T_{0}}=\widehat{\omega}_{T_{0}}+\sqrt{-1}\partial_{B}\overline
{\partial}_{B}\varphi_{T_{0}}\geq0.
\end{array}
\label{91}%
\end{equation}
Moreover from Lemma \ref{L62}, $\varphi(t)$ converges to $\varphi_{T_{0}}$
pointwise on $M$ and smoothly on any compact subset of $M\backslash V$. Since
$\varphi_{T_{0}}$ is constant on $V,$ there exists a bounded function
$\phi_{T_{0}}$ on $N$ and is smooth on $N\backslash S^{1}$ such that
$\psi^{\ast}\phi_{T_{0}}=\varphi_{T_{0}}.$ We then define a closed positive
$(1,1)$ current $\omega^{\prime}$ on $N$ by%
\begin{equation}%
\begin{array}
[c]{c}%
\omega^{\prime}=\omega_{N}+\sqrt{-1}\partial_{B}\overline{\partial}_{B}%
\phi_{T_{0}}\geq0.
\end{array}
\label{92}%
\end{equation}
This implies $\omega^{\prime}$ is smooth on $N\backslash S^{1}$ which
satisfies $\omega_{T_{0}}=\psi^{\ast}\omega^{\prime}$. It follows from
\cite{kol1}, \cite{kol3} and \cite[Lemma 5.1 and Lemma 5.2]{sw2} that there
exists $p>1$ such that
\begin{equation}
\frac{\omega^{\prime n}\wedge\eta_{N}}{\omega_{N}^{n}\wedge\eta_{N}}\in
L^{p}(N)\label{2022aaa}%
\end{equation}
and then $\phi_{T_{0}}$ is a bounded basic continuous function on $N$, which
is smooth on $N\backslash S^{1}$, such that $\varphi_{T_{0}}=\psi^{\ast}%
\phi_{T_{0}}$.

Then, by using the method as in section $4$ and \cite{st}, we construct a
solution of the Sasaki-Ricci flow on $N$ starting at $\omega^{\prime}$. First
one fix a smooth closed $(1,1)$ form $\chi\in c_{1}^{B}(N)$. There exists
$T^{\prime}>T_{0},$ this $T^{\prime}$ is strictly less than the maximal time
$T_{N},$ such that the closed $(1,1)$ form%
\[%
\begin{array}
[c]{c}%
\widehat{\omega}_{t,N}:=\omega_{N}+(t-T)\chi
\end{array}
\]
is a transverse K\"{a}hler metric for $t\in\lbrack T_{0},T^{\prime}]$. We use
the metrics $\widehat{\omega}_{t,N}$ as our reference metrics for the
Sasaki-Ricci flow as it continues on $N$. Fix a smooth adapted measure
$\Omega_{N}$ and then a volume form $\Omega_{N}\wedge\eta_{N}$ on $N$
satisfying
\[%
\begin{array}
[c]{c}%
\sqrt{-1}\partial_{B}\overline{\partial}_{B}\log\Omega_{N}=\frac{\partial
}{\partial t}\widehat{\omega}_{t,N}=\chi\in c_{1}^{B}(N),
\end{array}
\]
for $t\in\lbrack T_{0},T^{\prime}]$.

Next by using the method as in section $4$ and \cite{st}$,$ we will now
construct a family of basic functions $\phi_{T_{0},\varepsilon}$ on $N$ which
converge to $\phi_{T_{0}}$. For $\varepsilon>0$ sufficiently small and $K$
fixed be sufficiently large, define a family of volume forms $\Omega
_{\varepsilon}\wedge\eta_{N}$ on $N$ by%
\[%
\begin{array}
[c]{c}%
\Omega_{\varepsilon}\wedge\eta_{N}=(\psi|_{M\backslash V}^{-1})^{\ast}%
(\frac{|s|_{h}^{K}\omega^{2}(T-\varepsilon)}{|s|_{h}^{K}+\varepsilon}%
)\wedge\eta_{N}+\varepsilon\Omega_{N}\wedge\eta_{N}\text{ \ }\mathrm{on}\text{
}N\backslash S^{1},
\end{array}
\]
and $(\Omega_{\varepsilon}\wedge\eta_{N})|_{S^{1}}=\varepsilon(\Omega
_{N}\wedge\eta_{N})|_{S^{1}}$. We then define the basic functions $\phi
_{T_{0},\varepsilon}$ to be solutions of the transverse complex
Monge-Amp\'{e}re equations%
\begin{equation}%
\begin{array}
[c]{c}%
(\omega_{N}+\sqrt{-1}\partial_{B}\overline{\partial}_{B}\phi_{T_{0}%
,\varepsilon})^{2}\wedge\eta_{N}=C_{\varepsilon}\Omega_{\varepsilon}\wedge
\eta_{N};\sup(\phi_{T_{0},\varepsilon}-\phi_{T_{0}})=\sup(\phi_{T_{0}}%
-\phi_{T_{0},\varepsilon})
\end{array}
\label{93}%
\end{equation}
\textrm{with}\ $C_{\varepsilon}\int_{N}\Omega_{\varepsilon}\wedge\eta_{N}%
=\int_{N}\omega_{N}^{2}\wedge\eta_{N}.$ Such $\phi_{T_{0},\varepsilon}$ exist
and are unique due to El Kacimi-Alaoui (\cite{eka}) These solution functions
$\phi_{T_{0},\varepsilon}$ are continuous on $N$ and smooth on $N\backslash
S^{1}.$ Furthermore, it follows from (\ref{2022aaa}) and \cite{kol2} that
\[
||\frac{C_{\varepsilon}\Omega_{\varepsilon}\wedge\eta_{N}}{\Omega_{N}%
\wedge\eta_{N}}-\frac{\omega^{\prime n}\wedge\eta_{N}}{\Omega_{N}\wedge
\eta_{N}}||_{L^{1}(N)}\rightarrow0\text{ \ \ }\mathrm{as}\text{\quad
}\varepsilon\rightarrow0
\]
and thus
\[
|\phi_{T_{0},\varepsilon}-\phi_{T_{0}}|_{L^{\infty}(N)}\rightarrow0\text{
\ \ }\mathrm{as}\text{\quad}\varepsilon\rightarrow0.
\]

Now let $\varphi_{\varepsilon}=\varphi_{\varepsilon}(t)$ to be the basic
functions which are solutions of the transverse parabolic complex
Monge-Amp\'{e}re equations%

\[%
\begin{array}
[c]{c}%
\frac{\partial}{\partial t}\varphi_{\varepsilon}=\log\frac{(\widehat{\omega
}_{t,N}+\sqrt{-1}\partial_{B}\overline{\partial}_{B}\varphi_{\varepsilon}%
)^{2}\wedge\eta_{N}}{\Omega_{N}\wedge\eta_{N}},\text{ \ }\varphi_{\varepsilon
}(T_{0})=\phi_{T_{0},\varepsilon},
\end{array}
\]
for $t\in\lbrack T_{0},T^{\prime}]$. It follows from \cite[Proposition
5.1]{sw2} and section $4$ that there exists a basic function $\varphi$ in
$C^{0}([T_{0},T^{\prime}]\times N)\cap C^{\infty}((T_{0},T^{\prime}]\times N)$
such that $\varphi_{\varepsilon}\rightarrow\varphi$ in $L^{\infty}%
([T_{0},T^{\prime}]\times N)$ as $\varepsilon\rightarrow0,$ such a $\varphi$
is continuous on $[T_{0},T^{\prime}]\times N$ and smooth on any compact subset
of $(T_{0},T^{\prime}]\times N.$ This basic function $\varphi$ satisfies the
transverse complex Monge-Amp\'{e}re equation%
\begin{equation}%
\begin{array}
[c]{c}%
\frac{\partial}{\partial t}\varphi=\log\frac{(\widehat{\omega}_{t,N}+\sqrt
{-1}\partial_{B}\overline{\partial}_{B}\varphi)^{2}\wedge\eta_{N}}{\Omega
_{N}\wedge\eta_{N}},\text{ }\varphi(T_{0})=\phi_{T_{0}},t\in(T_{0},T^{\prime}]
\end{array}
\label{94}%
\end{equation}
in $C^{0}([T_{0},T^{\prime}]\times N)\cap C^{\infty}((T_{0},T^{\prime}]\times
N)$. Then we define for $t\in\lbrack T_{0},T^{\prime}],$%
\[%
\begin{array}
[c]{c}%
\omega_{T_{0},\varepsilon}(t):=T_{0}\omega_{N}+\sqrt{-1}\partial_{B}%
\overline{\partial}_{B}\phi_{T_{0},\varepsilon},
\end{array}
\]
and
\[%
\begin{array}
[c]{c}%
\omega_{\varepsilon}(t)=\widehat{\omega}_{t,N}+\sqrt{-1}\partial_{B}%
\overline{\partial}_{B}\varphi_{\varepsilon}.
\end{array}
\]
We thus obtain estimates on $\omega_{T_{0},\varepsilon}$ and $\omega
_{\varepsilon}$ and its volume form $\omega_{\varepsilon}^{2}\wedge\eta_{N}$
for $t\in\lbrack T_{0},T^{\prime}]$ as Lemmas \ref{L62}, \ref{L64} and
\ref{L65}.

\begin{lemma}
\label{L51} There exist positive constants $\alpha$ and $C,$ independent of
$\varepsilon$, such that
\begin{equation}%
\begin{array}
[c]{c}%
\frac{|s|_{h}^{2\alpha}}{C}\omega_{N}\leq\omega_{T,\varepsilon}\leq\frac
{C}{|s|_{h}^{2\alpha}}\omega_{N}%
\end{array}
\label{95}%
\end{equation}
and also
\begin{equation}%
\begin{array}
[c]{c}%
\frac{|s|_{h}^{2\alpha}}{C}\omega_{N}\leq\omega_{\varepsilon}\leq\frac
{C}{|s|_{h}^{2\alpha}}\omega_{N}\text{ \ \textrm{and} \ }\omega_{\varepsilon
}^{2}\wedge\eta_{N}\leq\frac{C}{|s|_{h}^{2\alpha}}\Omega_{N}\wedge\eta_{N}%
\end{array}
\label{96}%
\end{equation}
on $[T_{0},T^{\prime}]\times(N\backslash S^{1})$. Here we write $(\psi
|_{M\backslash V}^{-1})^{\ast}|s|_{h}^{2}$ on $N\backslash S^{1}$ as
$|s|_{h}^{2}.$ For a fix large positive integer $L$, for each integer $0\leq
m\leq L$ there exist $C_{m},$ $\alpha_{m}>0$ such that%
\[%
\begin{array}
[c]{c}%
|(\nabla_{\mathbb{R}}^{T})^{m}g^{T}(t)|_{g_{N}^{T}}\leq\frac{C_{m}}%
{|s|_{h}^{2\alpha_{m}}}\text{ \ \textrm{and } }|(\nabla_{\mathbb{R}}^{T}%
)^{m}g_{\varepsilon}^{T}(t)|_{g_{N}^{T}}\leq\frac{C_{m}}{|s|_{h}^{2\alpha_{m}%
}}.
\end{array}
\]
where $\nabla_{\mathbb{R}}^{T}$ denote the real covariant derivative with
respect to the fixed metric $g_{N}^{T}$.
\end{lemma}

We can now prove that the Sasaki-Ricci flow can be smoothly connected at time
$T_{0}$ between $[0,T_{0})\times M$ and $(T_{0},T^{\prime}]\times N$, outside
$T_{0}\times S^{1}\cong T_{0}\times V$ via the map $\psi$.

Recall that for $t\in\lbrack T_{0},T^{\prime}],$ $\varphi(t)$ is the limit of
$\varphi_{\varepsilon}$ as $\varepsilon\rightarrow0$ and as in (\ref{94}), the
metric
\[
\omega=\widehat{\omega}_{t,N}+\sqrt{-1}\partial_{B}\overline{\partial}%
_{B}\varphi
\]
is the solution of the Sasaki-Ricci flow on $N$%
\begin{equation}%
\begin{array}
[c]{c}%
\frac{\partial}{\partial t}\omega=-\mathrm{Ric}^{T}(\omega)\text{
}\mathrm{with}\text{ }\omega(T_{0})=\omega^{\prime},
\end{array}
\label{94a}%
\end{equation}
for $t\in(T_{0},T^{\prime}]$. We define%
\[%
\begin{array}
[c]{c}%
Z=([0,T_{0})\times M)\cup(T_{0}\times N\backslash S^{1})\cup((T_{0},T^{\prime
}]\times N).
\end{array}
\]
Consider a family of metrics $\omega(t,x)$, for $(t,x)\in Z$. We define what
is $\omega$ is smooth on $Z$. If $(t,x)\in\lbrack0,T_{0})\times M$ or
$(t,x)\in(T_{0},T^{\prime}]\times N,$ then we have $\omega$ to be smooth in
the usual sense. Recall that we have the estimates for $\omega(t,x)$ on
$[0,T_{0})\times M\backslash V$ by (\ref{90}) and on $[T_{0},T^{\prime}]\times
N\backslash S^{1}$ from (\ref{96}). When $(t,x)=(T_{0},x)\in T_{0}\times
N\backslash S^{1}\cong T_{0}\times M\backslash V$, we take a small
neighborhood $U$ of $x$ in $M\backslash V$ and consider $\omega$ as a metric
on $(T_{0}-\delta,T_{0}+\delta)\times U$, for some $\delta>0$. We say $\omega$
is smooth at $(T_{0},x)$ if $\omega$ is smooth at $(T_{0},x)$ in the
$(T_{0}-\delta,T_{0}+\delta)\times U$. It follows from Lemma \ref{L51} that
$\omega=\omega(t)$ satisfies the Sasaki-Ricci flow (\ref{94}) and is smooth a
time $T_{0}$ in the sense above. Hence the metric $\omega=\omega(t)$ is a
smooth solution of the Sasaki-Ricci flow in the space-time region $Z$.

This establishes the part $(3)$ in the Definition \ref{D12} of floating
foliation canonical surgical contraction under the assumption (\ref{2022}).

\subsubsection{Gromov-Hausdorff Convergence as $t\rightarrow T_{0}^{+}$}

To complete the proof of Theorem \ref{T61}, it remains to show that
$(N,\omega(t))$ converges in the Gromov-Hausdorff sense to $(N,d_{T})$ as
$t\rightarrow T_{0}^{+}$.

\begin{lemma}
There exists $\delta>0$ and a uniform constant $C$ such that for
$\omega=\omega(t)$ a solution of the Sasaki-Ricci flow (\ref{94}) satisfies
\[%
\begin{array}
[c]{c}%
\omega\leq C\frac{\omega_{N}}{(\psi|_{M\backslash V}^{-1})^{\ast}|s|_{h}^{2}%
}\text{ \ \textrm{and} \ }\omega\leq C(\psi|_{M\backslash V}^{-1})^{\ast
}(\frac{\omega_{0}}{|s|_{h}^{2(1-\delta)}}),
\end{array}
\]
for $t\in\lbrack T_{0},T^{\prime}].$ Here $\omega_{0}$ is the initial metric
on $M$.
\end{lemma}

\begin{proof}
The arguments of the proof involved is similar to the estimate as in Lemma
\ref{L65}, we omit it. We refer to \cite[Proposition 6.1]{sw2} for some
details as in the K\"{a}hler-Ricci flow.
\end{proof}

Finally, it follows by the arguments of earlier sections that $(N,\omega(t)) $
converges in the Gromov-Hausdorff sense to $(N,d_{T})$ as $t\rightarrow
T_{0}^{+}$. This establishes the part $(4)$ in the Definition \ref{D12} of
floating foliation canonical surgical contraction under the assumption
(\ref{2022}).

This completes the proof of Theorem \ref{T61}.

\subsection{Analytic Foliation Minimal Model Program with Scaling}

As a consequence of Theorem \ref{T61} and Proposition \ref{P41}, we will prove
our main result in the paper on the analytic f%
oliation minimal model program with scaling in a compact
quasi-regular Sasakian 
$%
5
$%
-Mmanifold. We will follow 
the lines of the arguments of the proof of Proposition \ref{P41} via the
Sasaki-Ricci flow.

\textbf{The proof of Theorem \ref{T62} : }

\begin{proof}
Let $(M,\xi,\omega_{0})$ be a compact quasi-regular Sasakian $5$-manifold with
a smooth transverse Kaehler metric $\omega_{0}$. In view of the cohomological
characterization of the maximal solution of the Sasaki-Ricci flow (\ref{1}),
we start with a pair $(M,H^{T}),$ where $M$ is a Sasakian manifold with an
ample basic divisor $H^{T}$. Let
\[
T_{0}=\sup\{t>0\mathbf{\ |\ }H^{T}+tK_{M}^{T}\text{ \ is \textrm{nef}}\}.
\]
Denote
\[
L_{0}^{T}:=H^{T}+T_{0}K_{M}^{T}%
\]
which is a basic $Q$-line bundle and semi-ample. In fact, it follows from
Kleiman criterion that $\ $%
\[
mL_{0}^{T}-T_{0}K_{M}^{T}%
\]
is ample and then nef and big for some sufficiently large $m$. Then by
Kawamata criterion for base-point free, we have the semi-ample for $L_{0}%
^{T}.$

Next we define a subcone
\[
R:=\overline{NE(M)}_{K_{M}^{T}<0}\cap(L_{0}^{T})^{\perp}%
\]
which is a foliation extremal ray $R$ with the generic choice of $H^{T}.$ For
$V$ with $0=L_{0}^{T}\cdot V,$ we have
\[
K_{M}^{T}\cdot V=-\frac{1}{T_{0}}(H^{T}\cdot V)<0.
\]
That is the map $\Psi:M\rightarrow N$ induced from $(L_{0}^{T})^{m}$ contract
all foliation curves whose class lies in the foliation extremal ray $R$ with
\[
L_{0}^{T}\cdot V=0.
\]
and
\[
K_{M}^{T}\cdot V<0.
\]
We observe that $K_{M}^{T}\cdot V=0$ as $T_{0}\rightarrow\infty.$

Now we consider a sequence of contarctions $g(t)$ on the manifolds
$M_{0},M_{1},...,M_{k}$ on the time intervals $[0,T_{0}),$ $(T_{1},T_{2}),$
$\cdot\cdot\cdot,(T_{k-1},T_{k}).$ Denote $L_{-1}^{T}=H^{T}$ as above and
\[
L_{l}^{T}:=L_{l-1}^{T}+T_{i}K_{M_{l}}^{T}.
\]

Then the nef class $L_{l}^{T}$ is semi-ample, there exists a map $\psi
:M_{l}\rightarrow N$ where $N$ is a quasi-regular Sasakian $5$-manifold with
foliation cyclic quotient singularities. The exceptional locus of $\psi$ is a
sum of irreducible foliation curves $W=\Sigma_{i}W_{i}$ with
\[
W_{i}\cdot L_{l}^{T}=0
\]
and
\[
K_{M}^{T}\cdot W_{i}<0.
\]

\textbf{(I) If }$L_{l}^{T}$\textbf{\ is big :}

\textbf{(i)} \ If all foliation curves $W=\Sigma_{i}W_{i}$ do not pass the
singularity on $M_{l}$ : We show that the Sasaki-Ricci flow will perform a
canonical surgical contraction in $M_{l}$ as in Theorem \ref{T61}.\ \ Since
$L_{l}^{T}$ is big and all foliation curves $W_{i}$ do not pass the
singularity on $M_{l}$, thus by index theorem we have $W_{i}^{2}<0,$ and by
Adjunction Formula, all foliation curves $W_{i}$ are foliation $(-1)$-curves
and floating on $M_{l}.$ On the other hand,
\[
W_{i}\cdot L_{l}^{T}=0\text{ \ \ and \ \ }W_{j}\cdot L_{l}^{T}=0,i\neq j
\]
and
\[
W_{i}^{2}=-1=W_{j}^{2}.
\]
Then%
\[
(W_{i}+W_{j})\cdot L_{l}^{T}=0
\]
and index theorem again
\[
(W_{i}+W_{j})^{2}<0.
\]
Hence
\[
W_{i}\cdot W_{j}=0
\]
and $\{W_{i}\}$ is a finite number of disjoint floating foliation
$(-1)$-curves in $M_{l}.$ Thus by Sasaki analogue of Castelnuovo's contraction
theorem (Theorem \ref{T33}), Thus $\psi$ is a map blowing down the basic
exceptional curves $W_{i}$. It follows from Proposition \ref{P31} that
$L_{l}^{T}$ is the pull-back of an ample line bundle over $M_{l},$ we obtain
(\ref{2022}) for some transverse $\omega_{M_{l}}$ on $M_{l}$. It follows from
Theorem \ref{T61} that the Sasaki-Ricci flow $g(t)$ performs a foliation
canonical surgical contraction with respect to the data $W_{1},...,W_{k}$, $N$
and $\psi$.

\textbf{(ii)} If some of $W_{i_{0}}:=\Gamma$ pass the foliation singularity of
type $\frac{1}{r}(1,a)$ : We show that the Sasaki-Ricci flow will perform
\textbf{f}oliation extremal contractions of foliation $K_{M}^{T}$-negative
curves. Following above notions, it follows from Theorem \ref{T32} that we can
have the minimal resolutions of foliation singularities of $M_{l}$ and $N $
\[
\varphi:\widetilde{M_{l}}\rightarrow M_{l}\text{ \ \ \textrm{and}
\ \ \ }\widetilde{\varphi}:\widetilde{N}\rightarrow N.
\]

Our goal is to find $\widetilde{\psi}:\widetilde{M_{l}}\rightarrow
\widetilde{N}$ and $\psi:M_{l}\rightarrow N$ such that the following diagram
is commutative :%
\begin{equation}%
\begin{array}
[c]{lcl}%
\widetilde{M_{l}}\supset V_{i},\text{ }\widetilde{\Gamma} &
\overset{\widetilde{\psi}}{\longrightarrow} & \widetilde{N}\\
\downarrow\varphi & \circlearrowright & \downarrow\widetilde{\varphi}\\
M_{l}\supset\Gamma,\text{ }\mathbf{S}_{p}^{1} & \overset{\psi}{\longrightarrow
} & N\ni\mathbf{S}_{q}^{1}.
\end{array}
\label{2022-8}%
\end{equation}
Suppose $\widetilde{M_{l}}$ and $\widetilde{N}$ are not isomorphic; then as
they are both regular Sasakian manifolds, by Theorem \ref{T33}, there must
exist a foliation $(-1)$-curve $\widetilde{\Gamma}$ such that
$\widetilde{\varphi}\circ\widetilde{\psi}(\widetilde{\Gamma})=\mathbf{S}%
_{q}^{1}.$ Let $\Theta$ be a set of the foliation curves in $\widetilde{M_{l}%
}$ such that $\Theta\mathcal{=}$ $\widetilde{\psi}^{-1}\circ\widetilde{\varphi
}^{-1}(\mathbf{S}_{q}^{1})$ and $D^{T}=$ $\{V_{i}\}$ be the exceptional locus
of $\varphi$ at the singular fibre $\mathbf{S}_{p}^{1}.$ Then the Hirzebruch
Jung continued fraction
\[
\frac{r}{a}=[b_{1},\cdot\cdot\cdot,b_{m}]
\]
say that
\[
V_{i}^{2}=-b_{i}.
\]

Now there is at most one such a foliation exceptional curve $V_{j}$ for some
$j$ so that $V_{j}\cdot\widetilde{\Gamma}\neq\emptyset$ with $V_{j}^{2}%
=-b_{j}$. Then
\begin{equation}
K_{M_{l}}^{T}\cdot\Gamma<0\label{2022-7}%
\end{equation}
and $\Gamma$ passed the foliation singularity of type $\frac{1}{r}(1,a).$ As
in section $4$, we lift the Sasaki-Ricci flow to the foliation minimal
resolution $\widetilde{M_{l}}$ of $M_{l}$ by lifting the Monge--Ampere
equation (\ref{1b}) to (\ref{1d}), we thus obtain uniform estimates\ as in
subsection $5.1.$ (cf Theorem \ref{T51} and Proposition \ref{P51}) so that the
Sasaki-Ricci flow (\ref{1d}) performs a canonical surgical contraction of
floating foliation $(-1)$-curves
\[
\widetilde{\psi_{l}}:\widetilde{M_{l}}\rightarrow\widetilde{N_{l}}%
\]
by Corollary \ref{C61}. Furthermore, the canonical surgical contraction maps
$V_{i}$ onto $V_{i}^{\prime}$ so that
\[
V_{i}^{\prime}\cdot V_{i}^{\prime}=-b_{i}%
\]
for all $i\neq j$ and
\[
V_{j}^{\prime}\cdot V_{j}^{\prime}=-(b_{j}-1).
\]
Finally, by (\ref{2022-7}) and (\ref{2022-6}), it descends to perform
foliation extremal contractions of foliation $K_{M_{l}}^{T}$-negative curves
\[
\psi_{l}:M_{l}\rightarrow M_{l+1}%
\]
with at worst singularity type $\frac{1}{r_{\widetilde{N_{l}}}}%
(1,a_{\widetilde{N_{l}}})$ so that the Hirzebruch Jung continued fraction%
\[
\frac{r_{\widetilde{N_{l}}}}{a_{\widetilde{N_{l}}}}=[b_{1},\cdot\cdot
\cdot,(b_{j}-1),\cdot\cdot\cdot,b_{m}]
\]
and it will be $\frac{r_{\widetilde{N_{l}}}}{a_{\widetilde{N_{l}}}}%
=[b_{1},\cdot\cdot\cdot,b_{j-1},b_{j+1},\cdot\cdot\cdot,b_{m}]$ if
$(b_{j}-1)=1.$

Now take $\widetilde{\psi_{l}}=\widetilde{\psi}$ and $\psi_{l}=\psi$ into
(\ref{2022-8}) with $N=M_{l+1}$ and $\widetilde{N}=\widetilde{N}_{l}$, it
follows that the following diagram is commutative :%
\[%
\begin{array}
[c]{lcl}%
\widetilde{M_{l}}\supset\widetilde{\Gamma},\text{ }V_{j} &
\overset{\widetilde{\psi_{l}}}{\longrightarrow} & \widetilde{N_{l}}\supset
V_{j}^{\prime}\\
\downarrow\varphi & \curvearrowright & \downarrow\widetilde{\varphi_{l}}\\
M_{l}\supset\Gamma,\text{ }\mathbf{S}_{p}^{1} & \overset{\psi_{l}%
}{\longrightarrow} & M_{l+1}.
\end{array}
\]
Then we are done.

Finally, it follows that the divisoral contractions end with either
\[
T_{k}=\infty
\]
and $M_{k}$ is nef which has no foliation $K_{M_{k}}^{T}$-negative curves or
\[
L_{k}^{T}\text{ \ \textrm{is \ not \ big} }%
\]
on $M_{k}$.

Note that similar method can be applied to a finite number of cyclic quotient
foliation singularities.

\textbf{(II) If }$L_{k}^{T}$\textbf{\ is not big : }In this situation, we
have\textbf{\ }
\[
VolM=\int_{M}\omega^{2}(t)\wedge\eta_{0}\rightarrow(c_{1}^{B}(L_{k}^{T}%
))^{2}=0
\]
as $t\rightarrow T_{k}^{-}$ and we can not have a canonical surgical
contraction. However since it is semi-ample, as in Proposition \ref{P41}, we have

(i) there exists a transverse morphism
\[
\phi:M_{k}\rightarrow pt,
\]
then $K_{M_{k}}^{T}<0$ and thus $M_{k}$ is transverse minimal Fano and the
foliation space $M_{k}/\mathcal{F}_{\xi}$ is minimal log del Pezzo surface of
at worst $\frac{1}{r}(1,a)$-type singularities. Or

(ii)
\[
\phi:M_{k}\rightarrow\Sigma_{h},
\]
then $M_{k}$ is an $S^{1}$-orbibundle of a rule surface over \ Riemann
surfaces $\Sigma_{h}$\ of genus $h$.
\end{proof}

\appendix

\section{ \ \ Foliation Minimal Model Program}

\subsection{%
Basic Holomorphic Line Bundles, Basic Divisors on Sasakian
Manifolds
}

For a completeness, we will address basic holomorphic line bundles, basic
divisors over Sasakian manifolds. We refer to \cite{bg}, \cite{m}, and
references therein for some details.

Let $(M,\eta,\xi,\Phi,g)$ be a compact Sasakian $(2n+1)$-manifold $M.$ We
define $D=\ker\eta$ to be the holomorphic contact vector bundle of $TM$ such
that
\[
TM=D\oplus<\xi>=T^{1,0}(M)\oplus T^{0,1}(M)\oplus<\xi>.
\]
Then its associated strictly pseudoconvex CR $(2n+1)$-manifold to be denoted
by $(M,T^{1,0}(M),\xi,\Phi).$

\begin{definition}
( \cite{ta}) Let $(M,T^{1,0}(M))$ be a strictly pseudoconvex CR $(2n+1)$%
-manifold and $E\rightarrow M$ a $C^{\infty}$ complex vector bundle over $M. $
A pair $(E,\overline{\partial}_{b})$ is a CR-holomorphic vector bundle if the
differential operator
\[
\overline{\partial}_{b}:\Gamma^{\infty}(E)\rightarrow\Gamma^{\infty}%
(T^{0,1}(M)^{\ast}\otimes E)
\]
is defined by%
\ 

(i)
\[
\overline{\partial}_{\overline{Z}}(fs)=(\overline{\partial}_{b}f)(\overline
{Z})\otimes s+f\overline{\partial}_{\overline{Z}}s,
\]

(ii)
\[
\overline{\partial}_{\overline{Z}}\overline{\partial}_{\overline{W}%
}s-\overline{\partial}_{\overline{W}}\overline{\partial}_{\overline{Z}%
}s-\overline{\partial}_{[\overline{Z},\overline{W}]}s=0,
\]
for any $f\in C^{\infty}(M)\otimes\mathbf{C}$, $s\in\Gamma^{\infty}(E)$ and
$Z,W\in\Gamma^{\infty}(T^{1,0}(M))$.
\end{definition}

The condition (ii) of the definition means that $(0,2)$-component of the
curvature operator $R(E)$ is vanishing when $E$ admits a connection $D$ whose
$(0,1)$-part is the operator $\overline{\partial}_{b}$ as in the following Lemma.

\begin{lemma}
Let $(M,T^{1,0}(M),\theta)$ be a strictly pseudoconvex CR $(2n+1)$-manifold
and $(E,\overline{\partial}_{b})$ a CR-holomorphic vector bundle over $M$. Let
$h=<,>_{h}$ be a Hermitian structure in $E$. Then there exists a unique
(Tanaka) connection $D$ in $E$ such that

(i)
\[
D_{\overline{Z}}s=(\overline{\partial}_{b}s)\overline{Z},
\]

(ii)
\[
Z<s_{1},s_{2}>_{h}=<D_{Z}s_{1},s_{2}>_{h}+<s_{1},D_{\overline{Z}}s_{2}>_{h},
\]

(iii) The $(0,2)$-component of the curvature operator $\Theta(E)$ is
vanishing. Here $\Theta(E):=D^{2}s.$
\end{lemma}

\begin{definition}
Let $(M,\eta,\xi,\Phi,g)$ be a compact Sasakian $(2n+1)$-manifold. A
CR-holomorphic vector bundle $(E,\overline{\partial}_{b})$ over $M$ is a
\textbf{basic transverse holomorphic }vector bundle over $(M,T^{1,0}(M))$ if
\ there exists an open cover $\{U_{\alpha}\}$ of $M$ and the trivializing
frames on $U_{\alpha}$ , such that its transition functions are matrix-valued
basic CR functions. The trivializing frames is called the basic tansverse
holomorphic frame.
\end{definition}

\begin{example}
Let $(M,\eta,\xi,\Phi,g)$ be a compact Sasakian $(2n+1)$-manifold. Then, with
respect to the trivializing frames
\[
\{\frac{\partial}{\partial x},Z_{j}=\left(  \frac{\partial}{\partial z^{j}%
}+ih_{j}\frac{\partial}{\partial x}\right)  ,\ \ \ j=1,2,\cdot\cdot\cdot,n\}
\]
the transition functions of such frames are basic transverse holomorphic
functions, that is $h$ is basic. Thus $T^{1,0}(M)$ is a basic transverse
holomorphic vector bundle. Moreover, the canonical (determinant) bundle
$K_{M}^{T}$ of $T^{1,0}(M)$ is a basic transverse holomorphic line bundle
whose transition functions are given by $t_{\alpha\beta}=\det(\partial
z_{\beta}^{i}/\partial z_{\alpha}^{j})$ on $U_{\alpha}\cap U_{\beta}$, where
$(x,z_{\alpha}^{1},...,z_{\alpha}^{n})$ is the normal coordinate on
$U_{\alpha}.$
\end{example}

\begin{definition}
(i) Let $(M,\eta,\xi,\Phi,g)$ be a Sasakian $(2n+1)$-manifold and $L$ be a
basic transverse holomorphic bundle over $M$. A basic transverse holomorphic
section $s$ of $L$ is a collection $\{s_{\alpha}\}$ of CR-holomorphic maps
$s_{\alpha}:U_{\alpha}\rightarrow\mathbf{C}$ satisfying the transformation
rule $s_{\alpha}=t_{\alpha\beta}s_{\beta}$ on $U_{\alpha}\cap U_{\beta}$. The
transition function $t_{\alpha\beta}$ is \textbf{basic}. A basic Hermitian
metric $h$ on $L$ is a collection $\{h_{\alpha}\}$ of smooth positive
functions $h_{\alpha}:U_{\alpha}\rightarrow\mathbf{R}$ satisfying the
transformation rule
\[
h_{\alpha}=|t_{\beta\alpha}|^{2}h_{\beta}%
\]
on $U_{\alpha}\cap U_{\beta}$. Given a basic transverse holomorphic section
$s$ and a Hermitian metric $h$, we can define the pointwise norm squared of
$s$ with respect to $h$ by
\[
|s|_{h}^{2}=h_{\alpha}s_{\alpha}\overline{s_{\alpha}}%
\]
on $U_{\alpha}$. The reader can check that $|s|_{h}^{2}$ is a well-defined
function on $M$.

(ii) A Hermitian metric is called a \textbf{basic hermitian metric} if
$h_{\alpha}$ is \textbf{basic. It always exists if }$L$ is a basic transverse
holomorphic line bundle.

(iii) We define the curvature $R_{h}^{T}$ of a basic Hermitian metric $h$ on
$L$ to be the basic closed $(1,1)$-form on $M$ given by
\[
R_{h}^{T}=-\frac{\sqrt{-1}}{2\pi}\partial_{B}\overline{\partial}_{B}\log
h_{\alpha}%
\]
on $U_{\alpha}$. This is well-defined. The basic first Chern class $c_{1}%
^{B}(L)$ of $L$ to be the cohomology class $[R_{h}^{T}]_{B}\in H_{\overline
{\partial}_{B}}^{1,1}(M,R)$. Since any two basic Hermitian metrics
$h,h^{\prime}$ on $L$ are related by $h^{\prime}=e^{-\phi}h$ for some smooth
\textbf{basic} function $\phi$, we see that $R_{h^{\prime}}^{T}=R_{h}%
^{T}+\frac{\sqrt{-1}}{2\pi}\partial_{B}\overline{\partial}_{B}\phi$ and hence
$c_{1}^{B}(L)$ is well-defined,independent of choice of basic Hermitian metric
$h.$ We say that $(L,h)$ is positive if the curvature $R_{h}^{T}$ is positive
definite at every $p\in M.$
\end{definition}

\begin{example}
Let $(M,\eta,\xi,\Phi,g)$ be a compact Sasakian $(2n+1)$-manifold. If $g^{T}$
is a transverse Kaehler metric on $M,$ then $h_{\alpha}=\det((g_{i\overline
{j}}^{\alpha})^{T})$ on $U_{\alpha}$ defines a basic Hermitian metric on the
canonical bundle $K_{M}^{T}$. The inverse $(K_{M}^{T})^{-1}$ of $K_{M}^{T}$ is
sometimes called the anti-canonical bundle. Its basic first Chern class
$c_{1}^{B}((K_{M}^{T})^{-1})$ is called the basic first Chern class of $M$ and
is often denoted by $c_{1}^{B}(M).$Then it follows from the previous result
that $c_{1}^{B}(M)$ that $c_{1}^{B}(M)=[Ric^{T}(\omega)]_{B}$ for any
transversal Kaehler metric $\omega$ on a Sasakian manifold $M$.
\end{example}

\begin{definition}
(i) Let $(L,h)$ be a basic transverse holomorphic line bundle over a Sasakian
manifold $(M,\eta,\xi,\Phi,g)$ with the basic hermitian metric $h.$ We say
that $L$ is \textbf{very ample} if for any ordered basis $\underline{s}%
=(s_{0},...,s_{N})$ of $H_{B}^{0}(M,L)$, the map $i_{\underline{s}%
}:M\rightarrow\mathbf{CP}^{N}$ given by%
\[
i_{\underline{s}}(x)=[s_{0}(x),...,s_{N}(x)]
\]
is well-defined and an embedding which is $S^{1}$-equivariant with respect to
the weighted $\mathbf{C}^{\ast}$action in $\mathbf{C}^{N+1}$ as long as not
all the $s_{i}(x)$ vanish. We say that $L$ is ample if there exists a positive
integer $m_{0}$ such that $L^{m}$ is very ample for all $m\geq m_{0}.$

(ii) $L$ is a \textbf{semi-ample} basic transverse holomorphic line bundle if
there exists a basic Hermitian metric $h$ on $L$ such that $R_{h}^{T}$ is a
nonnegative $(1,1)$-form. In fact, there exists a foliation basepoint-free
holomorphic map
\[
\Psi:M\rightarrow(\mathbf{CP}^{N},\omega_{FS})
\]
defined by the basic transverse holomorphic section $\{s_{0},s_{1},...s_{N}\}
$ of $H_{B}^{0}(M.L^{m})$ which is $S^{1}$-equivariant with respect to the
weighted $\mathbf{C}^{\ast}$action in $\mathbf{C}^{N+1}$ with $N=\dim
H^{0}(M.L^{m})-1$ for a large positive integer $m$ and
\[
0\leq\frac{1}{m}\Psi^{\ast}(\omega_{FS})=\widehat{\omega}_{\infty}\in
c_{1}^{B}(L).
\]

\end{definition}

There is a Sasakian analogue of
Kodaira embedding theorem on 
a compact quasi-regular Sasakian $(2n+1)$-manifold
due to
\cite{rt}, and \cite{hlm} :

\begin{proposition}
Let $(M,\eta,\xi,\Phi,g)$ be a compact quasi-regular Sasakian $(2n+1)$%
-manifold and $(L,h)$ be a basic transverse holomorphic line bundle over $M$
with the basic hermitian metric $h.$ Then $L$ is ample if and only if $L$ is positive.
\end{proposition}

\begin{definition}
(i) First we say that a subset $V$ of a (quasi-regular) Sasakian manifold
$(M^{2n+1},\eta,\xi,\Phi,g)$ an \textbf{invariant (Sasakian) submanifold (with
or without singularities) }of dimension $2n-1$. if $\xi$ is tangent to $V$ and
$\Phi TV\subset TV$ at all points of $V$ and is locally given as the zero set
$\{f=0\}$ of a locally defined \textbf{basic} CR holomorphic function $f$.
\ In general, $V$ may not be a submanifold. Denote by $V^{reg}$ the set of
points $p\in V$ for which $V$ is a submanifold of $M$ near $p$. We say that
$V$ is irreducible if $V^{reg}$ is connected. A transverse divisor $D^{T}$ on
$M$ is a formal finite sum $\sum_{i}a_{i}V_{i}$ where $a_{i}\in\mathbf{Z}$ and
each $V_{i}$ is an irreducible \textbf{invariant submanifold} of dimension
$2n-1$. We say that $D^{T}$ is effective if the $a_{i}$ are all nonnegative.
The support of $D^{T}$ is the union of the $V_{i} $ for each $i$ with
$a_{i}\neq0.$

(ii) Given a transverse divisor $D^{T},$ we define an associated line bundle
as follows. Suppose that $D^{T}$ is given by local defining basic functions
$f_{\alpha}$ (vanishing on $D^{T}$ to order $1$) over an open cover
$U_{\alpha}$. Define transition functions $f_{\alpha}=t_{\alpha\beta}f_{\beta
}$ on $U_{\alpha}\cap U_{\beta}$. These are basic CR holomorphic and
nonvanishing in $U_{\alpha}\cap U_{\beta}$, and satisfy
\[
t_{\alpha\beta}t_{\beta\alpha}=1;t_{\alpha\beta}t_{\beta\gamma}=t_{\alpha
\gamma}.
\]
Write $[D^{T}]$ for the associated basic line bundle, which is well-defined
independent of choice of local defining functions.

(iii) One can define
\begin{equation}
L_{M}\cdot V=\int_{V}R_{h}^{T}\wedge\eta\label{A2}%
\end{equation}
for all invariant Sasakian $3$-manifold $V$ in $M.$ $h$ is a \textbf{basic}
Hermitian metric on the basic line bundle $L_{M}$ \ From (ii), for a compact
Sasakian $5$-manifold $M,$ a transverse divisor $D^{T}$ defines an element of
$H_{\overline{\partial}_{B}}^{1,1}(M,R)$ by $D^{T}\rightarrow\lbrack R_{h}%
^{T}]\in H_{\overline{\partial}_{B}}^{1,1}(M,R)$ for a basic Hermitian metric
on the associate basic line bundle $[D^{T}]$, and we define
\[
\alpha\cdot\beta=\int_{M}\alpha\wedge\beta\wedge\eta
\]
for $\alpha,\beta\in H_{\overline{\partial}_{B}}^{1,1}(M,R).$ Then for an
invariant $3$-manifold $V_{i}$ which is both a foliation curve and a
transverse divisor, the $V\cdot V$ is well-defined and we may write $V^{2}$
instead of $V\cdot V.$
\end{definition}

\begin{remark}
(\cite{gei}) Note that the Sasakain $3$-manifold $V$ is either canonical,
anticanonical or null. $V$ is up to finite quotient a regular Sasakian
$3$-manifold, i.e., a circle bundle over a Riemann surface of positive genus.
In the positive case, $V$ is covered by $S^{3}$ and its Sasakian structure is
a deformation of a standard Sasakian structure.
\end{remark}

\begin{definition}
(i) We say that a \textbf{basic} line bundle $L$ is \textbf{nef} if $L\cdot
V\geq0$ for any invariant Sasakian $3$-manifold $V$ in $M.$ In particular if
$M$ is quasi-regular, then $V$ is the $\mathbf{S}^{1}$-oribundle over the
curve $C$ in $Z$ so that
\begin{equation}
L_{Z}\cdot C=\int_{C}R_{h_{Z}}\geq0.\label{A1}%
\end{equation}
Here $c_{1}^{B}(L_{M})=\pi^{\ast}c_{1}^{orb}(L_{Z})$ and $h_{Z}$ is the
hermitian metric in the corresponding line bundle $L_{Z}.$ Define
\begin{equation}
C_{M}^{B}=\{[\alpha]_{B}\in H_{B}^{1,1}(M,\mathbf{R})|\text{ }\exists\text{
}\omega>0\text{ such that }[\omega]_{B}=[\alpha]_{B}\}.\text{ }\label{D}%
\end{equation}
Then we can also define a class $[\alpha]_{B}$ called nef class if
$[\alpha]_{B}\in\overline{C_{M}^{B}}$ and a class $[\alpha]_{B}$ called big
if
\[
\int_{M}\alpha^{n}\wedge\eta>0.
\]

(ii) If the Sasakian manifold $(M^{2n+1},\eta,\xi,\Phi,g)$ has the canonical
basic line bundle $K_{M}^{T}$ nef, then we say that $M$ is a smooth transverse
\textbf{minimal model}. If $M$ has $K_{M}^{T}$ big, then we say that $M$ is of
general type.
\end{definition}

\subsection{Minimal Model Program on Compact Quasi-Regular Sasakian
$5$-Manifolds}

We will focus on the proof of foliation minimal model program on a compact
\textbf{quasi-regular} Sasakian $5$-manifold with the foliation singularitie
of type $\frac{1}{r}(1,a).$

Let $(M,\eta,\xi,\Phi,g)$ be a compact quasi-regular Sasakian $5$-manifold and
the space $Z$ of leaves be normal projective orbifold with finite cyclic
quotient \textbf{singularities of type }$\frac{1}{r}(1,a)$ which are klt
singularities. It follows from the first strcture theorem for Sasakian
structures that
\[
K_{M}^{T}=\pi^{\ast}(K_{\emph{Z}}^{orb}).
\]
On the other hand, the leave space $Z$ is well-formed and the orbifold
canonical divisor $K_{\emph{Z}}^{orb}$ and canonical divisor $K_{Z}$ are the
same, then via the $S^{1}$-orbibundle
\[
\pi:(M,g)\rightarrow(Z,\omega),
\]
we have the following Sasaki analogue of basepoint-free theorem, rationality
theorem, cone and contraction theorem (\cite{km}, \cite{kmm}, \cite{m}) :

\begin{proposition}
(Foliation Base-point free Theorem) Let $(M,\eta,\xi,\Phi,g)$ be a compact
quasi-regular Sasakian $5$-manifold and the space $Z$ of leaves be normal
projective orbifold with finite cyclic quotient \textbf{singularities of type
}$\frac{1}{r}(1,a)$ which are klt singularities. Suppose that $L^{T}$ is a nef
Cartier basic line bundle and $aL^{T}-K_{M}^{T}$ is nef and big Cartier basic
line bundle for some $a\in\mathbf{N}$. Then $|mL^{T}|$ is transverse
basepoint-free for $m>>0.$ More precisely, there exists a $S^{1}$-equivariant
foliation basepoint-free holomorphic map%
\[
\Psi_{|mL^{T}|}:M\rightarrow\mathcal{P(}H_{B}^{0}(M,(L^{T})^{m}))
\]
which is $S^{1}$-equivariant with respect to the weighted $\mathbf{C}^{\ast}
$action in $\mathbf{C}^{N+1}$ with $N=\dim H^{0}(M.L^{m})-1.$
\end{proposition}

\begin{proof}
Note that the leave space $Z$ is a normal projective orbifold surface with
finite cyclic quotient \textbf{singularities of type }$\frac{1}{r}(1,a)$ which
are klt, Suppose that $L^{T}$ is a nef Cartier basic line bundle and
$aL^{T}-K_{M}^{T}$ is nef and big Cartier basic line bundle for some
$a\in\mathbf{N}$. On the other hand by applying Proposition \ref{P21}, there
exists a Riemannian submersion, $S^{1}$-orbibundle $\pi:M\rightarrow Z,$ such
that
\[
K_{M}^{T}=\pi^{\ast}(K_{\emph{Z}}^{orb})=\pi^{\ast}(K_{\emph{Z}})
\]
and
\[
\pi^{\ast}(L)=L^{T}.
\]
Then $L$ is a nef Cartier line bundle and $aL-K_{Z}$ is nef and big Cartier
line bundle. Therefore by Kawamata base-point free theorem, there is a
basepoint-free holomorphic map
\[
\psi_{|mL|}:Z\rightarrow\mathcal{P(}H^{0}(Z,L^{m})).
\]
Define
\[
\Psi_{|mL^{T}|}=\psi_{|mL|}\circ\pi
\]
such that
\[
\Psi_{|mL^{T}|}:M\rightarrow\mathcal{P(}H_{B}^{0}(M,(L^{T})^{m})).
\]
It follows that $\Psi_{|mL^{T}|}$ is a $S^{1}$-equivariant foliation
basepoint-free holomorphic map with respect to the weighted $\mathbf{C}^{\ast
}$action in $\mathbf{C}^{N+1}$ with $N=\dim H^{0}(M.L^{m})-1$ for a large
positive integer $m$.
\end{proof}

\begin{proposition}
(Rationality Theorem) Let $(M,\eta,\xi,\Phi,g)$ be a compact quasi-regular
Sasakian $5$-manifold and the space $Z$ of leaves be normal projective
orbifold with finite cyclic quotient \textbf{singularities of type }$\frac
{1}{r}(1,a)$ which are klt singularities. Suppose that $K_{M}^{T} $ is not
nef. Let $a(Z)$ be an integer such that $a(Z)K_{M}^{T}$ is Cartier. Let
$H^{T}$ be a nef and big Cartier basic divisor and define
\[
r=\sup\{t\in\mathbf{R:}H^{T}+tK_{M}^{T}\text{ \ \ \textrm{is nef} }\}.
\]
Then $r$ is a rational number of the form $\frac{p}{q}$ with
\[
0<q\leq a(Z)(\dim Z+1).
\]

\end{proposition}

A basic $1$-cycle $V$ on $M$ \ is a formal finite sum $V=\sum a_{i}V_{i}$, for
$a_{i}\in\mathbf{Z}$ and $V_{i}$ is the irreducible invariant Sasakian
$3$-manifold. We denote by $N_{1}(M)_{\mathbf{Z}}$ the space of $1$-cycles
modulo numerical equivalence. Write%
\[
N_{1}(M)_{\mathbf{Q}}=N_{1}(M)_{\mathbf{Z}}\otimes_{\mathbf{Z}}\mathbf{Q}%
\text{ \ \ \ \textrm{and} \ \ \ }N_{1}(M)_{\mathbf{R}}=N_{1}(M)_{\mathbf{Z}%
}\otimes_{\mathbf{Z}}\mathbf{R.}%
\]
Then write $NE(M)$ for the cone of effective elements of $N_{1}(M)_{\mathbf{R}%
}$ and $\overline{NE(M)}$ for its closure. Furthermore, a basic divisor
$D^{T}$ is ample if and only if
\[
D^{T}\cdot V>0
\]
for all nonzero $V\in\overline{NE(M)}$. It is the Kleiman criterion for the
ample line bundle. Also we define the Picard number%
\[
\rho(M):=\dim N_{\mathbf{R}}^{1}(M)\leq\dim H_{B}^{2}(M,\mathbf{R})<\infty.
\]

\begin{proposition}
\label{P31}(Foliation Contraction Theorem) Let $(M,\eta,\xi,\Phi,g)$ be a
compact quasi-regular Sasakian $5$-manifold and the space $Z$ of leaves be
normal projective orbifold surface with finite cyclic quotient
\textbf{singularities of type }$\frac{1}{r}(1,a)$ which are klt singularities.
Then there exists a countable collection of Sasakain $3$-spheres $\{V_{i}\}$
such that%
\[
0<-K_{M}^{T}\cdot V_{i}\leq2\dim Z
\]
and
\[
\overline{NE(M)}=\overline{NE(M)}_{K_{M}^{T}\geq0}+\sum a_{i}|V_{i}|,\text{
\ \ }a_{i}\geq0.
\]
The rays $a_{i}|V_{i}|$ are locally discrete in the half space $\{K_{M}%
^{T}<0\}.$ If $R\in\overline{NE(M)}$ is a $K_{M}^{T}$-negative foliation
extremal ray such that
\[
R=\overline{NE(M)}_{K_{M}^{T}<0}\cap(L^{T})^{\perp}%
\]
for some nef basic line bundle $L^{T}$ which can be chosed by
\[
\pi^{\ast}L=L^{T}%
\]
for a nef line bundle $L$ over $Z$. Then there is a unique foliation extremal
ray contraction
\[
\psi_{(L^{T})^{m}}=cont_{V}:M\rightarrow N
\]
for some $m>>1$ such that an irreducible foliation curve $V\subset M$ is
mapped to a leave $\mathbf{S}^{1}$ by $\psi$ if and only if $[V]_{B}\in R$.
Furthermore, $L^{T}=\psi^{\ast}A$ for some basic ample line bundle on $N.$
\end{proposition}

\begin{proof}
Since $Z$ is a normal projective orbifold surface with finite cyclic quotient
\textbf{singularities of type }$\frac{1}{r}(1,a)$ which are klt singularities,
it follows (\cite{km}) that there exists a nef and semi-ample line bundle $L$
over $Z$ such that if $R_{Z}\in\overline{NE()}$ is a $K_{Z}$-negative extremal
ray with
\[
R_{Z}=\overline{NE(Z)}_{K_{Z}<0}\cap(L)^{\perp}.
\]
Then there is a unique extremal ray contraction
\[
\psi_{(L)^{m}}=cont_{C}:Z\rightarrow Y\subset\mathcal{P(}H^{0}(Z,mL))
\]
for some $m>>1$ such that an irreducible curve $C\subset Z$ is mapped to a
point by $\psi_{(L)^{m}}$ if and only if $[C]\in R_{Z}$.

Now for a Riemannian submersion, $S^{1}$-orbibundle $\pi:M\rightarrow Z,$ we
choose $L^{T}$ such that
\[
\pi^{\ast}L=L^{T}%
\]
ann thus $L^{T}$ is a basic nef and semi-ample line bundle over $M$ such that
there is a unique $S^{1}$-equivariant foliation extremal ray contraction
\[
\psi_{(L^{T})^{m}}=cont_{V}:M\rightarrow N\subset\mathcal{P(}H_{B}^{0}%
(M_{k},(L^{T})^{m}))
\]
with%
\[
\psi_{(L^{T})^{m}}=\psi_{(L)^{m}}\circ\pi
\]
for which an irreducible foliation curve $V\subset M$ is mapped to a leave
$\mathbf{S}^{1}$ by $\psi_{(L^{T})^{m}}$ if and only if
\[
\lbrack V]_{B}\in R=\overline{NE(M)}_{K_{M}^{T}<0}\cap(L^{T})^{\perp}.
\]

\end{proof}

Now we are ready to prove the following foliation minimal model program :

\begin{proposition}
\label{P41} Let $(M,\eta,\xi,\Phi,g)$ be a compact quasi-regular Sasakian
$5$-manifold with finite cyclic quotient foliation singularities of type
$\frac{1}{r}(1,a).$ Then there exists a finite sequence of foliation extremal
ray contractions
\[
\psi_{i}:M_{i-1}\rightarrow M_{i}\text{ , \ }i=1,...k
\]
such that every $M_{i}$ is a Sasakian manifold having at worst foliation
cyclic quotient singularities and for every $\psi_{i}$ one of the following
holds :

\begin{enumerate}
\item Foliation divisorial contraction : (The locus of the foliation extremal
ray is an irreducible basic divisor) : $\psi_{i}$ is a foliation divisorial
contraction of a foliation curve $V$ with $V^{2}<0$, and the Picard number
satisfies $\rho(M_{i})=\rho(M_{i-1})-1.$ or

\item Foliation fibre contraction (transverse Mori fibre space) : (The locus
of the foliation extremal ray is $M_{i-1}$) : $\psi_{i}$ is a singular
fibration such that either

\begin{enumerate}
\item there is a map%
\[
f:M_{k}\rightarrow pt,
\]
then $K_{M_{k}}^{T}<0$ and thus $M_{k}$ is transverse minimal Fano and the
leave space $Z_{k}$ is minimal log del Pezzo surface of $\frac{1}{r}%
(1,a)$-type singularities and Picard number one. Or

\item
\[
f:M_{k}\rightarrow\Sigma_{h},
\]
then $Z_{k}$ is a rule surface over \ Riemann surfaces $\Sigma_{h}$\ of genus
$h$ and thus $M_{k}$ is an $S^{1}$-orbibundle over $Z_{k}.$ Or
\end{enumerate}

\item $M_{k}$ is nef : $\psi_{i}=\psi_{k}$, and $M_{k}$ has at worst foliation
cyclic quotient singularities and has no foliation $K_{M}^{T}$-negative curves.
\end{enumerate}
\end{proposition}

\begin{proof}
Let $\mathbf{S}_{p}^{1}$ be the singular fibre of type $\frac{1}{r}(1,a),p\in
M.$ Let $\psi:M\rightarrow N$ be a proper transverse birational morphism. By
applying Theorem \ref{T32} and \cite{cu}, we consider the minimal resolution
of foliation singularities of $M$ to be $\varphi:\overline{M}\rightarrow M$
and $\overline{\varphi}:\overline{N}\rightarrow N $ such that the following
diagram is commutative :%
\[%
\begin{array}
[c]{lcl}%
(\overline{M},\overline{\Gamma}) & \overset{\overline{\psi}}{\longrightarrow}
& \overline{N}\\
\downarrow\varphi & \circlearrowright & \downarrow\overline{\varphi}\\
(M,\Gamma,\mathbf{S}_{p}^{1}) & \overset{\psi}{\longrightarrow} &
(N,\mathbf{S}_{q}^{1}).
\end{array}
\]
Since $\overline{M}$ and $\overline{N}$ are regular Sasakian manifolds, by
Theorem \ref{T33}, the transverse birational morphism $\overline{\psi}$ can be
factored into a sequence of ordinary blow ups. Suppose $\overline{M}$ and
$\overline{N}$ are not isomorphic, there must exist a foliation $(-1)$-curve
$\overline{\Gamma}$ such that $\overline{\varphi}\circ\overline{\psi
}(\overline{\Gamma})=\mathbf{S}_{q}^{1},$ for $q\in N.$ Let $\Theta$ be a set
of foliation curves such that $\Theta=$ $\overline{\psi}\circ\overline
{\varphi}^{-1}(\mathbf{S}_{q}^{1})$ and $D^{T}=$ $\{V_{i}\}$ be the
exceptional locus of $\varphi$ at the singular fibre $\mathbf{S}_{p}^{1}.$
Then the Hirzebruch Jung continued fraction
\[
\frac{r}{a}=[b_{1},\cdot\cdot\cdot,b_{l}]
\]
say that
\[
V_{i}^{2}=-b_{i}.
\]

(i) If $V_{i}\cdot\overline{\Gamma}=\emptyset$, for all $i,$ then
$\overline{\Gamma}$ is not exceptional for $\varphi$ and its birational
transform $\Gamma$ is a floating foliation $(-1)$-curve on $M$. \ Then by
Theorem \ref{T33}, we obtain a transverse birational morphism $\psi^{\prime
}:M\rightarrow M_{1}$ corresponding to the blow up of a smooth point with
exceptional locus given by $\Gamma.$ Thus $\psi$ factors through $\psi
^{\prime}$ and we obtain a transverse birational morphism $M_{1}\rightarrow
N.$

(ii) If $V_{j}\cdot\overline{\Gamma}\neq\emptyset$ for some $j$ ( at most one
such a foliation exceptional curve) , then $\ V_{j}^{2}=-b_{j}$ and $\Gamma$
is passing the singular fibre $\mathbf{S}_{p}^{1}$ of the foliation
singularities of type $\frac{1}{r}(1,a)$ and $K_{M}^{T}\cdot\Gamma<0.$
Moreover, $\ \overline{\Gamma}$ is also a foliation $(-1)$-curve on
$\overline{M}$. Again by Theorem \ref{T33}, we obtain a transverse birational
morphism $\psi^{\prime\prime}:M\rightarrow M_{1}$ corresponding to the
transverse extremal ray contraction of $\Gamma$ into a singular fibre of
foliation singularities of at worst type $\frac{1}{r_{M_{1}}}(1,a_{M_{1}})$
such that
\begin{equation}
\frac{r_{M_{1}}}{a_{M_{1}}}=[b_{1},\cdot\cdot\cdot,(b_{i}-1),\cdot\cdot
\cdot,b_{l}].\label{2022-6}%
\end{equation}
Again $\psi$ factors through $\psi^{\prime}$ and we obtain a transverse
birational morphism $M_{1}\rightarrow N.$ Similar method can be applied to a
finite number of cyclic quotient foliation singularities.

On the other hand, since for divisoral extremal contraction, we have
$\rho(M_{i})=\rho(M_{i-1})-1,$ then after a finite number of foliation
extremal ray contractions, we end up
\[
M=M_{0}\overset{\psi_{1}}{\rightarrow}M_{1}\overset{\psi_{2}}{\rightarrow
}M_{2}\overset{\psi_{3}}{\rightarrow}\cdot\cdot\cdot\overset{\psi
_{k}}{\rightarrow}M_{k}%
\]

with $M_{k}$ is nef or

by \cite[Lemma 2.1.27]{l} and $L^{T\text{ }}$is semi-ample over $M_{k}$ (see
Proposition \ref{P31}) to see that there exist a morphism
\[
\lambda:Z_{k}\rightarrow\mathcal{P(}H^{0}(Z_{k},mL))
\]
such that $f=\lambda\circ\pi$ is a $S^{1}$-equivariant transverse morphism
\[
f:M_{k}\rightarrow\mathcal{P(}H_{B}^{0}(M_{k},mL^{T})),
\]
with $L^{\text{ }}$is semi-ample over $Z_{k}$ for $\pi^{\ast}L=L^{T}$.

Then $Y:=h(Z_{k})$ is either a point or a nonsingular curve.

(a). If $Y$ is a point, then $K^{orb}(Z_{k})<0$ and $K_{M_{k}}^{T}<0$. Hence
$M_{k}$ is transverse Fano.

(b). If $Y=\Sigma_{h}$ is a non-singular curve, then a generic fiber $C$ of
$\lambda$ is a smooth curve and then foliation curve $V$ over $C$ satisfying
$L_{k}\cdot C=0$ and $K_{Z_{k}}^{orb}\cdot C<0$ as well as $L_{k}^{T}\cdot
V=0$ and $K_{M_{k}}^{T}\cdot V<0$. It follows from the adjunction formula that
$C$ is a rational curve and thus $V$ is covered by the three sphere
$\mathbf{S}^{3}$. Hence $Z_{k}$ is a rule surface over \ Riemann surfaces
$\Sigma_{h}$\ of genus $h$ and thus $M_{k}$ is an $S^{1}$-orbibundle over
$Z_{k}.$
\end{proof}

\begin{remark}
Let $(Z,\emptyset)$ be a log del Pezzo orbifold surface with no branch
divisors (well-formed). It follows from \cite[Theorem 4.7.14]{bg} that the
total space of an $S^{1}$-orbibundle over $Z$ is diffeomorphic to some
$k(\mathbf{S}^{2}\mathbf{\times S}^{3})$ for some $k=0,1,...$ ($k=0$ means
$\mathbf{S}^{5}$).
\end{remark}

\begin{corollary}
Let $(M,\eta,\xi,\Phi,g)$ be a compact \textbf{regular} Sasakian $5$-manifold.
Then there exists a transverse morphism $\psi:M\rightarrow M^{\prime},$ which
is a composition of blowing down foliation $(-1)$-curves and a morphism
$\varphi:M^{\prime}\rightarrow N$ such that one of the following holds :

\begin{enumerate}
\item Nef : $M^{\prime}\simeq N$ is a compact regular Sasakian $5$-manifold
with $K_{N}^{T}$ nef;

\item Mori fibre space :

\begin{enumerate}
\item $N$ is a Riemann surface $\Sigma_{h}$ and $M^{\prime}$ is $\Sigma
_{h}\mathbf{\times S}^{3}$ or $X_{\infty,h}=\Sigma_{h}\widetilde{\times
}\mathbf{S}^{3}$ or

\item $N$ is a point and $M^{\prime}$ is isomorphic to $\mathbf{S}^{5}$.
\end{enumerate}
\end{enumerate}
\end{corollary}

As a consequence, we have

\begin{corollary}
Let $(M,\eta,\xi,\Phi,g)$ be a compact simply connected \textbf{regular}
Sasakian $5$-manifold.

\begin{enumerate}
\item If the Sasakian structure is positive, then $M=\mathbf{S}^{5}$ $(k=0)$;
or $k(\mathbf{S}^{2}\mathbf{\times S}^{3})$ or $M=X_{\infty,0}%
\#(k-1)(\mathbf{S}^{2}\mathbf{\times S}^{3}),1\leq k\leq8.$

\item If the Sasakian structure is indefinite ($K_{M}^{T}$ is nef), then
$M=k(\mathbf{S}^{2}\mathbf{\times S}^{3})$; or $M=X_{\infty,0}%
\#(k-1)(\mathbf{S}^{2}\mathbf{\times S}^{3}),1\leq k$.
\end{enumerate}
\end{corollary}


\begin{thebibliography}{9999}                                                                                             %
\bibitem[B]{b}D. Barden, \textit{Simply connected five-manifolds}, Ann. of
Math. (2) 82 (1965), 365-385.

\bibitem[BCHM]{bchm}C. Birkar, P. Cascini, C. D. Hacon, and J. McKernan,
\textit{Existence of minimal models for varieties of log general type}, J.
Amer. Math. Soc. 23 no. 2(2010), 405-468.

\bibitem[BG]{bg}C. P. Boyer and K. Galicki, \textit{Sasaki geometry}. Oxford
Mathematical Monographs. Oxford University Press, Oxford (2008).

\bibitem[Bir1]{bir1}C. Birkar, \textit{Lecture on birational geometry}, arXiv:1210.2670.

\bibitem[Bir2]{bir2}C. Birkar, \textit{Singularities of linearsystems and
boundedness of Fano varieties}, Annals of Mathematics 193 (2021), 347-405.

\bibitem[Cao]{cao}H. Cao, \textit{Deformation of K\"{a}hler metrics to
K\"{a}hler--Einstein metrics on compact K\"{a}hler manifolds}, Invent. Math.
81(1985), 359--372.

\bibitem[CHLW]{chlw}S.-C. Chang, Y. Han, C. Lin and C.-T. Wu, Convergence of
the Sasaki-Ricci Flow on Sasakian $5$-Manifolds of General Type, preprint.

\bibitem[Co]{co}T. Collins, \textit{The transverse entropy functional and the
Sasaki-Ricci flow}, Trans. AMS., Volume 365, Number 3, March 2013, Pages 1277-1303.

\bibitem[CT]{ct}T. Collins and V. Tosatti, \textit{K\"{a}hler currents and
null loci,} Invent. math. (2015) 202, 1167-1198.

\bibitem[Cu]{cu}Alice Cuzzucoli, \textit{On the classification of orbifold del
Pezzo surfaces}, Thesis submitted to the University of Warwick for the degree
of Doctor of Philosophy, 2020.

\bibitem[DP]{dp}J. P. Demailly and M. Paun, \textit{Numerical characterization
of the K\"{a}hler cone of a compact K\"{a}hler manifold}, Annals of
Mathematics, 159 (2004), 1247-1274.

\bibitem[EKA]{eka}A. El Kacimi-Alaoui, Operateurs transversalement elliptiques
sur un feuilletage riemannien et applications, Compos. Math. 79 (1990) 57--106.

\bibitem[FOW]{fow}A. Futaki, H. Ono and G.Wang, \textit{Transverse K\"{a}hler
geometry of Sasaki manifolds and toric Sasaki--Einstein manifolds}, J.
Differential Geom. 83 (2009) 585--635.

\bibitem[Gei]{gei}H. Geiges, Normal Contact Structures on $3$-manifolds,
T\'{o}hoku Math. J. 49 (1997), 415-422.

\bibitem[GKN]{gkn}M. Godlinski, W. Kopczynski and P. Nurowski, \textit{Locally
Sasakian manifolds}, Classical Quantum Gravity 17 (2000) L105--L115.

\bibitem[GSW]{gsw}B. Guo, J. Song and B. Weinkovw\textit{, Geometric
convergence of the K\"{a}hler-Ricci flow on complex surfaces of general type},
IMRN 2016, No. 18, 5652-5669.

\bibitem[H]{h}R. S. Hamilton, \textit{Three-manifolds with positive Ricci
curvature}, J. Differential Geom. 17 (1982), no. 2, 255--306.

\bibitem[He]{he}W. He, \textit{The Sasaki-Ricci flow and compact Sasaki
manifolds of positive transverse holomorphic bisectional curvature}, J. Geom.
Anal. (2013) 23:1876-931.

\bibitem[HLM]{hlm}C.-Y. Hsiao, X. Li and G. Marinescu, \textit{Equivariant
Kodaira embedding for CR manifolds with circle action}, Michigan Math. J. 70
(2021), no. 1, 55--113.

\bibitem[K1]{k1}J. Kollar, \textit{Singular of pairs}, Algebraic geometry%
$\vert$
Santa Cruz 1995, Proc. Sympos. Pure Math., vol. 62, Amer. Math. Soc.,
Providence, RI, 1997, pp. 221-287.

\bibitem[K2]{k2}J. Kollar,\textit{\ Einstein metrics on connected sums of
}$S^{2}\times S^{3},$ J. Differential Geom. 75 (2007), no. 2, 259--272.

\bibitem[K3]{k3}Kollar, \textit{Einstein metrics on five-dimensional Seifert
bundles}, J. Geom. Anal. 15 (2005), no. 3, 445--476.

\bibitem[KM]{km}J. Kollar and S. Mori, \textit{Birational geometry of
algebraic varieties}. With the collaboration of C. H. Clemens and A. Corti.
Translated from the 1998 Japanese original. Cambridge Tracts in Mathematics,
134. Cambridge University Press, Cambridge, 1998. viii+254 pp.

\bibitem[KMM]{kmm}Y. Kawamata, K. Matsuda and K. Matsuki, \textit{Introduction
to the minimal model problem, }Algebraic geometry, Sendai, 1985, 283--360,
Adv. Stud. Pure Math., 10, North-Holland, Amsterdam, 1987.

\bibitem[Kol1]{kol1}S. Kolodziej, The complex Monge-Amp\`{e}re equation, Acta
Math. 180 (1998), no. 1, 69--117.

\bibitem[Kol2]{kol2}S. Kolodziej, The Monge-Amp\`{e}re equation on compact
Kahler manifolds, Indiana Univ. Math. J. 52 (2003), no. 3, 667--686.

\bibitem[Kol3]{kol3}S. Kolodziej, The complex Monge-Amp\`{e}re equation and
pluripotential theory. Mem. Amer. Math. Soc. 178 (2005), no. 840, x+64 pp.

\bibitem[L]{l}R. Lazarsfeld, Positivity in algebraic geometry. I. Classical
setting: line bundles and linear series. Ergebnisse der Mathematik und ihrer
Grenzgebiete. 3. Folge. A Series of Modern Surveys in Mathematics [Results in
Mathematics and Related Areas. 3rd Series. A Series of Modern Surveys in
Mathematics], 48. Springer-Verlag, Berlin, 2004. xviii+387 pp.

\bibitem[LZ]{lz}J. Liu and X. Zhang, The conical Kaehler-Ricci flow on Fano
manifolds, Advances in Mathematics, 307(2017), 1324--1371.

\bibitem[M]{m}Kenji Matsuki, \textit{An introduction to Mori program},
Springer-Verlag New York, lnc. 2002.

\bibitem[P1]{p1}G. Perelman, \textit{The entropy formula for the Ricci flow
and its geometric applications}, preprint, arXiv: math.DG/0211159.

\bibitem[P2]{p2}G. Perelman, \textit{Ricci flow with surgery on
three-manifolds}, preprint, arXiv: math.DG/0303109.

\bibitem[P3]{p3}G. Perelman, \textit{Finite extinction time for the solutions
to the Ricci flow on certain three-manifolds}, preprint, arXiv: math.DG/0307245.

\bibitem[PSS]{pss}D.-H. Phong, J. Sturm and N. Sesum, Multiplier ideal sheaves
and the Kaehler-Ricci flow, Comm. Anal. Geom. 15 (2007), no. 3, 613--632.

\bibitem[R]{r}M. Reid, \textit{Surface cyclic quotient singularities and
Hirzebruch-Jung resolutions.}

\bibitem[RT]{rt}J. Ross and R.P. Thomas, \textit{Weighted projective
embeddings, stability of orbifolds and constant scalar curvature K\"{a}hler
metrics,} JDG 88 (2011), no. 1, 109-160.

\bibitem[Ru]{ru}P. Rukimbira, \textit{Chern-Hamilton's conjecture and
K-contactness}, Houston J. Math. 21 (1995), no. 4, 709-718.

\bibitem[S]{s}S. Smale, \textit{On the structure of 5-manifolds,} Ann. of
Math. (2) 75 (1962), 38-46.

\bibitem[Shen]{shen}L. Shen, \textit{Maximal time existence of unnormalized
conical K\"{a}hler-Ricci flow}, J. reine anbew. Math. 760 (2020), 169-193.

\bibitem[Sp]{sp}James Sparks, Sasaki-Einstein Manifolds, Surveys Diff. Geom.
16 (2011) 265-324.

\bibitem[ST]{st}J. Song and G. Tian, \textit{The K\"{a}hler-Ricci flow through
singularities}. Invent. Math. 207 (2017), no. 2, 519--595.

\bibitem[SW1]{sw1}J. Song and B. Weinkove, \textit{Lecture notes on the
K\"{a}hler-Ricci flow}, introduction to the Kaehler-Ricci flow, Chapter 3 of
`Introduction to the Kaehler-Ricci flow', eds S. Boucksom, P. Eyssidieux, V.
Guedj, Lecture Notes Math. 2086, Springer 2013.

\bibitem[SW2]{sw2}J. Song and B. Weinkove, \textit{Contracting exceptional
divisors by the K\"{a}hler-Ricci flow}, Duke Math. J. 162 (2013), no. 2, 367--415.

\bibitem[SW3]{sw3}J. Song and B. Weinkove, \textit{Contracting exceptional
divisors by the K\"{a}hler-Ricci flow, II}, Proc. Lond. Math. Soc. 108 (2014),
no. 6, 1529--1561.

\bibitem[SWZ]{swz}K. Smoczyk, G. Wang and Y. Zhang, \textit{The Sasaki-Ricci
flow}, Internat. J. Math. 21 (2010), no. 7, 951--969.

\bibitem[T]{t}G. Tian, \textit{On Calabi's conjecture for complex surfaces
with positive first Chern class}, Invent. Math., 101, (1990), 101-172.

\bibitem[Ta]{ta}N. Tanaka, A Differential Geometric Study on Strongly
Pseudo-Convex Manifold, Kinokuniya, Tokyo, 1975.

\bibitem[To]{to}V. Tosatti, \textit{Kawa lecture notes on the K\"{a}hler-Ricci
flow}, Ann. Fac. Sci. Toulouse Math. 27 (2018), no.2, 285-376..

\bibitem[Tsu]{tsu}H. Tsuji, \textit{Existence and degeneration of
K\"{a}hler-Einstein metrics on minimal algebraic varieties of general type},
Math. Ann. 281 (1988), no. 1, 123--133.

\bibitem[W]{w}Yi-Sheng Wang, Resolution of singular fibers in an $S^{1}%
$-fibered $5$-manifold, preprint.

\bibitem[Y]{y}S.-T. Yau, \textit{On the Ricci curvature of a compact
K\"{a}hler manifold and the complex Monge-Ampere equation, I}, Comm. Pure.
Appl. Math. 31 (1978), 339-411.
\end{thebibliography}
\end{document}